\documentclass{article}
\usepackage{preamble}
\usepackage{appendix}

\theoremstyle{plain}
\newtheorem{theoremint}{Theorem}

\title{
\textbf{Foundations for operator algebraic tricategories}}
\author{Giovanni Ferrer}

\begin{document}

\maketitle

\begin{abstract}
An operator algebraic tricategory is a higher categorical analogue of an operator algebra. 
For algebraic tricategories, Gordon, Power, and Street proved that every algebraic tricategory is equivalent to a Gray-category, a result later refined by Gurski.
We adapt this result to the context of functional analysis, showing that every operator algebraic tricategory is equivalent to an operator Gray-category.
We then categorify the Gelfand-Naimark theorem for operator algebras, inductively proving that every (small) operator algebraic tricategory is equivalent to a concrete operator Gray-category.
We also provide several examples of interest for operator algebraic tricategories.
\end{abstract}

\addtocontents{toc}{}
\tableofcontents

\section{Introduction}
Operator algebras were first conceived in the context of quantum mechanics emerging as concretely C*-algebras and von Neumann algebras of operators acting on a Hilbert space \cite{Gelfand_Neumark_1994} \cite{Segal_1947} \cite{Murray_v.Neumann_1936}. 
We have since seen a movement throughout mathematics of categorifying structures and their results. 
Higher operator algebraic categories arise in the study of subfactor theory where, in particular, operator algebras with their bimodules and intertwiners form an operator 2-category. 
Picking a particular operator algebra $A$, i.e a C*/W*-algebra, functors from a unitary tensor category $\cC$ into $\End(A)$ in this operator 2-category are then quantum categorical/non-invertible symmetries.

Recently, even higher (braided) operator algebraic categories have emerged as
invariants of operator algebraic objects, like that of the categorified Connes' $\chi(M)$ \cite{CJP2021}. 
Moreover, operator algebraic 2-categories are expected to form an operator algebraic 3-category \cite{CP22}.
Higher operator algebraic categories have also seen important applications in topological order.
Starting in a 2+1D topological order, that is a unitary modular tensor category $\cC$, the topological domain walls for $\cC$ are classified by the fusion 2-category $\operatorname{Mod}(\cC)$ \cite{Kitaev_Kong_2012}.
Condensed matter physics is manifestly unitary, where time reversal symmetry is taking the Hermitian adjoint.
We thus desire a definition of unitarity for a fusion 2-category, which can also be viewed as a 3-category with 1 object.
Furthermore, physicists want to go even higher --- 3+1D topological order is expected to be described by a unitary modular fusion 2-category, for which there is not yet a formal definition in the literature \cite{Johnson-Freyd_2022}. 

Algebraic models for higher categories do, however, pose difficulties due to their large amount of constraint data. 
Coherence theorems allow us to consider semi-strictified models of algebraic higher categories while retaining full generality.
These semistrictified models are then easier to work with as they require less data. 
Notably, Gordon, Power, and Street proved that every algebraic tricategory is equivalent to a Gray-category, the strictest model of a general algebraic tricategory \cite{GPS95}. The definition of an algebraic category together with this theorem were later refined by Gurski \cite{gurski_2013}.

Operator algebraic tricategories are tricategorical analogues of operator algebras, for which one would desire an analoguous coherence theorem. 
However, constructions from ordinary category theory may fail the principle of equivalence \cite{nLab23} for $\dag$-categories, as noted in \cite{Penneys_2020} for dual functors on a unitary multifusion category. 
Moreover, classical results may run into complications when equipping algebraic objects with topologies. For example, the algebraic tensor product of C*-algebras may be equipped with many possible norms, including the maximal and minimal tensor product norms. Out of these only the maximal tensor product satisfies hom-tensor adjunction, while the minimal tensor product is more concrete and proves to be useful in practice.

In this article, we present a definition for operator algebraic tricategories and operator Gray-categories, as well as provide examples of interest. 
This requires us to construct an operator algebraic Gray tensor product, which uses substantial functional analysis, e.g., the notion of biduals for C*-algebras, and its categorification to C*-categories and C*-2-categories.
We then adapt the coherence theorem for algebraic tricategories to the context of functional analysis, obtaining our following main result.

\begin{theoremint}
    Every operator algebraic tricategory is equivalent to an operator Gray-category.
\end{theoremint}

Operator categories, and in particular operator algebras, are said to be concrete when they are made up of Hilbert spaces and operators between them.
We inductively extend this notion to higher operator categories, saying that an operator $n$-category is concrete when it is made up of concrete operator $(n-1)$-categories and higher operators between them.
As a corollary of our main result, we provide the following categorification of the Gelfand-Naimark theorem.

\begin{theoremint}
    Every (small) operator algebtraic tricategory is equivalent to a concrete operator Gray-category.
\end{theoremint}

This articles also provides a first stepping stone for future projects in 3-categorical unitary linear algebra. 
In particular, there is relevant interest in a theory of higher Hilbert spaces and manifestly unitary higher condensation, as well as a definition and classification of unitary dual 2-functors on a unitary multifusion 2-category.

The structure of this article is as follows. In Section \S\ref{sec:2cats}, we construct the operator algebraic versions of the Gray-tensor product. In Section \S\ref{sec:3cats}, we define operator 3-categories, operator Gray-categories, and state our main result. We provide background needed for Section \S\ref{sec:2cats} in Appendix \S\ref{1cats}, as well as defer some details from this section to Appendix \S\ref{sec:background2cats}. We then construct in detail the operator algebraic Morita 3-categories $\CHaus$ and $\Meas$ in Appendix \S\ref{sec:examples}. Finally, we provide a detailed account of the Yoneda embedding for so-called cubical operator algebraic 3-categories in Appendix \S\ref{sec:yoneda}.

\paragraph{Acknowledgements.}
The author would like to thank David Penneys and Corey Jones for many useful conversations.
This work was supported by NSF DMS grants 1654159 and 2154389. 

\section{2-categorical results} \label{sec:2cats}

We mimic this section's results one categorical dimension down in Appendix \S\ref{1cats}, and defer some details about operator 2-categories to Appendix \S\ref{sec:background2cats}.

\subsection{Operator 2-categories}

We will follow \cite{CHJP2022} in the following treatment of operator algebraic bicategories, except that we do not assume such 2-categories are unitarily Cauchy complete.

\begin{definition}
A C*-2-category consists of a (linear, weak) 2-category $(\cC,\ccirc,1,\alpha,\lambda,\rho)$ together with:
\begin{itemize}
\item[($\dag$)] a dagger structure, i.e. an involutive contravariant conjugate-linear 2-functor $\dag: \overline{\cC}^{\op} \to \cC$, which is the identity on objects and 1-morphisms, such that the associator, left unitor, and right unitor
\begin{align*}
\alpha_{X,Y,Z}&\colon X \ccirc_B (Y \ccirc_C Z) \tto (X \ccirc_B Y) \ccirc_C Z,\\
 \lambda_X &\colon \id_A \ccirc_A X \tto X,\\
 \rho_X&\colon X \ccirc_B \id_B \tto X
 \end{align*}
are unitary ($u^\dag = u^{-1}$) for composable 1-morphisms $A \xrightarrow{X} B \xrightarrow{Y} C \xrightarrow{Z} D$ in $\cC$ and the composition $\ccirc$ of 1-morphisms in $\cC$ is $\dag$-preserving,
\item[(C*)] 
the hom-category $\cC(A \to B)$ is a C*-category for every pair of objects $A,B \in \cC$.
\end{itemize}
Furthermore, we say that a C*-2-category is a W*-2-category when:
\begin{itemize}
    \item[(W*1)] every hom-category $\cC(A \to B)$ is a W*-category, 
    \item[(W*2)] 1-composition $\ccirc$ is separately normal (weak*-continuous).
\end{itemize}
In \cite{CHJP2022} it is shown that (W*2) is equivalent to the following condition:
\begin{itemize}
    \item[(W*2$'$)] \label{condition:W*2'} For every object $A \in \cC$, the pre-composition and post-composition maps $\id_A \ccirc_A -$ and $- \ccirc_A \id_A$ are normal.
\end{itemize}
We will use the terms operator algebraic bicategory or operator 2-category to refer to C*/W*-2-categories.
Moreover, we say that an operator 2-category is strict if the underlying 2-category is strict. Notice we use the terms 2-category and bicategory interchangeably, specifying whenever they are strict. 
\end{definition}

\begin{example}
Given an operator 2-category $\cA$, we may produce related operator 2-categories $\cA^{\twoop}$ and $\cA^{\op}$ given by reversing the direction of 1-morphisms and 2-morphisms respectively.
\end{example}

\begin{definition}
A $\dag$-2-functor $F \colon \cC \to \cD$ between C*-2-categories is a (linear) 2-functor $(F,F^2,F^0)$ such that
\begin{itemize}
    \item[($\dag$)] $F$ is dagger-preserving, i.e $F(x^\dag) = F(x)^\dag$ for every 2-morphism $x$ in $\cC$,
    \item[(coh)] the compositor and unitor 
    \begin{align*}
    F^2_{X,Y}&\colon F(X) \ccirc_{F(B)} F(Y) \tto F(X \ccirc_B Y),\\
    F^0_A&\colon \id_{F(A)} \tto F(\id_A),
    \end{align*}
    are unitary for $A \xrightarrow{X} B \xrightarrow{Y} C$ in $\cC$.
\end{itemize}
Furthermore, we say that a $\dag$-2-functor $F \colon \cC \to \cD$ between W*-2-categories is normal when
\begin{itemize}
\item[(W*)] it is normal (weak*-continuous) on hom-spaces $\cC(X \tto Y) \to \cD(F(X) \tto F(Y))$.
\end{itemize}
Moreover, we say that a (normal) $\dag$-2-functor is strict if the underlying 2-functor is strict.
\end{definition}

\begin{definition}
A $\dag$-2-natural transformation $\alpha\colon F \tto G$ between (normal) $\dag$-2-functors is a 2-natural transformation $(\alpha,\alpha^n)$ such that
\begin{itemize}
\item[(coh)] The naturator 
$$\alpha^n_{X}\colon F(X) \ccirc_{F(B)} \alpha_B \tto \alpha_A \ccirc_{G(A)} G(X)$$
is a unitary for every 1-morphism $A \xrightarrow{X} B$ in $\cC$.
\end{itemize}
\end{definition}

\begin{definition}
A uniformly bounded modification $m\colon \alpha \ttto \beta$ between $\dag$-2-natural transformations is a modification $m = (m_A \colon \alpha_A \tto \beta_A)_{A \in \cC}$ such that
$$\|m\| \coloneqq \sup_{A \in \cC} \|m_A\|_{\cD} < \infty.$$
\end{definition}

\begin{example}
We denote the strict C*-2-category of C*-categories, $\dag$-functors, and uniformly bounded natural transformations by $\CstarCats$.
Similarly, we denote the W*-2-category of W*-categories, normal $\dag$-functors and uniformly bounded natural transformation by $\WstarCats$.
\end{example}

\begin{example} 
In Proposition 2.13 of \cite{CP22}, it is show that for C*-2-categories $\cA$ and $\cB$, there is a C*-2-category $\CstarTwoCat(\cA \to \cB)$ of $\dag$-2-functors from $\cA$ to $\cB$, $\dag$-2-natural transformations, and uniformly bounded modifications. 
Moreover, $\CstarTwoCat(\cA \to \cB)$ is strict whenever $\cB$ is strict. 
There is an analogous W*-2-category $\WstarTwoCat(\cA \to \cB)$ of normal $\dag$-2-functors between W*-2-categories $\cA$ and $\cB$. 
\end{example}

\begin{definition}
A $\dag$-functor $F \in \CstarTwoCat(\cA \to \cB)$ is said to be a $\dag$-equivalence if it is an equivalence of the underlying 2-categories. This is equivalent to $F$ being unitarily essentially surjective on objects and 1-morphisms, and fully faithful on 2-morphisms.
\end{definition}

\subsection{Coherence and concreteness for operator algebraic bicategories} 

For algebraic 2-categories, it is well-known that every 2-category is equivalent to a strict category. This coherence theorem can be easily adapted to operator 2-categories as follows.

\begin{theorem}[Coherence for operator 2-categories] \label{thm:coherence2cats}
Every operator 2-category is (unitarily) equivalent to a strict operator 2-category. 
In particular, every C*-2-category $\cA$ admits a monic\footnote{By monic, we mean injective on objects, 1-morphisms, and 2-morphisms. 
One can easily show that the image of $\cA$ under such a $\dag$-2-functor into a strict C*-2-category forms a (strict) C*-2-subcategory which is equivalent to $\cA$.} $\dag$-2-functor into a strict C*-2-category afforded by the Yoneda embedding:
$$\yo \colon \cA \to \CstarTwoCat(\cA^{\twoop} \to \CstarCats).$$
In fact, there is a monic $\dag$-2-functor $\yo^\amalg \colon \cA \to \CstarCats$ given on objects by
$$
\yo^\amalg(B) = \coprod_{A \in \cA} \yo^A_B = \coprod_{A \in \cA} \cA(A \to B).\footnote{Here we use the disjoint union notation $\cA = \coprod_{i \in I} \cA_i$ to denote the C*-category with objects $\mathsf{Ob} \cA$ being the disjoint union $\coprod \mathsf{Ob} \cA_i$, which admits fully-faithful inclusions of each $\cA_i$ in the obvious way, and with trivial hom-spaces between objects from different components.}
$$
When $\cA$ is a W*-2-category, the Yoneda embedding lands in the strict W*-2-category $\WstarTwoCat(\cA^{\op} \to \WstarCats)$.
\end{theorem}

For operator 1-categories, it is also well-known that every small operator 1-category can be concretely realized as a category of Hilbert spaces and operators between them \cite{GLR1985}. One can categorify this Gelfand-Naimark-Segal construction to the level of operator 2-categories, obtaining the following result. 

\begin{theorem}[Concreteness for operator 2-categories]
For a small operator 2-category $\cA$, there exists a universal representation   
$$\Upsilon \colon \cA \to \TwoHilb,$$
where $\TwoHilb$ is the strict W*-2-category of 
all W*-subcategories of $\Hilb$, the W*-category of Hilbert spaces. 
Moreover, the image of $\Upsilon$ forms a strict operator 2-category $\GNS(\cA)$ which is equivalent to $\cA$.
\end{theorem}
We provide the details of these results in Appendix \S \ref{subsec:coherencesforonecats}.
\subsection{Local tensor product for operator algebraic bicategories}
In this section, we use the tensor product of operator categories to obtain a local tensor product of operator 2-categories.
\begin{definition}
Given C*-2-categories $\cA_1$ and $\cA_2$, we define the
C*-local tensor product
$\cA_1 \blacktimes \cA_2$ by:
\begin{itemize}
\item[(0)] Objects are tuples of objects $(A_1,A_2) \in \cA_1 \times \cA_2$; 
\item[(hom)] Hom-spaces are given by:
\begin{align*}
\cA_1 \blacktimes \cA_2 \big( (A_1,A_2) \to (A_1,B_2)\big) \coloneqq \cA_1(A_1 \to B_1) \maxtimes \cA_2(A_2 \to B_2),
\end{align*}
where we are using the maximal tensor product of C*-categories (see Appendix \ref{subsec:tensorproductsCstar1cats}).

\item[($\ccirc$)] Composition $\ccirc$ of 1-morphisms is given pointwise and 1-composition $\ccirc$ of 2-morphisms is determined using the universal property of $\maxtimes$. In particular
$$(x \otimes y) \ccirc (x' \otimes y') = (x \ccirc x') \otimes (y \ccirc y'),$$
whenever it makes sense.
\item[($\circ$)] Composition $\circ$ of 2-morphism is defined similarly.
\item[(coh)] Constraint data is given by simple tensors of constraint data in $\cA_1$ and $\cA_2$. 
These are indeed natural by the bilinearity and bicontinuity of $\ccirc$ and $\circ$.
\end{itemize}
Similarly, when $\cA_1$ and $\cA_2$ are W*-2-categories, we define the W*-2-category $\cA_1 \Wblacktimes \cA_2$ by replacing $\maxtimes$ with W*-tensor product $\Wmaxtimes$ (see Appendix \ref{subsec:Wstartensorproducts1cats}).
\end{definition}

\begin{remark}
When $\cA_1$ and $\cA_2$ are strict operator 2-categories, their local tensor product is also strict.
\end{remark}

\subsection{C*-Gray tensor product}

In this section, we adapt \cite[\S 3]{gurski_2013} to the operator algebraic setting.

\begin{definition}
Let $\cA_1$, $\cA_2$, and $\cB$ be strict C*-2-categories. A $\dag$-2-functor $F\colon \cA_1 \maxblacktimes \cA_2 \to \cB$ is \emph{cubical} if it is strictly identity-preserving and the following condition holds:
\begin{itemize}\label{eq:cubical}
\item[(\mancube)] If $(X_1 ,X_2)$ and $(Y_1 ,Y_2)$ are composable 1-morphisms in $\cA_1 \maxblacktimes \cA_2$ such that $Y_2$ or $X_1$ is an identity, then the compositor 2-morphism
\begin{equation*}
F(X_1 , X_2) \ccirc F(Y_1 , Y_2) \Rightarrow F\big((X_1 , X_2) \ccirc (Y_1 , Y_2)\big).
\end{equation*}
is an identity. 
\end{itemize}
When $\cA_1$, $\cA_2$, $\cB$ are W*, we say that a $\dag$-2-functor $F\colon \cA_1 \Wblacktimes \cA_2 \to \cB$ is a \emph{separately normal} cubical functor if it is both normal and cubical. In particular, a cubical functor $F\colon \cA_1 \maxblacktimes \cA_2 \to \cB$ extends to a separately normal cubical functor only when $\alpha^\lambda_i \to \alpha_i$ weak* in $\cA_i$ for $i = 1,2$ implies
    $$F(\alpha^\lambda_1 \otimes \alpha_2) \to F(\alpha_1 \otimes \alpha_2) \text{ and } F(\alpha_1 \otimes \alpha^\lambda_2) \to F(\alpha_1 \otimes \alpha_2) \text{ weak* in } \cB.$$ 
\end{definition}


\begin{proposition}\label{prop:dataofcubicals}
The data of a (separately normal) cubical $\dag$-2-functor $F\colon \cA_1 \maxblacktimes \cA_2 \to \cB$ is determined uniquely by:
\begin{enumerate}
\item[$(1)$] For each object $A_1 \in \cA_1$, a strict (normal) $\dag$-2-functor $F_{A_1}\colon \cA_2 \to \cB$;
\item[$(2)$] For each object $A_2 \in \cA_2$, a strict (normal) $\dag$-2-functor $F_{A_2}\colon \cA_1 \to \cB$;
\end{enumerate}
such that
$$F_{A_1}(A_2) = F_{A_2}(A_1) \coloneqq F(A_1,A_2);$$
\begin{enumerate}
\item[$(\Sigma)$] For each pair of 1-morphisms $X_i\colon A_i \to A_i'$ in $\cA_i$ for $i = 1,2,$ an ``interchanger'' unitary 2-morphism 
\[
\begin{tikzcd}[row sep=40pt]
F(A_1,A_2) \arrow[d,"F_{A_2}(X_1)"'] \arrow[r,"F_{A_1}(X_2)"] & \arrow[dl,"\Sigma_{X_1,X_2}" description, Rightarrow] F(A_1,A_2') \arrow[d,"F_{A_2'}(X_1)"] \\
F(A_1',A_2) \arrow[r,"F_{A_1'}(X_2)"'] & F(A_1',A_2')
\end{tikzcd}
\]
which is an identity 2-morphism whenever $X_1$ or $X_2$ is an identity 1-morphism;
\end{enumerate}
satisfying the following three axioms:
\begin{itemize}
\item[$(\Sigma1)$] Naturality:
For each pair of 2-morphisms $\alpha_1,\alpha_2$ in $\cA_1$ and $\cA_2$ respectively
\[
\begin{tikzcd}[row sep=40pt]
{F(A_1,A_2)} & {F(A_1,A_2')} \\
	{F(A_1',A_2)} & {F(A_1',A_2')}
	\arrow[""{name=0, anchor=center, inner sep=0}, from=1-1, to=2-1]
	\arrow[from=1-1, to=1-2]
	\arrow[from=1-2, to=2-2]
	\arrow[""{name=1, anchor=center, inner sep=0}, from=2-1, to=2-2]
	\arrow["\Sigma"{description}, Rightarrow, from=1-2, to=2-1]
	\arrow[""{name=2, anchor=center, inner sep=0}, shift right=9, curve={height=24pt}, from=1-1, to=2-1]
	\arrow[""{name=3, anchor=center, inner sep=0}, shift right=3, curve={height=18pt}, from=2-1, to=2-2]
	\arrow["F_{A_2}\alpha_1"{description}, shorten <=6pt, shorten >=6pt, Rightarrow, from=0, to=2]
	\arrow["F_{A_1'}\alpha_2"{description}, shorten <=3pt, shorten >=3pt, Rightarrow, from=1, to=3]
\end{tikzcd}
=
\begin{tikzcd}[row sep=40pt]
	{F(A_1,A_2)} & {F(A_1,A_2')} \\
	{F(A_1',A_2)} & {F(A_1',A_2')}
	\arrow[from=1-1, to=2-1]
	\arrow[""{name=0, anchor=center, inner sep=0}, from=1-1, to=1-2]
	\arrow[""{name=1, anchor=center, inner sep=0}, from=1-2, to=2-2]
	\arrow[from=2-1, to=2-2]
	\arrow["\Sigma"{description}, Rightarrow, from=1-2, to=2-1]
	\arrow[""{name=2, anchor=center, inner sep=0}, shift left=3, curve={height=-18pt}, from=1-1, to=1-2]
	\arrow[""{name=3, anchor=center, inner sep=0}, shift left=11, curve={height=-18pt}, from=1-2, to=2-2]
	\arrow["\scriptsize{F_{A_1}\alpha_2}"{description}, shorten <=3pt, shorten >=3pt, Rightarrow, from=2, to=0]
	\arrow["\scriptsize{F_{A_2'}\alpha_1}"{description}', shorten <=5pt, shorten >=5pt, Rightarrow, from=3, to=1]
\end{tikzcd}
\]
\item[$(\Sigma2)$] Composability in the first entry:
\[\begin{tikzcd}[row sep=40pt]
	{F(A_1,A_2)} & {F(A_1,A_2')} \\
	{F(A_1',A_2)} & {F(A_1',A_2')} \\
	{F(A_1'',A_2')} & {F(A_1'',A_2')}
	\arrow[from=1-1, to=2-1]
	\arrow[from=1-1, to=1-2]
	\arrow[from=1-2, to=2-2]
	\arrow[from=2-1, to=2-2]
	\arrow["\Sigma"{description}, Rightarrow, from=1-2, to=2-1]
	\arrow[from=2-1, to=3-1]
	\arrow[from=2-2, to=3-2]
	\arrow[from=3-1, to=3-2]
	\arrow["\Sigma"{description}, Rightarrow, from=2-2, to=3-1]
\end{tikzcd}
=
\begin{tikzcd}[row sep=40pt]
	{F(A_1,A_2)} & {F(A_1,A_2')} \\
	{F(A_1',A_2)} & {F(A_1',A_2')} \\
	{F(A_1'',A_2')} & {F(A_1'',A_2')}
	\arrow[from=1-1, to=2-1]
	\arrow[from=1-1, to=1-2]
	\arrow[from=1-2, to=2-2]
	\arrow[from=2-1, to=3-1]
	\arrow[from=2-2, to=3-2]
	\arrow[from=3-1, to=3-2]
	\arrow["\Sigma"{description}, Rightarrow, from=1-2, to=3-1]
\end{tikzcd}
\]
\item[$(\Sigma3)$] Composability in the second entry: 
\[
\begin{tikzcd}[row sep=40pt]
	{F(A_1,A_2)} & {F(A_1,A_2')} & {F(A_1,A_2'')} \\
	{F(A_1',A_2)} & {F(A_1',A_2')} & {F(A_1',A_2'')}
	\arrow[from=1-1, to=2-1]
	\arrow[from=1-1, to=1-2]
	\arrow[from=1-2, to=2-2]
	\arrow[from=2-1, to=2-2]
	\arrow[from=1-2, to=1-3]
	\arrow[from=2-2, to=2-3]
	\arrow[from=1-3, to=2-3]
	\arrow["\Sigma"{description}, Rightarrow, from=1-2, to=2-1]
	\arrow["\Sigma"{description}, Rightarrow, from=1-3, to=2-2]
\end{tikzcd}
=
\begin{tikzcd}[row sep=40pt]
	{F(A_1,A_2)} & {F(A_1,A_2')} & {F(A_1,A_2'')} \\
	{F(A_1',A_2)} & {F(A_1',A_2')} & {F(A_1',A_2'')}
	\arrow[from=1-1, to=2-1]
	\arrow[from=1-1, to=1-2]
	\arrow[from=2-1, to=2-2]
	\arrow[from=1-2, to=1-3]
	\arrow[from=2-2, to=2-3]
	\arrow[from=1-3, to=2-3]
	\arrow["\Sigma"{description}, Rightarrow, from=1-3, to=2-1]
\end{tikzcd}
\]
\end{itemize}
\end{proposition}

\begin{proof}
First suppose we have a (separately normal) cubical $\dag$-2-functor $F$. For an object $A_1 \in \cA_1$ we define 
\begin{align*}
F_{A_1}(A_2) &\coloneqq F(A_1,A_2),\\
F_{A_1}(X) &\coloneqq F(\id_{A_1},X),\\
F_{A_1}(\alpha) &\coloneqq F(\id_{\id_{A_1}} \otimes \alpha).
\end{align*}
Then $F_{A_1}\colon \cA_2 \to \cB$ strictly preserves identity 1-morphisms since $F$ does. Moreover, 
$$F_{A_1}( Y_1 \ccirc Y_2) = F(\id_{A_1}, Y_1 \ccirc Y_2) \underset{\left(\mbox{\mancube}\right)}{=} F(\id_{A_1},Y_1) \ccirc F(\id_{A_1},Y_2) = F_{A_1}(Y_1) \ccirc F_{A_1}(Y_2),$$
so $F_{A_1}$ is a (normal) strict $\dag$-2-functor. We define $F_{A_2}\colon \cA_1 \to \cB$ similarly for an object $A_2 \in \cA_2$. We define the unitary $\Sigma$ as follows:
\[
\begin{tikzcd}[row sep=60pt]
F(A_1,A_2) \arrow[d,"F_{A_2}(X_1)"'] \arrow[r,"F_{A_1}(X_2)"] & \arrow[dl,"\Sigma_{X_1,X_2}" description, Rightarrow] F(A_1,A_2') \arrow[d,"F_{A_2'}(X_1)"] \\
F(A_1',A_2) \arrow[r,"F_{A_1'}(X_2)"'] & F(A_1',A_2')
\end{tikzcd}
=
\begin{tikzcd}[row sep=60pt]
	{F(A_1,A_2)} & {F(A_1,A_2')} \\
	{F(A_1',A_2)} & {F(A_1',A_2')}
	\arrow[""{name=0, anchor=center, inner sep=0}, "{F(X_1,X_2)}"{description}, from=1-1, to=2-2]
	\arrow["{F(\id,X_2)}", from=1-1, to=1-2]
	\arrow["{F(X_1,\id)}", from=1-2, to=2-2]
	\arrow["{F(X_1,\id)}"', from=1-1, to=2-1]
	\arrow["{F(\id,X_2)}"', from=2-1, to=2-2]
	\arrow["{F^2}"{description}, shorten >=6pt, Rightarrow, from=1-2, to=0]
	\arrow["F^2"{description}, shorten >=6pt, Rightarrow, from=2-1, to=0]
\end{tikzcd}
\underset{\left(\mbox{\mancube}\right)}{=}
\begin{tikzcd}[row sep=60pt]
	{F(A_1,A_2)} & {F(A_1,A_2')} \\
	{} & {F(A_1',A_2')}
	\arrow[""{name=0, anchor=center, inner sep=0}, "{F(X_1,X_2)}"', from=1-1, to=2-2]
	\arrow["{F(\id,X_2)}", from=1-1, to=1-2]
	\arrow["{F(X_1,\id)}", from=1-2, to=2-2]
	\arrow["{F^2}"{description}, shorten >=6pt, Rightarrow, from=1-2, to=0]
\end{tikzcd}
\]
which is an identity 2-morphism when $X_1$ or $X_2$ is an identity by the cubicality (\mancube) of $F$ again. Axiom ($\Sigma1$) follows by the naturality of the compositor $F^2$, and axioms $(\Sigma2)$ and $(\Sigma3)$ follow by the associativity axiom $F^2$ satisfies.

Conversely, suppose we are given such collections of (normal) $\dag$-2-functors and unitaries $\Sigma$. We construct $F\colon \cA_1 \maxblacktimes \cA_2 \to \cB$ as follows.
\begin{itemize}
\item[(0)] For objects, $F(A_1,A_2) \coloneqq F_{A_1}(A_2) = F_{A_2}(A_1)$;
\item[({\scriptsize \textsf{hom}})] We first define $F_{\mathsf{hom}}: \cA_1(A_1 \to A_1') \times \cA_2(A_2 \to A_2') \to \cB\big((F(A_1,A_2) \to F(A_1',A'_2)\big)$ by 
\begin{align*}
F(X_1,X_2) &\coloneqq F_{A_2}(X_1) \ccirc F_{A'_1}(X_2)\\
F(\alpha_1 \otimes \alpha_2) 
&\coloneqq  F_{A_2}(\alpha_1) \ccirc F_{A'_1}(\alpha_2).
\end{align*}
Since $F_{\mathsf{hom}}$ is (separately normal) $\dag$-bilinear and $\cB\big((F(A_1,A_2) \to F(A_1',A'_2)\big)$ is C*/W*, we may uniquely extend $F_{\mathsf{hom}}$ to $\cA_1(A_1 \to A_1') \maxtimes \cA_2(A_2 \to A_2') = \cA_1 \maxblacktimes \cA_2 \big((A_1,A_2) \to (A_1',A_2')\big)$ by the universal property of the maximal C*-tensor product $\maxtimes$ (see Appendix \ref{subsec:tensorproductsCstar1cats}).
\end{itemize}
We set the unitor $F^0$ to be identity and define the compositor component 
$$F^2\colon F(X_1,X_2) \ccirc F(X'_1,X'_2) \Rightarrow F(X_1 \ccirc X'_1, X_2 \ccirc X'_2)$$
for $A_i \xRightarrow{X_i} A'_i \xRightarrow{X'_i} A''_i$ in $\cA_i$ to be
$$
F_{A_2}(X_1) \ccirc F_{A_1'}(X_2) \ccirc F_{A'_2}(X'_1) \ccirc F_{A''_2}(X'_2) \xRightarrow{\id \ccirc \Sigma_{X_2,X'_1} \ccirc \id}
F_{A_2}(X_1) \ccirc F_{A_2}(X'_1) \ccirc F_{A''_1}(X_2) \ccirc F_{A''_2}(X'_2).$$

The naturality of $F^2$ is guaranteed by $(\Sigma1)$ and the associativity axiom is given by $(\Sigma2)$ and $(\Sigma3)$. Finally, $F$ is cubical since $\Sigma_{X_1,X'_2}$ is trivial whenever $X_1$ or $X'_2$ is the identity.
\end{proof}

We now explicitly construct the Gray tensor product for strict C*-$2$-categories.

\begin{definition}
The algebraic Gray tensor product $\cA \boxtimes \cB$ of strict C*-2-categories $\cA$ and $\cB$ is composed of tuples of objects $(A,B) \in \cA \times \cB$. The 1-morphisms of $\cA \boxtimes \cB$ are produced by two kinds of generators:
\begin{itemize}
    \item $(X,\id_B): (A,B) \to (A',B)$ for $X\colon A \to A'$ in $\cA$,
    \item $(\id_A,Y): (A,B') \to (A,B')$ for $Y: B \to B'$ in $\cB$.
\end{itemize}
The 1-morphisms in $\cA \boxtimes \cB$ are equivalence classes of composable strings of generators. The equivalence relation is the smallest such that:
\begin{itemize}
    \item $(X,\id_B) (X',\id_B) \thicksim (X \ccirc X',\id_B)$ for $A \xrightarrow{X} A' \xrightarrow{X'} A''$ in $X$ and $B \in \cB$,
    \item $(\id_A,Y) (\id_A,Y') \thicksim (\id_A, Y \ccirc Y')$ for $A \in X$ and $B \xrightarrow{Y} B' \xrightarrow{Y'} B''$ in $Y$,
    \item If $W \thicksim W'$, then $W V \thicksim W' V$ and $V W \thicksim V W'$ whenever they make sense.
\end{itemize}
We define the composition $\ccirc$ of 1-morphisms to be string concatenation. Notice $W \thicksim W'$ only when $W,W'$ have the same source and target, and $(\id_A,\id_B) = \id_{(A,B)}$.\medskip

\noindent The 2-morphisms of $\cA \boxtimes \cB$ are generated by three kinds of 2-morphisms:
\begin{itemize}
    \item $x \otimes \id_{\id_B}: (X,\id_B) \Rightarrow (X',\id_B)$ for $x: X \Rightarrow X'$ in $\cA$ and $B \in \cB$,
    \item $\id_{\id_A} \otimes y: (\id_A,Y) \Rightarrow (\id_A,Y')$ for $A \in \cA$ and $y: Y \Rightarrow Y'$ in $\cB$,
    \item $\Sigma_{X,Y}: (X,\id_B) \ccirc (\id_A,Y) \xRightarrow{\thicksim} (\id_A,Y) \ccirc (X,\id_B)$ for 1-morphisms $X\colon A \to A'$ in $\cA$ and $Y \colon B \to B'$ in $Y$, such that $\Sigma_{\id_A,Y} = \id_{(\id_A,Y)}$ and $\Sigma_{X,\id_B} = \id_{(X,\id_B)}$.
\end{itemize}
We will omit the subscripts on $\id$'s for simplicity. The 2-morphisms in $\cA \boxtimes \cB$ include equivalence classes of strings of formal $\ccirc$-composites of generators. The equivalence relation is first defined horizontally (in terms of $\ccirc$-composites), then vertically (as strings). We first define $\thicksim$ as the smallest equivalence relation such that
\begin{itemize}
    \item $(x \otimes \id) \ccirc (x' \otimes \id) \thicksim (x \ccirc x') \otimes \id$ and $(\id \otimes y) \ccirc (\id \otimes y') \thicksim \id \otimes (y \ccirc y')$,
    \item If $s \thicksim s'$ then $s \ccirc t \thicksim s' \ccirc t$ and $t \ccirc s \thicksim t \ccirc s'$ whenever they make sense.
\end{itemize}
Notice $s \thicksim s'$ only when $s$ and $s'$ have the same source and target morphisms. In what follows, we will denote the equivalence class of $s$ by $[s]$. A 2-morphism in $\cA \boxtimes \cB$ is then an equivalence class of formal linear combinations of vertically composable strings $\sum_k \lambda_k [w_{1,k}] \cdots [w_{n,k}]$, where the equivalence relation is the smallest such that:
\begin{itemize}
\item
$\lambda [x \otimes \id] + [\widetilde{x} \otimes \id] \thicksim [(\lambda x + \widetilde{x}) \otimes \id]$ and $\lambda[\id \otimes y] + [\id \otimes \widetilde{y}] \thicksim [\id \otimes (\lambda y + \widetilde{y})]$ for $\lambda \in \bbC$,
\item 
$[x \otimes \id][x' \otimes \id] \thicksim [(x \circ x' ) \otimes \id]$,
and
$[\id \otimes y][\id \otimes y'] \thicksim [\id \otimes (y \circ y')]$
\item[($\Sigma$1)]
$[(\id \otimes y) \ccirc (x \otimes \id)] [\Sigma_{X,Y}] \thicksim [\Sigma_{X',Y'}] [(x \otimes \id) \ccirc (\id \otimes y)]$
\item[($\Sigma$2)]
$ [(\id_{X'} \otimes \id) \ccirc \Sigma_{X,Y}]  [\Sigma_{X',Y} \ccirc (\id_X \otimes \id)] \thicksim [\Sigma_{X \ccirc X', Y}]$
\item[($\Sigma$3)]
$[\Sigma_{X,Y'} \ccirc (\id \otimes \id_Y)][(\id \otimes \id_{Y'}) \ccirc \Sigma_{X,Y}] \thicksim [\Sigma_{X,Y \ccirc Y'}]$
\item For general 2-morphisms $w,v$,
$(w + \widetilde{w}) v \thicksim w v + \widetilde{w} v$, and $v (w + \widetilde{w})\thicksim v w  + v \widetilde{w}$ whenever they make sense, and similar relations for $\ccirc$; and
\item if $w \thicksim w',$ then 
$wv \thicksim w' v$, $vw \thicksim vw'$, and $\lambda w + v \thicksim \lambda w' + v$ for $\lambda \in \bbC$ whenever they make sense. 
\end{itemize}
Vertical composition $\circ$ of 2-morphisms is given by the bilinear extension of concatenation of strings.
For horizontal composition of strings $w,v$, we can always express $w$ and $w'$ by the sum of strings $[w_{1,k}]\cdots[w_{n,k}]$ and $[v_{1,k}]\cdots[v_{n,k}]$ of the same length (by adding identities) and then define 
$$w \ccirc v= \sum_{k,\ell} [w_{1,k} \ccirc v_{1,\ell}] \cdots [w_{n,k} \ccirc v_{n,\ell}].$$ 
We define $\dag$ on generator 2-morphisms as follows:
\begin{itemize}
    \item $(x \otimes \id)^\dag = x^\dag \otimes \id$,
    \item $(\id \otimes y)^\dag = \id \otimes y^\dag$,
    \item $\Sigma_{X,Y}^\dag = \Sigma_{X,Y}^{-1}$.
\end{itemize}
We then extend $\dag$ to sums, tensors, and composites by:
\begin{itemize}
    \item $(\sum_k \lambda_k [w_k])^\dag = \sum \overline{\lambda}_k [w_k]^\dag$
    \item $(w \circ v)^\dag = v^\dag \circ w^\dag$
    \item $(w \ccirc v)^\dag = w^\dag \ccirc v^\dag$.
\end{itemize}
We then define the C*-Gray tensor product $\cA_1 \maxboxtimes \cA_2$ to be the completion of $\cA_1 \boxtimes \cA_2$ on each hom-space with respect to the following maximal C*-norm
$$\|w\|_\mu \coloneqq \sup \sset{\|F(w)\|}{F\colon \cA \boxtimes \cB \to \CstarCats \text{ is a } \dag\text{-2-functor}},$$
after quotienting out by 2-morphisms $w$ such that $\|w\|_\mu = 0$.
\end{definition}

\begin{remark}
There indeed exists a $\dag$-2-functor 
$$F \colon \cA \boxtimes \cB \to \CstarCats.$$
For example, we may use the universal representations $\yo^\amalg_\cA \colon \cA \to \CstarCats$ and $\yo^\amalg_\cB \colon \cB \to \CstarCats$ in Theorem \ref{thm:coherence2cats} to determine a strict $\dag$-2-functor
\begin{enumerate}
    \item[(0)] on objects by
    $$(A,B) \mapsto \yo^\amalg_\cA(A) \maxtimes \yo^\amalg_\cB(B);$$
    \item[(1)] on 1-morphisms by
    \begin{align*}
(X,\id) &\mapsto \yo^\amalg_\cA(X) \maxtimes \id,\\ 
(\id,Y) &\mapsto \id \maxtimes \yo^\amalg_\cB(Y);
    \end{align*}
    \item[(2)] on 2-morphisms by
    \begin{align*}
x \otimes \id &\mapsto \yo^\amalg_\cA(x) \maxtimes \id,\\ 
\id \otimes y &\mapsto \id \maxtimes \yo^\amalg_B(y), \\
\Sigma_{X,Y} &\mapsto \id_{\yo^\amalg_\cA(X) \maxtimes \yo^\amalg_\cB(Y)}.
    \end{align*}
\end{enumerate}
\end{remark}

\begin{remark}
We expect there to exist a ``minimal'' Gray-tensor product $\cA \underset{\min}{\boxtimes} \cB$ of strict C*-2-categories $\cA$ and $\cB$, given by completing $\cA \boxtimes \cB$ via a monic $\dag$-2-functor
$$\cA \boxtimes \cB \to \CstarCats.$$
The existence of such a monic representation would imply that there do not exist negligible 2-morphisms in $\cA \boxtimes \cB$. However, this will not be necessary for our main coherence and concreteness results.   
\end{remark}

\begin{remark}\label{rem:functorsintogray}
For each object $B \in \cB$, there is an organic strict $\dag$-2-functor
$$- \boxtimes B: \cA \to \cA \boxtimes \cB$$
given by
\begin{align*}
    A &\mapsto (A,B),\\
    X &\mapsto (X,\id_B),\\
    x &\mapsto x \otimes \id_{\id_B},
\end{align*}
on objects, 1-morphisms, and 2-morphisms in $\cA$ respectively.
Similarly, each object $A \in \cA$ induces a strict $\dag$-2-functor
$$A \boxtimes - \colon \cB \to \cA \boxtimes \cB.$$  
\end{remark}

\begin{proposition}
For strict C*-2-categories $\cA$ and $\cB$, $\|\cdot\|_\mu$ is a norm on each hom-space of $\cA \boxtimes \cB$. Furthermore, $\cA \maxboxtimes \cB$ is a strict C*-2-category.
\end{proposition}

\begin{proof}
A simple verification reveals that $\| \cdot\|_\mu$ is a C*-norm, possibly taking value $\infty$. For example,
$$
\|w\|_\mu = \sup_{F} \|F(w)\| = \sup_F \|F(w) F(w)^\dag\|^{1/2} = \sup_F \|F(ww^\dag)\|^{1/2} 
= \|ww^\dag\|_\mu^{1/2}
$$
for every 2-morphism $w$ in $\cA \boxtimes \cB$.
Furthermore, $\| \cdot \|_\mu$ is sub-cross by the naturality and unitality of compositors. Indeed, if $\sigma\colon \phantom{}_AX_B \to \phantom{}_AY_B$ and $\sigma'\colon \phantom{}_BX'_C \to \phantom{}_BY'_C$, then
$$\| \sigma \ccirc \sigma'\|_\mu 
=
\sup_F \| F(\sigma \ccirc \sigma')\|
=
\sup_F \| F^2 (F\sigma \ccirc F\sigma') F^{-2} \|
\leq 
\sup_F \| F\sigma \ccirc F\sigma' \|
\leq
\sup_F \| F\sigma\| \|F\sigma' \|
\leq 
\|\sigma\|_\mu \|\sigma'\|_\mu.
$$
Thus, to prove $\| \cdot \|_\mu < \infty$ on each hom-set, it suffices to verify that this is the case for generator 2-morphisms $\alpha \otimes \id$, $\id \otimes \beta$, and $\Sigma_{f,g}$.
First observe that, for every object $B \in \cB$ and 2-morphism $\alpha$ in $\cA$, we have
$$\|F(\alpha \otimes \id_{\id_B})\| = \|F \circ (- \boxtimes B) (\alpha)\| \leq \|\alpha\| \quad\text{for every } F: \cA \boxtimes \cB \to \CstarCats.$$
Therefore $\|\alpha \otimes \id\|_\mu \leq \|\alpha\|$.
Similarly, one shows $\|\id \otimes \beta \|_\mu \leq \|\beta\|$. 
Finally, each $\|\Sigma_{f,g}\|_\mu = 1$ since $\Sigma_{f,g}$ is unitary and $\dag$-2-functors preserve unitaries.
Therefore $\|\cdot\|_\mu$ is a norm on each hom-space. 
It remains to show that each hom-category 
$$\fH \coloneqq \cA \maxboxtimes \cB\big((A_1,A_2) \to (B_1,B_2)\big)$$
satisfies the positivity condition required of C*-categories (see Appendix \ref{subsec:operator1categories}).
We already know that each endomorphism algebra in $\fH$ is a C*-algebra.
So for any morphism $\tau$ in $\fH$ (which is a 2-morphism in $\cA \maxboxtimes \cB$), $\tau^* \circ \tau$ is contained in such an endomorphism C*-algebra $\cE$. 
So there exists a positive $\sigma \in \cE$ such that $\sigma^4 = \tau^* \circ \tau \circ \tau^* \circ \tau$. 
We claim that $\sigma^2 = \tau^* \circ \tau$. 
Let $F\colon \cA_1 \boxtimes \cA_2 \to \CstarCats$ be a $\dag$-2-functor, which we may extend to a representation of $\cA_1 \maxboxtimes \cA_2$ by continuity.
The fact that $\sigma \geq 0$ implies that $F(\sigma^2) \geq 0$.
Now notice $F(\sigma^2)^2 = F(\tau^* \tau)^2$. 
Since $\CstarCats$ is a C*-2-category, $F(\tau^* \tau) = F(\tau)^* F(\tau) \geq 0$.
Since $F\cE$ is a C*-algebra, the uniqueness of positive square roots yields $F(\sigma^2) = F(\tau^*\tau)$.
Since $F$ was arbitrary, we conclude $\sigma^2 = \tau^* \tau$. 
Therefore $\fH$ is a C*-category and we conclude that $\cA_1 \maxboxtimes \cA_2$ is a C*-2-category.
\end{proof}

\begin{universalprop}
For strict C*-2-categories $\cA_1$ and $\cA_2$, there is a natural cubical $\dag$-2-functor
$$
\cA_1 \maxblacktimes \cA_2
\to
\cA_1 \maxboxtimes \cA_2.$$
For any cubical functor $F: \cA_1 \maxblacktimes \cA_2 \to \cB$ into a strict C*-2-category $\cB$, there exists a unique strict $\dag$-2-functor
\[\begin{tikzcd}
\cA_1 \maxblacktimes \cA_2 \arrow[d] \arrow[r,"F"] \arrow[d] & \cB\\
\cA_1 \maxboxtimes \cA_2 \arrow[ur,"F"',dashed] &
\end{tikzcd}\]
\end{universalprop}
\begin{verification}
We provide the data of the cubical functor $C: \cA_1 \maxblacktimes \cA_2
\to
\cA_1 \maxboxtimes \cA_2$ using Proposition \ref{prop:dataofcubicals}. For objects $A_1 \in \cA_1$ and $A_2 \in \cA_2$, we define the strict $\dag$-2-functors 
\begin{align*}
C_{A_2}&\colon \cA_1 \to \cA_1 \maxboxtimes \cA_2\\
C_{A_1}&\colon \cA_2 \to \cA_1 \maxboxtimes \cA_2
\end{align*}
to be $- \boxtimes A_2$ and $A_1 \boxtimes -$ as in Remark \ref{rem:functorsintogray}.
For a pair 1-morphisms $X_1$ and $X_2$ in $\cA_1$ and $\cA_2$ respectively, the unitary 2-morphism $\Sigma_{X_1,X_2}$ in $\cA_1 \maxboxtimes \cA_2$ serves the role of $\Sigma_{X_1,X_2}$ in Proposition \ref{prop:dataofcubicals}. This data satisfies all of the desired axioms by construction of the algebraic Gray tensor product. So we have successfully defined $C\colon \cA_1 \maxblacktimes \cA_2 \to \cA_1 \maxboxtimes \cA_2$.

Let $F\colon \cA_1 \maxblacktimes \cA_2 \to \cB$ be a cubical $\dag$-2-functor, with data of $F$ as in Proposition \ref{prop:dataofcubicals}. In what follows, we denote the interchanger of $F$ by $\Sigma^F$. We first determine a strict $\dag$-2-functor $F\colon \cA_1 \boxtimes \cA_2 \to \cB$:
\begin{itemize}
\item[(0)] on objects,
$$F(A_1,A_2) \coloneqq F_{A_1}(A_2) = F_{A_2}(A_1);$$
\item[(1)] on generator 1-morphisms, 
\begin{align*}
 F(X,\id_{A_2}) & \coloneqq F_{A_2}(X),\\
 F(\id_{A_1}, Y) &\coloneqq F_{A_1}(Y);
\end{align*}
\item[(2)] on generator 2-morphisms by
\begin{align*}
F(x \otimes \id_{\id_{A_2}}) &\coloneqq F_{A_2}(x)\\
F(\id_{\id_{A_1}} \otimes y) &\coloneqq F_{A_1}(y)\\
F(\Sigma_{X,Y}) &\coloneqq \Sigma^F_{X,Y}.
\end{align*}
We extend $F$ to all of $\cA_1 \boxtimes \cA_2$ as a strict $\dag$-2-functor, which is well-defined by construction.
\end{itemize}
Now consider some arbitrary 2-morphism $w$ in $\cA_1 \boxtimes \cA_2$. Let $\yo^\amalg\colon \cB \to \CstarCats$ be the universal representation on $\cB$, which is faithful on all levels. Notice that $\yo^\amalg$ is isometric on 2-morphisms by the C*-identity and the fact that injective maps between (endomorphism) C*-algebras are isometric. Since $\yo^\amalg \circ F\colon \cA_1 \boxtimes \cA_2 \to \CstarCats$ is a $\dag$-2-functor, we obtain
$$\|F(w)\| = \|(\yo^\amalg \circ F)(w)\| \leq \|w\|_\mu.$$
Therefore, we may uniquely extend $F$ to $\cA_1 \maxboxtimes \cA_2$ by continuity. 
\end{verification}

\begin{proposition}[Unitary hom-tensor Adjunction]\label{prop:CstarCartesianClosed}
For strict C*-2-categories $\cA_1$, $\cA_2$, and $\cB$, we have that the following strict C*-2-categories are unitarily naturally equivalent:
\begin{align*}
\CstarTwoCatst(\bbC \to \cB) &\cong \cB,\\
\CstarTwoCatst(\cA_1 \maxboxtimes \cA_2 \to \cB) &\cong 
\CstarTwoCatst( \cA_1 \to \CstarTwoCatst(\cA_2 \to \cB)).
\end{align*}
\end{proposition}

\begin{proof}
We will merely sketch the assignments at the level of objects for both isomorphisms.
For the first isomorphism, note that we are viewing $\bbC = B^2 \bbC$ as a C*-2-category with one object $\bullet$, a single 1-morphism $\id_\bullet$, with 
$$\bbC(\id_\bullet \tto \id_\bullet) \coloneqq \bbC.$$
Composition is given by multiplication and $\dag$ is given by conjugation. 
From this it is easy to see that the assignment 
$$
(F\colon \bbC \to \cB) \mapsto F(\bullet)
$$
extends to a strict $\dag$-2-functor which is bijective on all levels. 
For the latter isomorphism, one uses Proposition \ref{prop:dataofcubicals} to produce a bijective correspondence
$$
(F\colon \cA_1 \maxboxtimes \cA_2 \to \cB) 
\mapsto 
(A_1 \in \cA_1 \mapsto (F_{A_1}\colon \cA_2 \to \cB)).
$$
Here we have omitted the assignment of 1-morphisms and 2-morphisms. We do however note that the interchanger $\Sigma$ is used to form part of the naturator unitary for the $\dag$-2-natural transformation corresponding to a 1-morphism $X_1$ in $\cA_1$. The naturality axiom for this $\dag$-2-natural transformation follows by $(\Sigma 1)$ and $(\Sigma 3)$, while the strictness of this $\dag$-2-functor follows by cubicality and $(\Sigma 2)$. For 2-morphisms in $\cA_1$, the obviously assigned uniformly bounded modification in $\CstarTwoCatst(\cA_2 \to \cB)$ satisfies the modification axiom by $(\Sigma 1)$ as well.
\end{proof}

We obtain the following result as a corollary of \cite[\S II.3,4]{Eilenberg_Kelly_1966}.
\begin{corollary}
$\CstarGray \coloneqq (\CstarTwoCatst, \maxboxtimes)$ forms a closed symmetric monoidal category.
\end{corollary}

\subsection{W*-completion of a C*-2-category}\label{sec:2}

For a C*-2-category $\cA$, we wish to construct the enveloping W*-2-category $\rmW^*(\cA)$ together with a monic $\dag$-2-functor $\cA \hookrightarrow \rmW^*(\cA)$, which satisfies the following universal property:
\begin{itemize}
\item For every $\dag$-2-functor $F\colon \cA \to \cB$ into a W*-2-category $\cB$, there exists a unique normal extension making the following diagram commute:
\end{itemize}
\[
\begin{tikzcd}
	\cA & \cB \\
	{\rmW^*(\cA)}
	\arrow[hook, from=1-1, to=2-1]
	\arrow["F", from=1-1, to=1-2]
	\arrow["{\exists!\widetilde{F}}"', dashed, from=2-1, to=1-2]
\end{tikzcd}
\]

Consider the C*-category enriched graph $\cA^{**}$ with vertices $\Ob \cA$ and edges $\cA^{**}(A \to B) \coloneqq \cA(A \to B)^{**}$ for each pair $A,B \in \cA$. Here $\cA(A \to B)^{**}$ is the enveloping W*-category described in Appendix \ref{subsec:Wstarcompletion1cats}. We define two Arens 1-compositions on $\cA^{**}$, which serve to equip $\cA^{**}$ with the structure of a $\dag$-2-category.

\begin{definition}
For 1-morphisms in $\cA^{**}$, we define the left and right Arens 1-compositions $\ccirc_\ell$ and $\ccirc_r$ to act as 1-composition $\ccirc$ on $\cA$.

For 1-composable 2-morphisms $\Phi \in \cA^{**}(\morph{A}{X}{B} \tto \morph{A}{Y}{B})$ and $\Psi \in \cA^{**}(\morph{B}{X'\!}{C} \tto \morph{B}{Y'\!}{C})$, we 
define 
$$\Phi \ccirc_\ell \Psi,\, \Phi \ccirc_r \Psi \in \cA^{**}(\morph{A}{X \ccirc_B X'\!}{C} \tto \morph{A}{Y \ccirc_B Y'\!}{C})$$
\begin{itemize}
\item[($\ell$)] 
For $\varphi \in \cA(\morph{A}{X \ccirc X'\!}{C} \tto \morph{A}{Y \ccirc Y'\!}{C})^*$, we set $(\Phi \ccirc_\ell \Psi)(\varphi) \coloneqq \Phi(\varphi \triangleleft \Psi)$ where 
$$\varphi \triangleleft \Psi \in \cA(\morph{A}{X}{B} \tto \morph{A}{Y}{B})^*$$
is given by:
\begin{itemize}
    \item[($\triangleleft$)] For $a \in \cA(\morph{A}{X}{B} \tto \morph{A}{Y}{B})$, we set $(\varphi \triangleleft \Psi)(a) \coloneqq \Psi(a \triangleright \varphi)$ where $$a \triangleright \varphi \in \cA(\morph{B}{X'\!}{C} \tto \morph{B}{Y'\!}{C})^*$$
    is given by: 
\begin{itemize}
    \item[($\triangleright$)] For $b \in \cA(\morph{B}{X'\!}{C} \tto \morph{B}{Y'\!}{C})$, we set $(a \triangleright \varphi)(b) \coloneqq \varphi(a \ccirc b)$. 
\end{itemize}
\end{itemize}
More succinctly, $\Phi \ccirc_\ell \Psi$ is given by the following formula:
\begin{align*}
\Phi \ccirc_\ell \Psi 
&= \varphi \mapsto \Phi\Big(\varphi \triangleleft \Psi\Big),\\
&= \varphi \mapsto \Phi\Big(a \mapsto \Psi\big(a \triangleright \varphi\big)\Big),\\
&= \varphi \mapsto \Phi\Big(a \mapsto \Psi\big(b \mapsto \varphi(a \ccirc b)\big)\Big).
\end{align*}
\item[($r$)] 
For $\varphi \in \cA(\morph{A}{X \ccirc X'\!}{C} \tto \morph{A}{Y \ccirc Y'\!}{C})^*$, we set $(\Phi \ccirc_r \Psi)(\varphi) \coloneqq \Psi(\Phi \triangleright \varphi)$ where 
$$\Phi \triangleright \varphi \in \cA(\morph{B}{X'\!}{C} \tto \morph{B}{Y'\!}{C})^*$$
is given by:
\begin{itemize}
    \item[($\triangleright$)] For $b \in \cA(\morph{B}{X'\!}{C} \tto \morph{B}{Y'\!}{C})$, we set $(\Phi \triangleright \varphi)(b) \coloneqq \Phi(\varphi \triangleleft b)$ where 
    $$\varphi \triangleleft b \in \cA(\morph{A}{X}{B} \tto \morph{A}{Y}{B})^*$$
    is given by:
\begin{itemize}
    \item[($\triangleleft$)] For $a \in \cA(\morph{A}{X}{B} \tto \morph{A}{Y}{B})$, we set $(\varphi \triangleleft b)(a) \coloneqq \varphi(a \ccirc b)$. 
\end{itemize}
\end{itemize}
More succinctly, $\Phi \ccirc_r \Psi$ is given by the following formula:
\begin{align*}
\Phi \ccirc_r \Psi 
&= \varphi \mapsto \Psi\Big(\Phi \triangleright \varphi \Big),\\
&= \varphi \mapsto \Psi\Big(b \mapsto \Phi\big(\varphi \triangleleft b\big)\Big),\\
&= \varphi \mapsto \Psi\Big(b \mapsto \Phi\big(a \mapsto \varphi(a \ccirc b)\big)\Big).
\end{align*}
\end{itemize}
\end{definition}

\begin{proposition}\label{prop:ArensCoincide2Cats}
For a C*-2-category $\cA$, the left and right Arens 1-compositions on $\cA^{**}$ coincide and serve to equip $\cA^{**}$ with the structure of a W*-2-category.
\end{proposition}

\begin{proof}
Notice the following facts about $\cA^{**}_\ell \coloneqq (\cA^{**},\ccirc_\ell)$ and $\cA^{**}_r \coloneqq(\cA^{**},\ccirc_r)$.
\begin{itemize}
\item $\ev_a \ccirc_\ell \ev_b = \ev_{a \ccirc b} = \ev_a \ccirc_r \ev_b$ for 1-composable 2-morphisms $a,b$ in $\cA$. 
Therefore, we may upgrade $\ev$ into $\dag$-2-functors $\cA \hookrightarrow \cA^{**}_\ell $ and $\cA \hookrightarrow \cA^{**}_r $ which act as the identity on objects, and as the $\dag$-functor $\ev$ on hom W*-categories. 
\item $\cA^{**}_\ell$ and $\cA^{**}_r$ inherit associators $\ev_{\alpha_{XYZ}}$ which are natural since $\Im \ev$ is dense in $\cA^{**}$ and 2-composition is weak*-continuous in each hom W*-category of $\cA^{**}$.
Similarly, units $\id_A$ and unitors $\ev_{\lambda_X}$, $\ev_{\rho_X}$ are inherited from $\cA$.
\end{itemize}
Therefore $\cA^{**}_\ell$ and $\cA^{**}_r$ are C*-2-categories.
It remains to show that $\ccirc_\ell$ and $\ccirc_r$ are separately normal.
It is clear that $- \ccirc_\ell \Psi$ and $\Phi \ccirc_r -$ are normal for 2-morphisms $\Phi,\Psi$ in $\cA^{**}$.

We will now show condition \hyperref[condition:W*2']{(W*2$'$)}, that is $- \ccirc_\ell \ev_{\id_{X'}}$ and $\ev_{\id_{X}} \ccirc_\ell -$ are separately normal for 1-morphisms $X,X'$ in $\cA^{**}$.
For $\Phi_\lambda \to \Phi$ in the weak topology, we have
$$
(\Phi_\lambda \ccirc_\ell \ev_{\id_{X'}})(\varphi)
=
\Phi_\lambda(\varphi \triangleleft \ev_{\id_{X'}})
\to 
\Phi(\varphi \triangleleft \ev_{\id_{X'}}) 
=
(\Phi \ccirc_\ell \ev_{\id_{X'}})(\varphi).
$$
Hence $\Phi_\lambda \ccirc_\ell \ev_{\id_{X'}} \to \Phi \ccirc_\ell \ev_{\id_{X'}}$ weakly. 
Moreover, for $\Psi_\lambda \to \Psi$ in the weak topology, observe
\begin{align*}
(\ev_{\id_X} \ccirc \Psi_\lambda)(\varphi) 
&=    
\ev_{\id_X}(\varphi \triangleleft \Psi_\lambda)\\
&=
\Psi_\lambda(b \mapsto \varphi(\id_X \ccirc b))\\
&\to 
\Psi(b \mapsto \varphi(\id_X \ccirc b))\\
&= 
(\ev_{\id_X} \ccirc \Psi_\lambda)(\varphi).
\end{align*}
Hence $\ev_{\id_X} \ccirc \Psi_\lambda \to \ev_{id_X} \ccirc \Psi$ weakly. Therefore $\ccirc_\ell$ is separately normal, and one similarly shows $\ccirc_r$ is separately normal. Finally, since $\Im \ev$ is dense at the level of two morphisms, $\ccirc_\ell$ and $\ccirc_r$ are separately normal, and they agree on $\Im \ev$, we conclude that $\ccirc_\ell = \ccirc_r$ on $\cA^{**}$.
\end{proof}

\begin{universalprop}
For every $\dag$-2-functor $F\colon \cA \to \cB$ into a W*-2-category $\cB$, there exists a unique normal $\dag$-2-functor $\widetilde{F}\colon \cA^{**} \to \cB$ making the following diagram commute.
\[
\begin{tikzcd}
	\cA & \cB \\
	{\cA^{**}}
	\arrow[hook, from=1-1, to=2-1, "\ev"']
	\arrow["F", from=1-1, to=1-2]
	\arrow["{\exists!\widetilde{F}}"', dashed, from=2-1, to=1-2]
\end{tikzcd}
\]
\end{universalprop}

\begin{verification}
First note that $\widetilde{F}$ must acts on objects as $F$. Then, for every hom C*-category $\cA(A \to B)$, we may extend $F$ to $\cA^{**}(A \to B)$ by the universal property of W*-envelopes or C*-categories. This yields a map $\widetilde{F}: \cA^{**} \to \cB^{**}$ which is locally a normal $\dag$-functor. Then observe that the tensorator and unitor of $F$ equips $\widetilde{F}$ with the structure of a $\dag$-2-functor since $\ccirc$ and $\circ$ are separately weak*-continuous and $\Im \ev$ is weak*-dense in $\cA^{**}$. 
\end{verification}

As a corollary, we obtain another proof of the concreteness theorem for C*-2-categories.

\begin{theorem}[Gelfand-Naimark for C*-2-categories]
For a small C*-2-category $\cA$, the universal representation
$$\Upsilon_2\colon \cA \to \TwoHilb$$
given by $\cA \xrightarrow{\ev} \cA^{**} \xrightarrow{\yo^\amalg} \WstarCats \xrightarrow{\GNS} \TwoHilb$ 
is monic. Thus, every C*-2-category $\cA$ can be realized as a norm-closed strict $\dag$-2-category $\GNS_2(\cA) \coloneqq \Im \Upsilon_2$ of W*-categories of Hilbert spaces and operators. Moreover, if $\cA$ is strict, $\Upsilon_2$ is a strict $\dag$-2-functor. 
\end{theorem}

\begin{remark}
We expect a Sherman-\!Takeda theorem to hold for C*-2-categories.
Namely, for a small C*-2-category $\cA$, the $\dag$-2-functor 
$$\widetilde{\Upsilon}_2\colon \cA^{**} \to \GNS_2(\cA)''$$
given by the universal property of the W*-envelope of a C*-2-category is an equivalence of W*-2-categories extending $\Upsilon_2\colon \cA \to \GNS_2(\cA)$.
\[\begin{tikzcd}
	\cA & {\GNS_2(\cA)} \\
	{\cA^{**}} & {\GNS_2(\cA)''}
	\arrow["\ev"', hook, from=1-1, to=2-1]
	\arrow["\Upsilon_2", from=1-1, to=1-2]
	\arrow["{\widetilde{\Upsilon}}_2"', from=2-1, to=2-2]
	\arrow[hook, from=1-2, to=2-2]
\end{tikzcd}\]
\end{remark}

\subsection{W*-Gray tensor product}\label{sec:W*Graytensor}

For strict W*-2-categories $\cA_1$ and $\cA_2$, we wish to construct a W*-2-category $\cA_1 \Wboxtimes \cA_2$ together with a separately normal cubical $\dag$-2-functor $\cA_1 \times \cA_2 \to \cA_1 \Wboxtimes \cA_2$ satisfying the following universal property:
\begin{itemize}
    \item For every separately normal cubical $\dag$-2-functor $H\colon \cA_1 \maxblacktimes \cA_2 \to \cB$ into a strict W*-2-category $\cB$, there exists a unique normal $\dag$-2-functor $H\colon \cA_1 \Wboxtimes \cA_2 \to \cB$ such that the following diagram commutes:
\[\begin{tikzcd}
\cA_1 \maxblacktimes \cA_2 \arrow[r,"H"] \arrow[d] & \cB\\
\cA_1 \Wboxtimes \cA_2 \arrow[ur,dashed,"H"'] & 
\end{tikzcd}\]
\end{itemize}
\begin{definition}
For strict W*-2-categories $\cA_1$, $\cA_2$, and composible 1-morphisms $X_i, Y_i \in \cA_i$ for $i = 1,2$, we define 
$$
(\cA_1 \Wboxtimes \cA_2)\big((X_1,X_2) \to (Y_1,Y_2)\big)_*
\subset (\cA_1 \maxboxtimes \cA_2)\big((X_1,X_2) \to (X_1,X_2)\big)^*$$
as follows. First, let $\SN\big((X_1,X_2) \to (Y_1,Y_2)\big)$ be the closed subspace consisting of all functionals which are separately normal, i.e. the functionals $\varphi \in (\cA_1 \maxboxtimes \cA_2)\big((X_1,X_2) \to (Y_1,Y_2)\big)^*$ such that 
\begin{align*}
\varphi \circ (\id_{\id_{A_1}} \otimes -) \in \cA_2(X_2 \to Y_2)_*,  \quad&\text{for every } A_1 \in \cA_1, \text{ and}\\
\varphi \circ (- \otimes \id_{\id_{A_2}}) \in \cA_1(X_1 \to X_1)_*,  \quad&\text{for every } A_2 \in \cA_2.
\end{align*}
We then define $(\cA_1 \Wboxtimes \cA_2)_*$ to be the largest closed subspace of $\SN$ invariant under the four actions of $\cA_1 \maxboxtimes \cA_2$ on $(\cA_1 \maxboxtimes \cA_2)^*$ given by precomposing and postcomposing the argument of $\varphi \in \SN$ horizontally (1-composition $\otimes$) and vertically (2-composition $\circ$). 
\end{definition}
\begin{remark}
We are using the suggestive notation $(\cA_1 \Wboxtimes \cA_2)\big((X_1,X_2) \to (Y_1,Y_2)\big)_*$ even though we have not yet defined the W*-category $\cA_1 \Wboxtimes \cA_2$. However, once we do construct this W*-Gray tensor product, we will see that this Banach space is indeed the predual of the hom space $(\cA_1 \Wboxtimes \cA_2)\big((X_1,X_2) \to (Y_1,Y_2)\big)$.
\end{remark}

The following result will serve as motivation for the definition of the W*-Gray tensor product $\Wboxtimes$.

\begin{proposition}\label{prop:IsomorphicW*CategoryConstruction}
For objects $X,Y \in \cB$ in a strict W*-2-category $\cB$, consider the polar
$$\cB(X \tto Y)_*^\perp \coloneqq \sset{\Phi \in \cB^{**}(X \tto Y)}{\Phi(\varphi) = 0 \text{ for all } \varphi \in \cB(X \tto Y)_*}.$$
Then $\cB_*^\perp \subseteq \cB^{**}$ is closed under pre- and post- 1-composition and 2-composition of 2-morphisms in $\cB^{**}$, so that $\cB^{**}/\cB_*^\perp$ is a (strict) W*-2-category. 
Moreover, if $\pi\colon \cB^{**} \to \cB^{**}/\cB^\perp_*$ is the natural quotient $\dag$-2-functor, then the composite
$$\cB \xhookrightarrow{\ev} \cB^{**} \xrightarrow{\pi} \cB^{**}/\cB^\perp_*$$
is an equivalence of (strict) W*-2-categories.
\end{proposition}
\begin{proof}
Let us first show that $\cB_*^\perp$ is closed under both pre-compositions and post-compositions with 2-morphisms in $\cB^{**}$. First note that normal functionals on hom spaces of $\cB$ are closed under the four actions of $\cB$ since 1-composition and 2-composition of 2-morphisms are separately normal in a W*-2-category. One can then use the fact that left and right Arens 1-compositions and 2-compositions agree respectively to show that $\cB^\perp_*$ is closed under pre- and post- 1-composing and 2-composing 2-morphisms in $\cB^{**}$. Indeed, if $\Phi \in \cB(X \tto Y)_*^\perp$ and $\Psi \in \cB^{**}(Y \tto Z)$ for parallel 1-morphisms $X,Y,Z$ in $\cB$, there exists some $\ev_{x_\lambda} \to \Psi$ weak* in $\cB^{**}$. So for $\varphi \in \cB(X \tto Z)_*$ we have $\{\varphi \triangleleft x_\lambda\} \subset \cB(X \tto Y)_*$ and hence
$$
(\Psi \circ \Phi)(\varphi) = \lim_\lambda (\ev_{x_\lambda} \circ_r \Phi)(\varphi) =
\lim_\lambda \Phi(\varphi \triangleleft x_\lambda) = 0.
$$
Therefore $\Psi \circ \Phi \in \cB(X \tto Z)_*^\perp$, and one similarly shows our three other claims.
Therefore $\cB^{**}/\cB^\perp_*$ is a C*-2-category. For Banach spaces $E \subseteq F$, it is an elementary fact that
$$F^*/E^\perp \cong E^*$$
via the map $[\Phi] \to \Phi|_E$ for $\Phi \in F^*$.
In particular, we have
$$(\cB^{**}/\cB^\perp_*)(X \tto Y)_* = \cB(X \tto Y)_*,$$
i.e. $(\cB^{**}/\cB_*^\perp)(X \tto Y) \cong \cB(X \tto Y)_*^* \cong \cB(X \tto Y)$. We then conclude $\cB^{**}/\cB^\perp_*$ is a W*-2-category.

It is now easy to see that $\pi \circ \ev\colon \cB \to \cB^{**}/\cB_*^\perp$ is an isomorphism of W*-2-categories.
Indeed, when viewing $\cB(X \tto Y)_* \subseteq \cB(X \tto Y)^*$ as the subspace of normal functionals on $\cB(X \tto Y)$, the isomorphism 
$$\cB(X \tto Y) \to \cB(X \tto Y)_*^*$$
is given by $x \mapsto \ev_x|_{\cB(X \tto Y)_*}$. 
On the other hand, the isomorphism 
$$(\cB^{**}/\cB_*^\perp)(X \tto Y) \to \cB(X \tto Y)_*^*$$
is given by $[\Phi] \to \Phi|_{\cB(X \tto Y)_*}$.
Hence, the map $\pi \circ \ev$ given by $x \mapsto [\ev_x]$ is an isomorphism on each hom-space. Since both $\pi$ and $\ev$ act as the identity on objects and 1-morphisms, we conclude that $\pi \circ \ev$ is an (automatically normal) equivalence of W*-2-categories.\qedhere
\end{proof}

\begin{definition}
For strict W*-2-categories $\cA_1,\cA_2$, we define their W*-Gray tensor product to be
$$\cA_1 \Wboxtimes \cA_2 \coloneqq (\cA_1 \maxboxtimes \cA_2)^{**} / (\cA_1 \Wboxtimes \cA_2)_*^\perp.$$ 
\end{definition}

\begin{lemma}
The W*-Gray tensor product $\cA_1 \Wboxtimes \cA_2$ is a (strict) W*-2-category.
\end{lemma}
\begin{proof}
Since $(\cA_1 \Wboxtimes \cA_2)_*^\perp \subseteq (\cA_1 \maxboxtimes \cA_2)^{**}$ is closed under pre- and post- 1-composition and 2-composition of 2-morphisms in $\cA_1 \Wmaxtimes \cA_2$ by definition, it follows that $\cA_1 \Wmaxtimes \cA_2$ is indeed a well-defined C*-2-category. As before, it is an elementary fact that the predual of each hom-space is the previously-defined $(\cA_1 \Wboxtimes \cA_2)_*$. Hence $\ccirc$ is separately normal on $\cA_1 \Wboxtimes \cA_2$ as it is separately normal on $(\cA_1 \maxboxtimes \cA_2)^{**}$, and we conclude that $\cA_1 \Wboxtimes \cA_2$ is indeed a W*-2-category.
\end{proof}

For the remaining portion of this section, we shall fix W*-2-categories $\cA_1,\cA_2,\cB$ and a separately normal cubical $\dag$-2-functor $H\colon \cA_1 \maxblacktimes \cA_2 \to \cB$. 
\begin{lemma}\label{lem:pullbacknormalfunctionaltograytensor}
Consider the induced $\dag$-2-functor $\widetilde{H}\colon \cA_1 \maxboxtimes \cA_2 \to \cB$ on the C*-Gray tensor product. For parallel 1-morphisms $X_i,Y_i$ in $\cA_i$ where $i = 1,2$ and $\psi \in \cB\big(\widetilde{H}(X_1,X_2) \to \widetilde{H}(Y_1,Y_2)\big)_*$, we may define 
$$\widetilde{H}^*\psi \in (\cA_1 \maxboxtimes \cA_2)\big((X_1,X_2) \to (Y_1,Y_2)\big)^* \quad \text{by } (\widetilde{H}^*\psi)(t) \coloneqq \psi(\widetilde{H}t).$$
Then $\widetilde{H}^*\psi \in (\cA_1 \Wboxtimes \cA_2)\big((X_1,X_2) \to (X_1,X_2)\big)_*$. We will denote this by
$$\widetilde{H}^* \cB_* \subseteq (\cA_1 \Wboxtimes \cA_2)_*$$
for simplicity.
\end{lemma}
\begin{proof}
Let $A_1 \in \cA_1$ and suppose $a_2^\lambda \to a_2$ weak* in $\cA_2(A_2 \to B_2)$. Since $H$ is separately normal, $H(\id_{\id_{A_1}} \otimes a_2^\lambda) \to H(\id_{\id_{A_1}} \otimes a_2)$ $\sigma$-WOT. Since $\psi$ is normal, we conclude
$$(\widetilde{H}^* \psi)(\id_{\id_{A_1}} \otimes a_2^\lambda) = \psi(H(\id_{\id_{A_1}} \otimes a_2^\lambda)) \to \psi(H(\id_{\id_{A_1}} \otimes a_2)) = (\widetilde{H}^* \psi)(\id_{\id_{A_1}} \otimes a_2).$$
A similar argument reveals $(\widetilde{H}^* \psi)(a^\lambda_1 \otimes \id_{\id_{A_2}}) \to (\widetilde{H}^* \psi)(a_1 \otimes \id_{\id_{A_1}})$ when $a_1^\lambda \to a_1$ weak* in $\cA_1$. Therefore $\widetilde{H}^* \psi$ is a separately normal functional. Moreover, since $H$ is a $\dag$-2-functor and both 1-composition and 2-composition of 2-morphisms in $\cB$ is separately normal, it is easy to see that $\widetilde{H}^*\psi \in (\cA_1 \Wboxtimes \cA_2)_*$.
\end{proof}
Note that the cubical $\dag$-2-functor $\cA_1 \maxblacktimes \cA_2 \to \cA_1 \Wboxtimes \cA_2$ given by
$$\cA_1 \maxblacktimes \cA_2 \to \cA_1 \maxboxtimes \cA_2 \xrightarrow{\ev} (\cA_1 \maxboxtimes \cA_2)^{**} \xrightarrow{\pi} \cA_1 \Wboxtimes \cA_2$$
is separately normal since $(\cA_1 \Wboxtimes \cA_2)_*$ consists of functionals on $\cA_1 \maxtimes \cA_2$ which are separately normal.

\begin{universalprop}
For strict W*-2-categories $\cA_1$ and $\cA_2$, if $H\colon \cA_1 \maxblacktimes \cA_2 \to \cB$ is a separately normal cubical $\dag$-2-functor into a strict W*-2-category $\cB$, then there exists a unique normal $\dag$-2-functor $\overline{H}\colon \cA_1 \Wboxtimes \cA_2 \to \cB$ such that the following diagram commutes:
\[\begin{tikzcd}
\cA_1 \maxblacktimes \cA_2 \arrow[r,"H"] \arrow[d] & \cB\\
\cA_1 \Wboxtimes \cA_2 \arrow[ur,dashed,"\overline{H}"'] & 
\end{tikzcd}\]
\end{universalprop}

\begin{proof}
By Lemma \ref{lem:pullbacknormalfunctionaltograytensor}, $\widetilde{H}^*\cB_* \subseteq (\cA_1 \Wboxtimes \cA_2)_*$, so $(\cA_1 \Wboxtimes \cA_2)_*^\perp \subseteq (\widetilde{H}^*\cB)_*^\perp$ and $\widetilde{H}^{**}(\cA_1 \Wboxtimes \cA_2)_*^\perp \subseteq \cB_*^\perp$. We then obtain a morphism of short exact sequences:
\[\begin{tikzcd}[background color = Dandelion!5,row sep = 30pt]
	0 & {(\cA_1 \Wboxtimes \cA_2)_*^\perp} & {(\cA_1 \maxboxtimes \cA_2)^{**}} & {\cA_1 \Wboxtimes \cA_2} & 0 \\
	0 & {\cB_*^\perp} & {\cB^{**}} & {\cB^{**}/\cB_*^\perp} & 0
	\arrow[from=1-1, to=1-2]
	\arrow[from=2-1, to=2-2]
	\arrow[from=2-2, to=2-3]
	\arrow[from=2-3, to=2-4,"\pi"']
	\arrow[from=2-4, to=2-5]
	\arrow[from=1-4, to=1-5]
	\arrow[from=1-3, to=1-4,"\pi"]
	\arrow[from=1-2, to=1-3]
	\arrow[from=1-2, to=2-2]
	\arrow[from=1-3, to=2-3,"\widetilde{H}^{**}"']
	\arrow[from=1-4, to=2-4,dashed,"{[\widetilde{H}^{**}]}"]
\end{tikzcd}\]
We now define the normal functor $\dag$-2-functor $\overline{H}\colon \cA_1 \Wboxtimes \cA_2 \to \cB$ to be $$\overline{H} \coloneqq  \cA_1 \Wboxtimes \cA_2  \xrightarrow{[\widetilde{H}^{**}]} \cB/\cB^{\perp} \xrightarrow{(\pi \circ \ev)^{-1}} \cB.$$ The desired triangle commutes by the commutativity of the following diagram:
\[\begin{tikzcd}[column sep = 30pt]
	{\cA_1 \maxblacktimes \cA_2} & \cB \\
	{\cA_1 \maxboxtimes \cA_2} & \cB \\
	{(\cA_1 \maxboxtimes \cA_2)^{**}} & {\cB^{**}} \\
	{\cA_1 \Wboxtimes \cA_2} & {\cB/\cB_*^\perp}
	\arrow[from=1-1, to=2-1]
	\arrow["\ev"', from=2-1, to=3-1]
	\arrow["\pi"', from=3-1, to=4-1]
	\arrow["{\widetilde{H}}"{description}, from=2-1, to=2-2]
	\arrow["{\widetilde{H}^{**}}"{description}, from=3-1, to=3-2]
	\arrow["{[\widetilde{H}^{**}]}"', from=4-1, to=4-2]
	\arrow["\pi", from=3-2, to=4-2]
	\arrow["\ev", from=2-2, to=3-2]
	\arrow["H", from=1-1, to=1-2]
	\arrow[no head, from=1-2, to=2-2]
	\arrow[shift left=1, no head, from=1-2, to=2-2]
\end{tikzcd}\]
To see that $\overline{H}$ is unique, suppose $K \colon \cA_1 \Wboxtimes \cA_2 \to \cB$ is a normal $\dag$-2-functor making the following diagram commute:
\[\begin{tikzcd}
\cA_1 \maxblacktimes \cA_2 \arrow[r,"H"] \arrow[d] & \cB\\
\cA_1 \Wboxtimes \cA_2 \arrow[ur,"K"'] & 
\end{tikzcd}\]
By the universal property of the C*-Gray tensor product, we have 
$$\widetilde{H} =
\cA_1 \maxboxtimes \cA_2 \xrightarrow{\ev} (\cA_1 \maxboxtimes \cA_2)^{**} \xrightarrow{\pi} \cA_1 \Wboxtimes \cA_2 \xrightarrow{K} \cB.$$
We also know $\ev \circ \widetilde{H} = \widetilde{H}^{**}  \circ \ev$, so by the universal property of W*-completion, we have 
$$
\widetilde{H}^{**} =
(\cA_1 \maxboxtimes \cA_2)^{**} \xrightarrow{\pi} \cA_1 \Wboxtimes \cA_2 \xrightarrow{K} \cB \xrightarrow{\ev} \cB^{**}. 
$$
Finally, we know $\pi \circ \widetilde{H}^{**} = [\widetilde{H}^{**}] \circ \pi$, so the surjectivity of $\pi$ implies
$$[\widetilde{H}^{**}] =
\cA_1 \Wboxtimes \cA_2 \xrightarrow{K} \cB \xrightarrow{\ev} \cB^{**} \xrightarrow{\pi} \cB^{**}/\cB_*^\perp.
$$
From this we conclude that $K = \overline{H}$ and hence $\overline{H}$ is unique.
\end{proof}

\begin{definition}
Let $F_i\colon \cA_i \to \cA_i'$ be normal $\dag$-2-functors between strict W*-2-categories $\cA_i$ for $i = 1,2$. We define $F_1 \Wboxtimes F_2$ to be the normal $\dag$-2-functor induced by the universal property of $\Wboxtimes$ as follows:
\[\begin{tikzcd}[column sep = 30pt]
	{\cA_1 \maxblacktimes \cA_2} & {\cA_1' \maxblacktimes \cA_2'} \\
	{\cA_1 \Wboxtimes \cA_2} & {\cA_1' \Wboxtimes \cA_2'}
	\arrow["{F_1 \maxblacktimes F_2}", from=1-1, to=1-2]
	\arrow[from=1-1, to=2-1]
	\arrow[from=1-2, to=2-2]
	\arrow["{F_1 \Wmaxtimes F_2}"', dashed, from=2-1, to=2-2]
\end{tikzcd}\]
\end{definition}

\begin{remark}
Due to the uniqueness of normal $\dag$-2-functors induced by the universal property of $\Wboxtimes$, it is easy to see that $\Wboxtimes\colon \WstarTwoCatst \times \WstarTwoCatst \to \WstarTwoCatst$ is a functor when viewing $\Wboxtimes$ as a 1-category of strict W*-2-categories and normal $\dag$-2-functors.
\end{remark}

\begin{theorem}
[Unitary hom-Tensor Adjunction] For strict W*-2-categories $\cA_1$, $\cA_2$, and $\cB$, we have that the following W*-2-categories are unitarily naturally isomorphic:
\begin{align*}
\WstarTwoCatst(\bbC \to \cB) &\cong \cB,\\
\WstarTwoCatst(\cA_1 \Wboxtimes \cA_2 \to \cB) &\cong \WstarTwoCatst(\cA_1 \to \WstarTwoCatst(\cA_2 \to \cB)).
\end{align*}
\end{theorem}

\begin{proof}
This follows from an identical proof to that of Proposition \ref{prop:CstarCartesianClosed}, using the fact that $\Wboxtimes$ represents separately normal cubical functors to obtain normal strict $\dag$-2-functors as expected.
\end{proof}

As in the C* case, we obtain the following as a corollary of \cite[\S II.3,4]{Eilenberg_Kelly_1966}.

\begin{corollary}
$\WstarGray \coloneqq (\WstarTwoCatst, \Wboxtimes)$ forms a closed symmetric monoidal category.
\end{corollary}

\subsection{Cofibrant replacement for operator 2-categories}

In the following section, we provide cofibrant replacement results for operator 2-categories, (normal) $\dag$-functors, as well as a tensorator for this replacement. This section follows the treatment in \cite[\S 2.2]{gurski_2013} closely, so we have provided the details in Appendix \S \ref{subsec:CofibrantReplacement}.

\begin{proposition}
For every C*-2-category $\cA$, there exists a strict C*-2-category $\widehat{\cA}$ together with an epic 2-equivalence $\ev_{\cA}\colon \widehat{\cA} \to \cA$. Moreover, when $\cA$ is a W*-2-category we have that $\widehat{\cA}$ is a W*-2-category and hence $\ev_{\cA}$ is automatically normal.
\end{proposition}

\begin{proposition}\label{prop:cofibfunctors}
For each $\dag$-2-functor $F\colon \cA \to \cB$ between C*-2-categories, there exists a strict $\dag$-2-functor $\widehat{F}\colon \widehat{\cA} \to \widehat{\cB}$ and a unitary icon\footnote{By an icon, we mean a strictified version of a 2-natural transformation between 2-functors that agree on objects \cite{Lack_2008}.} $u^F$ as follows: 
\[\begin{tikzcd}[row sep = 30pt, column sep = 30pt]
	{\widehat{\cA}} & {\widehat{\cB}} \\
	{\cA_1} & {\cA_2}
	\arrow["{\widehat{F}}", from=1-1, to=1-2]
	\arrow["{\ev_\cA}"', from=1-1, to=2-1]
	\arrow["{\ev_{\cB}}", from=1-2, to=2-2]
	\arrow["F"', from=2-1, to=2-2]
	\arrow["{u^F}"{description}, shorten <=4pt, shorten >=4pt, Rightarrow, from=1-2, to=2-1]
\end{tikzcd}\]
Moreover, when $F$ is a normal $\dag$-2-functor between W*-2-categories, then $\widehat{F}$ is also normal.
\end{proposition}

\begin{proposition}\label{prop:cofibtensorator}
For C*-2-categories $\cA_1$ and $\cA_2$, there exists a cubical $\dag$-2-functor
$$C\colon \widehat{\cA}_1 \maxblacktimes \widehat{\cA}_2 \to \widehat{\cA_1 \maxblacktimes \cA_2}$$
which is the identity on objects, and a unitary icon $u$ as follows:
\[\begin{tikzcd}[row sep = 30pt, column sep = 10pt]
	& {\widehat{\cA_1 \maxblacktimes \cA_2}} \\
	{\widehat{\cA}_1 \maxblacktimes \widehat{\cA}_2} && {\cA_1 \maxblacktimes \cA_2}
	\arrow["C", from=2-1, to=1-2]
	\arrow["\ev", from=1-2, to=2-3]
	\arrow[""{name=0, anchor=center, inner sep=0}, "{\ev \maxblacktimes \ev}"', from=2-1, to=2-3]
	\arrow["u^{\ev}"{description}, shorten >=3pt, Rightarrow, from=1-2, to=0]
\end{tikzcd}\]
Moreover, when $\cA_1$ and $\cA_2$ are W*-2-categories, we may upgrade $C$ to a separately normal cubical $\dag$-2-functor
$$\overline{C}\colon \widehat{\cA}_1 \Wblacktimes \widehat{\cA}_2 \to \widehat{\cA_1 \Wblacktimes \cA_2}$$
and $u$ to a unitary icon $\overline{u}$ as follows:
\[\begin{tikzcd}[row sep = 30pt, column sep = 10pt]
	& {\widehat{\cA_1 \Wblacktimes \cA_2}} \\
	{\widehat{\cA}_1 \Wblacktimes \widehat{\cA}_2} && {\cA_1 \Wblacktimes \cA_2}
	\arrow["\overline{C}", from=2-1, to=1-2]
	\arrow["\ev", from=1-2, to=2-3]
	\arrow[""{name=0, anchor=center, inner sep=0}, "{\ev \Wblacktimes \ev}"', from=2-1, to=2-3]
	\arrow["\overline{u}^{\ev}"{description}, shorten >=3pt, Rightarrow, from=1-2, to=0]
\end{tikzcd}\]
\end{proposition}

\section{3-categorical results} \label{sec:3cats}

\subsection{Operator 3-categories}

\begin{definition}
A C*-3-category $\cA$ consists of an algebraic tricategory in the sense of \cite{gurski_2013}, equipped with a C*-2-category structure on each hom 2-category $\cA(A \to B)$, such that the underlying coherence 2-functors, 2-natural transformations, and modifications are $\dag$-2-functors, $\dag$-2-natural transformations, and unitary modifications respectively.
\end{definition}
\begin{remark}
We unpack the data of a C*-3-category as follows:
\begin{enumerate}
\item[(0)] 
A collection $\Ob T$ of \emph{objects}, or \emph{0-cells}, of $\cA$;
\item[(hom)] 
For $A,B \in \Ob T$, a C*-2-category $\cA(A \to B)$ where:  
\begin{itemize}
\item
The objects of $\cA(A \to B)$ are called \emph{1-cells} of $\cA$ with source $A$ and target $B$,
\item
The arrows of $\cA(A \to B)$ will be referred to as \emph{2-cells} of $\cA$ with their same source and target, and
\item 
The 2-morphisms of $\cA(A \to B)$ are called \emph{3-cells} of $\cA$, also with their same source and target.
    \end{itemize}
As \cite{gurski_2013}, we will denote $\cA(A_0 \to A_1) \maxblacktimes \cdots \maxblacktimes \cA(A_{n-1} \to A_n)$ by the abbreviation
$$\cA^n(A_0 \to \cdots \to A_n).$$

\item[($\cccirc$)] 
For $A,B,C \in \Ob \cA$, a $\dag$-2-functor $\cccirc\colon \cA^2(A \to B \to C) \to \cA(A \to C)$ called \emph{1-composition}. In the graphical calculus for the monoidal 3-category $(\CstarTwoCat, \maxtimes)$, we represent $\cccirc$ by:
$$
\begin{tikzpicture}
    \draw[thick] (0,0) -- (0,.5);
    \draw[thick] (-.5,-.5) arc (180:0:.5);
    \node at (-.75,-.75) {\scriptsize $\cA(A \to B)$};
    \node at (.75,-.75) {\scriptsize $\cA(B \to C)$};
    \node at (0,.75) {\scriptsize $\cA(A \to C)$};
\end{tikzpicture}
$$
\item[(I)] 
For $A \in \Ob \cA$, a $\dag$-2-functor $I_A\colon \bbC \to \cA(A \to A)$ which we represent diagrammatically by:
$$
\begin{tikzpicture}
    \draw[thick] (0,0) -- (0,.5);
    \draw[thick] (0,0) circle (.02);
    \node at (0,.75) {\scriptsize $\cA(A \to A)$};
\end{tikzpicture}
$$
\item[(a)]
For $A,B,C,D \in \Ob \cA$, an \emph{``associator''} unitary adjoint equivalence $\mathbf{a} = (a,a^{\sqdot},\epsilon^a,\eta^a)$ where $a$ is a $\dag$-2-natural transformation 
$$
\begin{tikzpicture}[baseline={([yshift=0]current bounding box.center)}]
    \draw[thick] (-.5,-.5) arc (180:0:.5);
    \draw[thick] (-1,-.5) -- (-1,0) arc (180:0:.5);
    \draw[thick] (-.5,0.5) -- (-.5,1);
    \node at (-1.6,-.75) {\scriptsize $\cA(A \to B)$};
    \node at (-.25,-.75) {\scriptsize $\cA(B \to C)$};
    \node at (1.15,-.75) {\scriptsize $\cA(C \to D)$};
    \node at (-.5,1.25) {\scriptsize $\cA(A \to D)$};
\end{tikzpicture}
\;\xRightarrow{a}\;
\begin{tikzpicture}[baseline={([yshift=0]current bounding box.center)}]
    \draw[thick] (-.5,-.5) arc (180:0:.5);
    \draw[thick] (1,-.5) -- (1,0) arc (0:180:.5);
    \draw[thick] (.5,0.5) -- (.5,1);
    \node at (-1.6+.5,-.75) {\scriptsize $\cA(A \to B)$};
    \node at (-.25+.5,-.75) {\scriptsize $\cA(B \to C)$};
    \node at (1.15+.5,-.75) {\scriptsize $\cA(C \to D)$};
    \node at (-.5+1,1.25) {\scriptsize $\cA(A \to D)$};
\end{tikzpicture}
$$
in $\CstarTwoCat( \cA^3(A \to B \to C \to D) \to \cA(A \to D))$;

\item[(u)] 
For $A,B \in \Ob \cA$, \emph{left and right ``unitor''} unitary adjoint equivalences $\mathbf{l}$ and $\mathbf{r}$ where $l$ and $r$ are $\dag$-2-natural transformations 
$$
\begin{tikzpicture}[baseline={([yshift=0]current bounding box.center)}]
    \draw[thick] (-1,0) circle (0.02);
    \draw[thick] (-1,0) arc (180:0:.5);
    \draw[thick] (0,-.5) -- (0,0);
    \draw[thick] (-.5,0.5) -- (-.5,1);
    \node at (0,-.75) {\scriptsize $\cA(A \to B)$};
    \node at (-.5,1.25) {\scriptsize $\cA(A \to B)$};
\end{tikzpicture}
\;\xRightarrow{l}\;
\begin{tikzpicture}[baseline={([yshift=0]current bounding box.center)}]
    \draw[thick] (-.5,-0.5) -- (-.5,1);
    \node at (-.5,-.75) {\scriptsize $\cA(A \to B)$};
    \node at (-.5,1.25) {\scriptsize $\cA(A \to B)$};
\end{tikzpicture}
\qquad\text{and}\qquad
\begin{tikzpicture}[baseline={([yshift=0]current bounding box.center)}]
    \draw[thick] (0,0) circle (0.02);
    \draw[thick] (-1,-.5) -- (-1,0) arc (180:0:.5);
    \draw[thick] (-.5,0.5) -- (-.5,1);
    \node at (-1,-.75) {\scriptsize $\cA(A \to B)$};
    \node at (-.5,1.25) {\scriptsize $\cA(A \to B)$};
\end{tikzpicture}
\;\xRightarrow{r}\;
\begin{tikzpicture}[baseline={([yshift=0]current bounding box.center)}]
    \draw[thick] (-.5,-0.5) -- (-.5,1);
    \node at (-.5,-.75) {\scriptsize $\cA(A \to B)$};
    \node at (-.5,1.25) {\scriptsize $\cA(A \to B)$};
\end{tikzpicture}
$$
in $\CstarTwoCat(\cA(A \to B) \to \cA(A \to B))$;

\item[($\pi$)] For $A,B,C,D,E \in \Ob \cA$, a \emph{``pentagonator''} unitary modification

$$
\begin{tikzcd}[column sep = -17]
&& \arrow[lld,Rightarrow,blue,"a"'] 
\;
\begin{tikzpicture}[baseline={([yshift=0]current bounding box.center)},scale=.5]
\draw[thick] (0,0) arc (180:0:.5);
\draw[thick] (-.5,0) -- (-.5,.5) arc (180:0:.5);
\draw[thick] (-1,0) -- (-1,1) arc (180:0:.5);
\draw[thick] (-.5,1.5) -- (-.5,2);
\draw[dashed,red,rounded corners = 5] (-1.25,.6) rectangle (.75,2); 
\draw[dashed,blue,rounded corners = 5] (-.75,0) rectangle (1.25,1.25); 
\end{tikzpicture}
\;
\arrow[rrd,Rightarrow,red,"a"]&&\\
\begin{tikzpicture}[baseline={([yshift=0]current bounding box.center)},scale=.5]
\draw[thick] (0,0) arc (180:0:.5);
\draw[thick] (.5,.5) arc (180:0:.5) -- (1.5,0);
\draw[thick] (-.5,0) -- (-.5,1) arc (180:0:.75);
\draw[thick] (.25,1.75) -- (.25,2);
\draw[dashed,blue,rounded corners = 5] (-.75,.6) rectangle (1.75,2); 
\end{tikzpicture}
\arrow[rd,Rightarrow,blue,"a"'] &&\scalebox{1.25}{$\pi \atop \ttto$}&&
\begin{tikzpicture}[baseline={([yshift=0]current bounding box.center)},scale=.5]
\draw[thick] (0,0) -- (0,.5) arc (180:0:.5) -- (1,0);
\draw[thick] (-1.5,0)arc (180:0:.5);
\draw[thick] (-1,.5) -- (-1,1) arc (180:0:.75);
\draw[thick] (-.25,1.75) -- (-.25,2);
\draw[dashed,red,rounded corners = 5] (-1.25,.6) rectangle (1.25,2); 
\end{tikzpicture}
\arrow[ld,Rightarrow,red,"a"]\\
&
\;
\begin{tikzpicture}[baseline={([yshift=0]current bounding box.center)},scale=.5]
\draw[thick] (0,0) arc (180:0:.5);
\draw[thick] (-0.5,0) -- (-.5,.5) arc (180:0:.5);
\draw[thick] (0,1) arc (180:0:.75) -- (1.5,0);
\draw[thick] (.75,1.75) -- (.75,2);
\draw[dashed,blue,rounded corners = 5] (-.75,0) rectangle (1.25,1.25);
\end{tikzpicture}
\;
\arrow[rr,Rightarrow,blue,"a"']&&
\;
\begin{tikzpicture}[baseline={([yshift=0]current bounding box.center)},scale=.5]
\draw[thick] (0,0) arc (180:0:.5);
\draw[thick] (0.5,.5) arc (180:0:.5) -- (1.5,0);
\draw[thick] (1,1) arc (180:0:.5) -- (2,0);
\draw[thick] (1.5,1.5) -- (1.5,2);
\end{tikzpicture}
\;
&
\end{tikzcd}
$$
in $\CstarTwoCat(\cA^4(A \to B \to C \to D \to E) \to \cA(A \to E)).$

\item[(coh)] For $A,B,C \in \Ob \cA$, \emph{middle, left, and right unity coheretor} unitary modification

$$
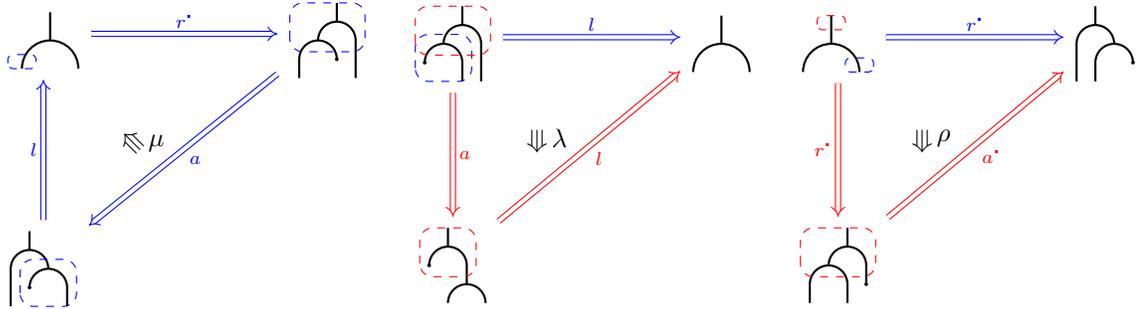
\begin{tikzcd}[column sep = 5,row sep = 10]
\begin{tikzpicture}[baseline={([yshift=0]current bounding box.center)},scale=.75]
\draw[thick] (0,0) -- (0,.5);
\draw[thick] (-.5,-.5) arc (180:0:.5);
\draw[dashed,blue,rounded corners = 2] (-.75,-.5) rectangle (-.25,-.25); 
\end{tikzpicture}
\arrow[rrr,Rightarrow,blue,"r^{\sqdot}"] \arrow[ddd,Leftarrow,blue,"l"'] & && 
\begin{tikzpicture}[baseline={([yshift=0]current bounding box.center)},scale=.5]
\draw[thick] (0,0) -- (0,.5) arc (180:0:.5) -- (1,.5);
\draw[thick] (1,.5) circle (0.02);
\draw[thick] (.5,1) arc (180:0:.5) -- (1.5,0);
\draw[thick] (1,1.5) -- (1,2);
\draw[dashed,blue,rounded corners = 5] (-.25,.7) rectangle (1.75,2); 
\end{tikzpicture}
\\
& \rotatebox{135}{$\ttto$}\mu & \qquad & \\
&&&\\
\begin{tikzpicture}[baseline={([yshift=0]current bounding box.center)},scale=.5]
\draw[thick] (0,0.5) arc (180:0:.5) -- (1,0);
\draw[thick] (0,.5) circle (0.02);
\draw[thick] (.5,1) arc (0:180:.5) -- (-.5,0);
\draw[thick] (0,1.5) -- (0,2);
\draw[dashed,blue,rounded corners = 5] (-.25,0) rectangle (1.25,1.25); 
\end{tikzpicture}
\arrow[rrruuu,Leftarrow,blue,"a"']&& &
\end{tikzcd}
\quad
\begin{tikzcd}[column sep = 5,row sep = 10]
\begin{tikzpicture}[baseline={([yshift=0]current bounding box.center)},scale=.5]
\draw[thick] (0,0.5) arc (180:0:.5) -- (1,0);
\draw[thick] (0,.5) circle (0.02);
\draw[thick] (.5,1) arc (180:0:.5) -- (1.5,0);
\draw[thick] (1,1.5) -- (1,2);
\draw[dashed,blue,rounded corners = 5] (-.25,0) rectangle (1.25,1.25); 
\draw[dashed,red,rounded corners = 5] (-.25,0.7) rectangle (1.75,2); 
\end{tikzpicture} \arrow[rrr,Rightarrow,blue,"l"] \arrow[ddd,Rightarrow,red,"a"] & && 
\begin{tikzpicture}[baseline={([yshift=0]current bounding box.center)},scale=.75]
    \draw[thick] (0,0) -- (0,.5);
    \draw[thick] (-.5,-.5) arc (180:0:.5);
\end{tikzpicture}
\\
& \raisebox{7pt}{\rotatebox{270}{$\ttto$}} \,\lambda & \qquad & \\
&&&\\
\begin{tikzpicture}[baseline={([yshift=0]current bounding box.center)},scale=.5]
\draw[thick] (0,0) -- (0,.5);
\draw[thick] (-.5,-.5) arc (180:0:.5);
\draw[thick] (-1,.5) arc (180:0:.5);
\draw[thick] (-.5,1) -- (-.5,1.5);
\draw[thick] (-1,.5) circle (0.02);
\draw[dashed,red,rounded corners = 5] (-1.25,0.2) rectangle (.25,1.5); 
\end{tikzpicture}
\arrow[rrruuu,Rightarrow,red,"l"']&& &
\end{tikzcd}
\quad
\begin{tikzcd}[column sep = 5,row sep = 10]
\begin{tikzpicture}[baseline={([yshift=0]current bounding box.center)},scale=.75]
\draw[thick] (0,0) -- (0,.5);
\draw[thick] (-.5,-.5) arc (180:0:.5);
\draw[dashed,red,rounded corners = 2] (-.25,.25) rectangle (.25,.5); 
\draw[dashed,blue,rounded corners = 2] (.75,-.5) rectangle (.25,-.25); 
\end{tikzpicture}
\arrow[rrr,Rightarrow,blue,"r^{\sqdot}"] \arrow[ddd,Rightarrow,red,"r^{\sqdot}"'] & && 
\begin{tikzpicture}[baseline={([yshift=0]current bounding box.center)},scale=.5]
\draw[thick] (0,0) -- (0,.5) arc (180:0:.5) -- (1,.5);
\draw[thick] (1,.5) circle (0.02);
\draw[thick] (.5,1) arc (0:180:.5) -- (-.5,0);
\draw[thick] (0,1.5) -- (0,2);
\end{tikzpicture}
\\
& \raisebox{5pt}{\rotatebox{270}{$\ttto$}} \,\rho & \qquad & \\
&&&\\
\begin{tikzpicture}[baseline={([yshift=0]current bounding box.center)},scale=.5]
\draw[thick] (-1.5,-.5) -- (-1.5,0) arc (180:0:.5) -- (-.5,-.5);
\draw[thick] (-1,.5) arc (180:0:.5) -- (0,0);
\draw[thick] (-.5,1) -- (-.5,1.5);
\draw[thick] (0,0) circle (0.02);
\draw[dashed,red,rounded corners = 5] (-1.75,0.2) rectangle (.25,1.5); 
\end{tikzpicture}
\arrow[rrruuu,Rightarrow,red,"a^{\sqdot}"']&& &
\end{tikzcd}
$$
in $\CstarTwoCat(\cA^2(A \to B \to C) \to \cA(A \to C))$.
\end{enumerate}
Furthermore, we ask that data of objects together with underlying hom-2-categories, functors, adjoint equivalences, and invertible modifications form a 3-category (algebraic tricategory) in the sense of \cite[Definition 4.1]{gurski_2013}. 
More explicitly, these include a 3-dimensional associahedron axiom for each tuple of 5 objects in $\cA$ (non-abelian 4-cocycle condition) and two axioms for each triplet of objects in $\cA$ relating the unity coheretors with the associator and pentagonators (normalized cocycle conditions). 
\end{remark}

\begin{definition}
A W*-3-category $\cA$ is a C*-3-category such that:
\begin{itemize}
    \item Each $\cA(A \to B)$ is a W*-2-category;
    \item The 1-composition $\dag$-2-functor $\cccirc$ is separately normal, hence extending to a normal $\dag$-2-functor
    $$\cccirc\colon \cA(A \to B) \Wblacktimes \cA(B \to C) \to \cA(A \to C).$$
\end{itemize}
Note that the unit $\dag$-2-functor $I_A\colon \cA(A \to A)$ is automatically normal. Using the fact that 2-composition $\ccirc$ and 3-composition $\circ$ are separately normal on each hom-W*-2-category, we may extend the constraint unitary adjoint equivalences so that they all occur in $(\WstarTwoCat, \Wmaxtimes)$.
\end{definition}

\begin{definition}
A $\dag$-3-functor $F \colon \cA \to \cA'$ between C*-3-categories consists of an underlying 3-functor in the sense of \cite{gurski_2013}, such that $F$ is locally a $\dag$-2-functor and the underlying coherence 2-natural transformations and modifications are $\dag$-natural transformations and unitary modifications respectively.
\end{definition}

\begin{remark}
We unpack the data of a $\dag$-3-functor as follows:
\begin{itemize}
    \item[(0)] A function from objects of $\cA$ to objects of $\cA'$;
    \item[(hom)] For $A,B \in \cA$, a $\dag$-2-functor $F \colon \cA(A,B) \to \cA'(FA,FB)$, which we represent diagrammatically by:
$$
\begin{tikzpicture}[baseline={([yshift=0]current bounding box.center)}]
\draw[thick] (0,-.5) -- (0,.5);
\draw[thick,fill=white,rounded corners=3] (-.2,-.2) rectangle (0.2,0.2);
\node at (0,0) {\scriptsize $F$};
\node at (0,-.7) {\scriptsize $\cA(A,B)$};
\node at (0,.7) {\scriptsize $\cA'(FA,FB)$};
\end{tikzpicture}
$$
\item[($\cccirc$)] A unitary adjoint equivalence $\mathbf{F^2}$ where $F^2$ is a $\dag$-2-natural transformation
$$
\begin{tikzpicture}[baseline={([yshift=0]current bounding box.center)}]
\draw[thick] (0,-.5) -- (0,0) arc (180:0:.5) -- (1,-.5);
\draw[thick] (.5,.5) -- (.5,1);
\draw[thick,fill=white,rounded corners=3] (-.2,-.2) rectangle (0.2,0.2);
\node at (0,0) {\scriptsize $F$};
\draw[thick,fill=white,rounded corners=3] (.8,-.2) rectangle (1.2,0.2);
\node at (1,0) {\scriptsize $F$};
\node at (0,-.7) {\scriptsize $\cA(A,B)$};
\node at (1,-.7) {\scriptsize $\cA(B,C)$};
\node at (0.5,1.2) {\scriptsize $\cA'(FA,FC)$};
\end{tikzpicture}
\;\xRightarrow{F^2}\;
\begin{tikzpicture}[baseline={([yshift=0]current bounding box.center)}]
\draw[thick] (0,-.5) -- (0,.5);
\draw[thick,fill=white,rounded corners=3] (-.2,-.2) rectangle (0.2,0.2);
\node at (0,0) {\scriptsize $F$};
\node at (-.5,-1.2) {\scriptsize $\cA(A,B)$};
\node at (.5,-1.2) {\scriptsize $\cA(B,C)$};
\node at (0,.7) {\scriptsize $\cA'(FA,FC)$};
\draw[thick] (-.5,-1) arc (180:0:.5);
\end{tikzpicture}
$$
in $\CstarTwoCat(\cA(A \to B \to C) \to \cA'(FA \to FC))$.
\item[(I)]
A unitary adjoint equivalence $\mathbf{F^0}$ where $F^0$ is a $\dag$-2-natural transformation
$$
\begin{tikzpicture}[baseline={([yshift=0]current bounding box.center)}]
\draw[thick] (0,0.5) -- (0,1);
\draw[thick] (0,0.5) circle (.02);
\draw[thick,white] (0,0) circle (.02);
\node at (0,1.25) {\scriptsize $\cA'(FA \to FA)$};
\end{tikzpicture}
\;\xRightarrow{F^0}\;
\begin{tikzpicture}[baseline={([yshift=0]current bounding box.center)}]
\draw[thick] (0,0) -- (0,1);
\draw[thick] (0,0) circle (.02);
\node at (0,1.25) {\scriptsize $\cA'(FA \to FA)$};
\draw[thick,fill=white,rounded corners=3] (-.2,.3) rectangle (0.2,0.7);
\node at (0,0.5) {\scriptsize $F$};
\end{tikzpicture}
$$
in $\CstarTwoCat(\bbC \to \cA'(FA \to FA))$.
\item[(a)] A unitary modification $F^a$
$$
\begin{tikzcd}[row sep = -2.5]
&  \arrow[ld,Rightarrow,blue,"F^2"'] 
\begin{tikzpicture}[baseline={([yshift=0]current bounding box.center)},scale=.7]
\draw[thick] (-.5,-.5) arc (180:0:.5);
\draw[thick] (-1,-.5) -- (-1,0) arc (180:0:.5);
\draw[thick] (-.5,0.5) -- (-.5,1);
\draw[thick] (-1,-1) -- (-1,-.5);
\draw[thick] (-.5,-1) -- (-.5,-.5);
\draw[thick] (.5,-1) -- (.5,-.5);
\draw[thick,fill=white,rounded corners=3] (-.2+.5,-.2-.5) rectangle (0.2+.5,0.2-.5);
\node at (0+.5,0-.1-.5) {\scriptsize $F$};
\draw[thick,fill=white,rounded corners=3] (-.2-.5,-.2-.5) rectangle (0.2-.5,0.2-.5);
\node at (0-.5,0-.1-.5) {\scriptsize $F$};
\draw[thick,fill=white,rounded corners=3] (-.2-1,-.2-.5) rectangle (0.2-1,0.2-.5);
\node at (0-1,0-.1-.5) {\scriptsize $F$};
\draw[thick,dashed,blue,rounded corners=5] (-.75,-1) rectangle (.75,0.25);
\draw[thick,dashed,red,rounded corners=5] (-1.25,-.25) rectangle (.75,1);
\end{tikzpicture}
\arrow[rd,Rightarrow,red,"a'"] & \\
\begin{tikzpicture}[baseline={([yshift=0]current bounding box.center)},scale=.7]
\draw[thick] (-.5,-1) arc (180:0:.5);
\draw[thick] (-1,-1) -- (-1,0) arc (180:0:.5);
\draw[thick] (-.5,0.5) -- (-.5,1);
\draw[thick] (0,0) -- (0,-.5);
\draw[thick,fill=white,rounded corners=3] (-.2,-.2) rectangle (0.2,0.2);
\node at (0,0-.1) {\scriptsize $F$};
\draw[thick,fill=white,rounded corners=3] (-.2-1,-.2) rectangle (0.2-1,0.2);
\node at (0-1,0-.1) {\scriptsize $F$};
\draw[thick,dashed,blue,rounded corners=5] (-1.25,-0.3) rectangle (.25,1);
\end{tikzpicture}
\arrow[dd,Rightarrow,blue,"F^2"'] && 
\begin{tikzpicture}[baseline={([yshift=0]current bounding box.center)},scale=.7]
\draw[thick] (0,0) -- (0,.5);
\draw[thick] (-1,-1) arc (180:0:.5);
\draw[thick] (-.5,-.5) arc (180:0:.5) -- (.5,-1);
\draw[thick] (-1,-1) -- (-1,-1.5);
\draw[thick] (0,-1) -- (0,-1.5);
\draw[thick] (.5,-1) -- (.5,-1.5);
\draw[thick,fill=white,rounded corners=3] (-.2+.5,-.2-1) rectangle (0.2+.5,0.2-1);
\node at (0+.5,0-.1-1) {\scriptsize $F$};
\draw[thick,fill=white,rounded corners=3] (-.2,-.2-1) rectangle (0.2,0.2-1);
\node at (0,0-.1-1) {\scriptsize $F$};
\draw[thick,fill=white,rounded corners=3] (-.2-1,-.2-1) rectangle (0.2-1,0.2-1);
\node at (0-1,0-.1-1) {\scriptsize $F$};
\draw[thick,dashed,red,rounded corners=5] (-1.25,-1.5) rectangle (.25,-0.25);
\end{tikzpicture}
\arrow[dd,Rightarrow,red,"F^2"] \\
& {F^a \atop \ttto}&  \\
\begin{tikzpicture}[baseline={([yshift=0]current bounding box.center)},scale=.7]
\draw[thick] (-.5,-1) arc (180:0:.5);
\draw[thick] (-1,-1) -- (-1,-.5) arc (180:0:.5);
\draw[thick] (-.5,0) -- (-.5,1);
\draw[thick,fill=white,rounded corners=3] (-.2-.5,-.2+.5) rectangle (0.2-.5,0.2+.5);
\node at (0-.5,0-.1+.5) {\scriptsize $F$};
\draw[thick,dashed,blue,rounded corners=5] (-1.25,-1) rectangle (.75,.2);
\end{tikzpicture}
\arrow[rd,Rightarrow,blue,"a"] && 
\begin{tikzpicture}[baseline={([yshift=0]current bounding box.center)},scale=.7]
\draw[thick] (0,0) -- (0,.5);
\draw[thick] (-1,-1.5) arc (180:0:.5);
\draw[thick] (-.5,-1) -- (-.5,-.5) arc (180:0:.5) -- (.5,-1);
\draw[thick] (.5,-1) -- (.5,-1.5);
\draw[thick,fill=white,rounded corners=3] (-.2+.5,-.2-.5) rectangle (0.2+.5,0.2-.5);
\node at (0+.5,0-.1-.5) {\scriptsize $F$};
\draw[thick,fill=white,rounded corners=3] (-.2-.5,-.2-.5) rectangle (0.2-.5,0.2-.5);
\node at (0-.5,0-.1-.5) {\scriptsize $F$};
\draw[thick,dashed,red,rounded corners=5] (-.75,-.75) rectangle (.75,.5);
\end{tikzpicture}
\arrow[ld,Rightarrow,red,"F^2"'] \\
&  
\begin{tikzpicture}[baseline={([yshift=0]current bounding box.center)},scale=.7]
\draw[thick] (0,0) -- (0,1);
\draw[thick] (-1,-1) arc (180:0:.5);
\draw[thick] (-.5,-.5) arc (180:0:.5) -- (.5,-1);
\draw[thick,fill=white,rounded corners=3] (-.2,-.2+.5) rectangle (0.2,0.2+.5);
\node at (0,0-.1+.5) {\scriptsize $F$};
\end{tikzpicture}
& 
\end{tikzcd}
$$
in $\CstarTwoCat(\cA(A \to B \to C \to D) \to \cA'(FA \to FD))$.
\item[(u)] Unitary modifications $F^l$ and $F^r$
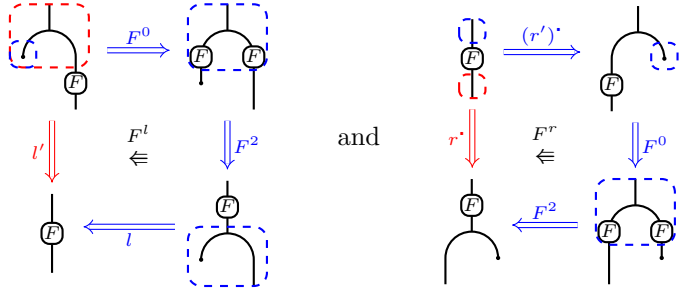
$$
\begin{tikzcd}[column sep = 2, row sep = -2]
\begin{tikzpicture}[baseline={([yshift=0]current bounding box.center)},scale=.7]
\draw[thick] (0,0) arc (180:0:.5) -- (1,-1);
\draw[thick] (0,0) circle (0.02);
\draw[thick] (.5,.5) -- (.5,1);
\draw[thick,fill=white,rounded corners=3] (1-.2,-.2-.5) rectangle (1.2,0.2-.5);
\node at (1,0-.1-.5) {\scriptsize $F$};
\draw[thick,dashed,red,rounded corners=5] (-.25,-.2) rectangle (1.25,1);
\draw[thick,dashed,blue,rounded corners=3] (-.25,-.2) rectangle (.25,.3);
\end{tikzpicture}
\arrow[rr,Rightarrow,blue,"F^0"] \arrow[dd,Rightarrow,red,"l'"'] & & 
\begin{tikzpicture}[baseline={([yshift=0]current bounding box.center)},scale=.7]
\draw[thick] (0,-.5) -- (0,0) arc (180:0:.5) -- (1,-1);
\draw[thick] (0,-.5) circle (0.02);
\draw[thick] (.5,.5) -- (.5,1);
\draw[thick,fill=white,rounded corners=3] (1-.2,-.2) rectangle (1.2,0.2);
\node at (1,0-.1) {\scriptsize $F$};
\draw[thick,fill=white,rounded corners=3] (-.2,-.2) rectangle (.2,0.2);
\node at (0,0-.1) {\scriptsize $F$};
\draw[thick,dashed,blue,rounded corners=5] (-.25,-.25) rectangle (1.25,1);
\end{tikzpicture}
\arrow[dd,Rightarrow,blue,"F^2"] \\
& 
{F^l \atop \Lleftarrow} & \\
\begin{tikzpicture}[baseline={([yshift=0]current bounding box.center)},scale=.7]
\draw[thick] (.5,-.25) -- (.5,1.25);
\draw[thick,fill=white,rounded corners=3] (.5-.2,-.2+.5) rectangle (.7,0.2+.5);
\node at (0.5,0.5-.1) {\scriptsize $F$};
\end{tikzpicture}
& & \arrow[ll,Rightarrow,blue,"l"] 
\begin{tikzpicture}[baseline={([yshift=0]current bounding box.center)},scale=.7]
\draw[thick] (0,-.5) arc (180:0:.5) -- (1,-1);
\draw[thick] (0,-.5) circle (0.02);
\draw[thick] (.5,0) -- (.5,1);
\draw[thick,fill=white,rounded corners=3] (.5-.2,-.2+.5) rectangle (.7,0.2+.5);
\node at (0.5,0.5-.1) {\scriptsize $F$};
\draw[thick,dashed,blue,rounded corners=5] (-.25,-1) rectangle (1.25,.125);
\end{tikzpicture}
\end{tikzcd}
\qquad\text{and}\qquad
\begin{tikzcd}[column sep = 2, row sep = -2]
\begin{tikzpicture}[baseline={([yshift=0]current bounding box.center)},scale=.7]
\draw[thick] (.5,-.25) -- (.5,1.25);
\draw[thick,fill=white,rounded corners=3] (.5-.2,-.2+.5) rectangle (.7,0.2+.5);
\node at (0.5,0.5-.1) {\scriptsize $F$};
\draw[thick,dashed,red,rounded corners=3] (.5-.25,-.25) rectangle (.5+.25,.2);
\draw[thick,dashed,blue,rounded corners=3] (.5-.25,.8) rectangle (.5+.25,1.25);
\end{tikzpicture}
\arrow[rr,Rightarrow,blue,"(r')^{\sqdot}"] \arrow[dd,Rightarrow,red,"r^{\sqdot}"'] & & 
\begin{tikzpicture}[baseline={([yshift=0]current bounding box.center)},scale=.7]
\draw[thick] (0,-1) -- (0,0) arc (180:0:.5);
\draw[thick] (1,0) circle (0.02);
\draw[thick] (.5,.5) -- (.5,1);
\draw[thick,fill=white,rounded corners=3] (-.2,-.2-.5) rectangle (.2,0.2-.5);
\node at (0,0-.1-.5) {\scriptsize $F$};
\draw[thick,dashed,blue,rounded corners=3] (1-.25,-.2) rectangle (1.25,.3);
\end{tikzpicture}
\arrow[dd,Rightarrow,blue,"F^0"] \\
& {F^r \atop \Lleftarrow} & \\
\begin{tikzpicture}[baseline={([yshift=0]current bounding box.center)},scale=.7]
\draw[thick] (0,-1) -- (0,-.5) arc (180:0:.5);
\draw[thick] (1,-.5) circle (0.02);
\draw[thick] (.5,0) -- (.5,1);
\draw[thick,fill=white,rounded corners=3] (.5-.2,-.2+.5) rectangle (.7,0.2+.5);
\node at (0.5,0.5-.1) {\scriptsize $F$};
\end{tikzpicture}
& & \arrow[ll,Rightarrow,blue,"F^2"'] 
\begin{tikzpicture}[baseline={([yshift=0]current bounding box.center)},scale=.7]
\draw[thick] (0,-1) -- (0,0) arc (180:0:.5) -- (1,-.5);
\draw[thick] (1,-.5) circle (0.02);
\draw[thick] (.5,.5) -- (.5,1);
\draw[thick,fill=white,rounded corners=3] (1-.2,-.2) rectangle (1.2,0.2);
\node at (1,0-.1) {\scriptsize $F$};
\draw[thick,fill=white,rounded corners=3] (-.2,-.2) rectangle (.2,0.2);
\node at (0,0-.1) {\scriptsize $F$};
\draw[thick,dashed,blue,rounded corners=5] (-.25,-.25) rectangle (1.25,1);
\end{tikzpicture}
\end{tikzcd}
$$
in $\CstarTwoCat(\cA(A \to B) \to \cA'(FA\to FB))$.
\end{itemize}
Furthermore, we ask that the data of the underlying $2$-functors, adjoint equivalences, and invertible modifications form a 3-functor in the sense of \cite[Definition 4.10]{gurski_2013}. More explicitly, these include an associativity axiom relating the pentagonators in $\cA$ and $\cA'$, and a unitality axiom relating the middle unity coheretors in $\cA$ and $\cA'$. 
\end{remark}

\begin{definition}
A $\dag$-3-natural transformation $\theta \colon F \tto F'$ between $\dag$-3-functors $F,F' \colon \cA \to \cA'$ consists of an underlying 3-natural transformation in the sense of \cite{gurski_2013} such that the underlying coherence 2-natural transformations and modifications are $\dag$-2-natural transformations and unitary modifications respectively.   
\end{definition}

\begin{remark}
We unpack the data of a $\dag$-3-natural transformation as follows:
\begin{itemize}
\item[(0)] A family of 1-morphisms $\theta_A \in \cA'(FA \to F'A)$ indexed by the objects of $\cA$, which we represent diagrammatically by
$$
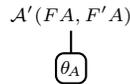
\begin{tikzpicture}[baseline={([yshift=0]current bounding box.center)}]
\draw[thick] (0,0) -- (0,.5);
\draw[thick,fill=white,rounded corners=3] (-.2,-.2) rectangle (0.2,0.2);
\node at (0,0) {\scriptsize $\theta_{\! A}$};
\node at (0,.7) {\scriptsize $\cA'(FA,F'A)$};
\end{tikzpicture}
$$
\item[(hom)] A unitary adjoint equivalence {\boldmath$\mathbf{\theta}$} where $\theta$ is a $\dag$-2-natural transformation
$$
\begin{tikzpicture}[baseline={([yshift=0]current bounding box.center)}]
\draw[thick] (0,0) arc (0:180:0.5) -- (-1,-.5);
\draw[thick] (-.5,.5) -- (-.5,1);
\draw[thick,fill=white,rounded corners=3] (-.2,-.2) rectangle (0.2,0.2);
\node at (0,0) {\scriptsize $\theta_{\! A'}$};
\draw[thick,fill=white,rounded corners=3] (-1.2,-.2) rectangle (-1+0.2,0.2);
\node at (-1,0) {\scriptsize $F$};
\node at (-.5,1.2) {\scriptsize $\cA'(FA,F'A')$};
\node at (-1,-.7) {\scriptsize $\cA(A,A')$};
\end{tikzpicture}
\quad\xRightarrow{\theta}\quad 
\begin{tikzpicture}[baseline={([yshift=0]current bounding box.center)}]
\draw[thick] (0,-.5) -- (0,0) arc (0:180:0.5);
\draw[thick] (-.5,.5) -- (-.5,1);
\draw[thick,fill=white,rounded corners=3] (-1.2,-.2) rectangle (-1+0.2,0.2);
\node at (-1,0) {\scriptsize $\theta_A$};
\draw[thick,fill=white,rounded corners=3] (-.2,-.2) rectangle (0.2,0.2);
\node at (0,0) {\scriptsize $F'$};
\node at (-.5,1.2) {\scriptsize $\cA'(FA,F'A')$};
\node at (0,-.7) {\scriptsize $\cA(A,A')$};
\end{tikzpicture}
$$
in $\CstarTwoCat(\cA(A,A') \to \cA'(FA, F'A'))$.
\item[($\cccirc$)] A unitary modification $\theta^2$
$$
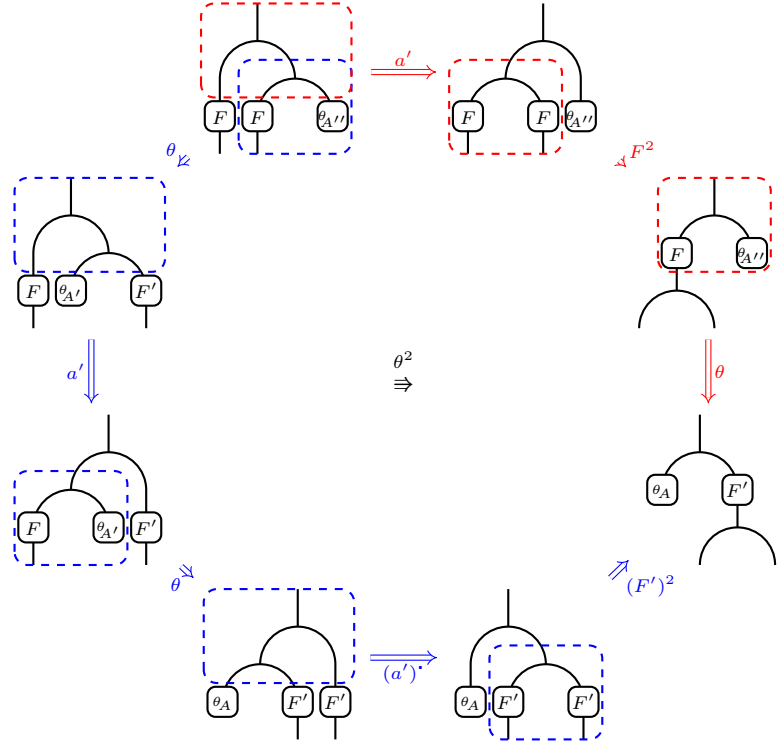
\begin{tikzcd}[column sep = 3,row sep = 1]
& 
\begin{tikzpicture}[baseline={([yshift=0]current bounding box.center)},scale=1]
\draw[thick] (0,0) arc (0:180:0.5) -- (-1,-.5);
\draw[thick] (-.5,.5) arc (0:180:.5) -- (-1.5,-.5);
\draw[thick] (-1,1) -- (-1,1.5);
\draw[thick,fill=white,rounded corners=3] (-.2,-.2) rectangle (0.2,0.2);
\node at (0,0-.05) {\tiny $\theta_{ \tiny \! \! A''}$};
\draw[thick,fill=white,rounded corners=3] (-1.2,-.2) rectangle (-1+0.2,0.2);
\node at (-1,0-.1) {\scriptsize $F$};
\draw[thick,fill=white,rounded corners=3] (-1.2-.5,-.2) rectangle (-1+0.2-.5,0.2);
\node at (-1-.5,0-.1) {\scriptsize $F$};
\draw[thick,dashed,blue,rounded corners=5] (-1.25,-.5) rectangle (.25,.75);
\draw[thick,dashed,red,rounded corners=5] (-1.75,0.25) rectangle (.25,1.5);
\end{tikzpicture}
&& 
\begin{tikzpicture}[baseline={([yshift=0]current bounding box.center)},scale=1]
\draw[thick] (0,-.5) -- (0,0) arc (0:180:0.5) -- (-1,-.5);
\draw[thick] (-.5,.5) arc (180:0:.5) -- (.5,0);
\draw[thick] (0,1) -- (0,1.5);
\draw[thick,fill=white,rounded corners=3] (-.2,-.2) rectangle (0.2,0.2);
\node at (0,0-.1) {\scriptsize $F$};
\draw[thick,fill=white,rounded corners=3] (-1.2,-.2) rectangle (-1+0.2,0.2);
\node at (-1,0-.1) {\scriptsize $F$};
\draw[thick,fill=white,rounded corners=3] (.3,-.2) rectangle (.7,0.2);
\node at (.5,0-.05) {\tiny $\theta_{\tiny \! \! A''}$};
\draw[thick,dashed,red,rounded corners=5] (-1.25,-.5) rectangle (.25,.75);
\end{tikzpicture}
\\
\begin{tikzpicture}[baseline={([yshift=0]current bounding box.center)},scale=1]
\draw[thick] (0,-.5) -- (0,0) arc (0:180:0.5);
\draw[thick] (-.5,.5) arc (0:180:.5) -- (-1.5,-.5);
\draw[thick] (-1,1) -- (-1,1.5);
\draw[thick,fill=white,rounded corners=3] (-.2,-.2) rectangle (0.2,0.2);
\node at (0,0-.1) {\scriptsize $F'$};
\draw[thick,fill=white,rounded corners=3] (-1.2,-.2) rectangle (-1+0.2,0.2);
\node at (-1,0-.05) {\tiny $\theta_{\tiny \!\! A'}$};
\draw[thick,fill=white,rounded corners=3] (-1.2-.5,-.2) rectangle (-1+0.2-.5,0.2);
\node at (-1-.5,0-.1) {\scriptsize $F$};
\draw[thick,dashed,blue,rounded corners=5] (-1.75,0.25) rectangle (.25,1.5);
\end{tikzpicture}
&&&& 
\begin{tikzpicture}[baseline={([yshift=0]current bounding box.center)},scale=1]
\draw[thick] (-1,-.5) arc (180:0:.5);
\draw[thick] (-.5,0) -- (-.5,.5) arc (180:0:.5);
\draw[thick] (0,1) -- (0,1.5);
\draw[thick,fill=white,rounded corners=3] (-.2-.5,-.2+.5) rectangle (0.2-.5,0.2+.5);
\node at (-0.5,0-.1+.5) {\scriptsize $F$};
\draw[thick,fill=white,rounded corners=3] (.3,-.2+.5) rectangle (.7,0.2+.5);
\node at (.5,0-.05+.5) {\tiny $\theta_{\tiny \! \! A''}$};
\draw[thick,dashed,red,rounded corners=5] (-.75,0.25) rectangle (.75,1.5);
\end{tikzpicture}
\\
&& {\theta^2 \atop \ttto} \\
\begin{tikzpicture}[baseline={([yshift=0]current bounding box.center)},scale=1]
\draw[thick] (0,0) arc (0:180:0.5) -- (-1,-.5);
\draw[thick] (-.5,.5) arc (180:0:.5) -- (.5,-.5);
\draw[thick] (0,1) -- (0,1.5);
\draw[thick,fill=white,rounded corners=3] (-.2,-.2) rectangle (0.2,0.2);
\node at (0,0-.05) {\tiny $\theta_{\tiny \! \! A'}$};
\draw[thick,fill=white,rounded corners=3] (-1.2,-.2) rectangle (-1+0.2,0.2);
\node at (-1,0-.1) {\scriptsize $F$};
\draw[thick,fill=white,rounded corners=3] (.3,-.2) rectangle (.7,0.2);
\node at (.5,0-.1) {\scriptsize $F'$};
\draw[thick,dashed,blue,rounded corners=5] (-1.25,-.5) rectangle (.25,.75);
\end{tikzpicture}
&&&& 
\begin{tikzpicture}[baseline={([yshift=0]current bounding box.center)},scale=1]
\draw[thick] (-1,-.5) arc (180:0:.5);
\draw[thick] (-.5,0) -- (-.5,.5) arc (0:180:.5);
\draw[thick] (-1,1) -- (-1,1.5);
\draw[thick,fill=white,rounded corners=3] (-.2-.5,-.2+.5) rectangle (0.2-.5,0.2+.5);
\node at (-0.5,0-.1+.5) {\scriptsize $F'$};
\draw[thick,fill=white,rounded corners=3] (-1.2-.5,-.2+.5) rectangle (-1+0.2-.5,0.2+.5);
\node at (-1-.5,0-.05+.5) {\tiny $\theta_{\tiny \! A}$};
\end{tikzpicture}
\\
& 
\begin{tikzpicture}[baseline={([yshift=0]current bounding box.center)},scale=1]
\draw[thick] (0,-.5) -- (0,0) arc (0:180:0.5);
\draw[thick] (-.5,.5) arc (180:0:.5) -- (.5,-.5);
\draw[thick] (0,1) -- (0,1.5);
\draw[thick,fill=white,rounded corners=3] (-.2,-.2) rectangle (0.2,0.2);
\node at (0,0-.1) {\scriptsize $F'$};
\draw[thick,fill=white,rounded corners=3] (-1.2,-.2) rectangle (-1+0.2,0.2);
\node at (-1,0-.05) {\tiny $\theta_{\! A}$};
\draw[thick,fill=white,rounded corners=3] (.3,-.2) rectangle (.7,0.2);
\node at (.5,0-.1) {\scriptsize $F'$};
\draw[thick,dashed,blue,rounded corners=5] (-1.75+.5,0.25) rectangle (.25+.5,1.5);
\end{tikzpicture}
&& 
\begin{tikzpicture}[baseline={([yshift=0]current bounding box.center)},scale=1]
\draw[thick] (0,-.5) -- (0,0) arc (0:180:0.5) -- (-1,-.5);
\draw[thick] (-.5,.5) arc (0:180:.5) -- (-1.5,0);
\draw[thick] (-1,1) -- (-1,1.5);
\draw[thick,fill=white,rounded corners=3] (-.2,-.2) rectangle (0.2,0.2);
\node at (0,0-.1) {\scriptsize $F'$};
\draw[thick,fill=white,rounded corners=3] (-1.2,-.2) rectangle (-1+0.2,0.2);
\node at (-1,0-.1) {\scriptsize $F'$};
\draw[thick,fill=white,rounded corners=3] (-1.2-.5,-.2) rectangle (-1+0.2-.5,0.2);
\node at (-1-.5,0-.05) {\tiny $\theta_{\! A}$};
\draw[thick,dashed,blue,rounded corners=5] (-1.25,-.5) rectangle (.25,.75);
\end{tikzpicture}
\arrow[Rightarrow, from=1-2, to=2-1,blue,"\theta"']
\arrow[Rightarrow, from=2-1, to=4-1,blue,"a'"']
\arrow[Rightarrow, from=4-1, to=5-2,blue,"\theta"']
\arrow[Rightarrow, from=5-2, to=5-4,blue,"(a')^{\sqdot}"']
\arrow[Rightarrow, from=5-4, to=4-5,blue,"(F')^2"']
\arrow[Rightarrow, from=1-2, to=1-4,red,"a'"]
\arrow[Rightarrow, from=1-4, to=2-5,red,"F^2"]
\arrow[Rightarrow, from=2-5, to=4-5,red,"\theta"]
\end{tikzcd}
$$
in $\CstarTwoCat(\cA(A \to A' \to A'') \to \cA'(FA \to F'A'')).$
\item[(I)] A unitary modification $\theta^0$
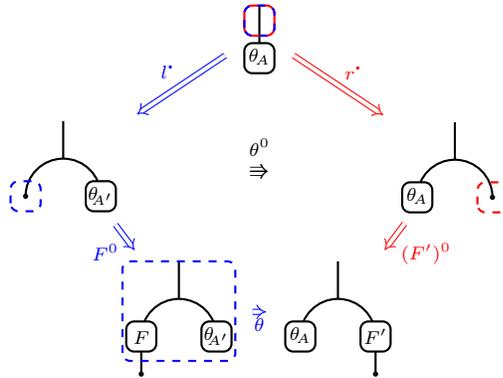
\[\begin{tikzcd}[column sep = -10, row sep = 10]
&& 
\begin{tikzpicture}[baseline={([yshift=0]current bounding box.center)}]
\draw[thick] (0,0) -- (0,.7);
\draw[thick,fill=white,rounded corners=3] (-.2,-.2) rectangle (0.2,0.2);
\node at (0,0-.05) {\scriptsize $\theta_{\!A}$};
\draw[thick,red, rounded corners=3] (-.2,-.2+.5) rectangle (0.2,0.2+.5);
\draw[thick,blue,dashed,rounded corners=3] (-.2,-.2+.5) rectangle (0.2,0.2+.5);
\end{tikzpicture}\\
\begin{tikzpicture}[baseline={([yshift=0]current bounding box.center)}]
\draw[thick] (0,0) arc (0:180:0.5);
\draw[thick] (-.5,.5) -- (-.5,1);
\draw[thick,fill=white,rounded corners=3] (-.2,-.2) rectangle (0.2,0.2);
\node at (0,0-.05) {\scriptsize $\theta_{\!\! A'}$};
\draw[thick] (-1,0) circle (0.02);
\draw[thick,blue,dashed,rounded corners=3] (-.2-1,-.2) rectangle (0.2-1,0.2);
\end{tikzpicture}
&& {\theta^0 \atop \ttto} && 
\begin{tikzpicture}[baseline={([yshift=0]current bounding box.center)}]
\draw[thick] (0,0) arc (0:180:0.5);
\draw[thick] (-.5,.5) -- (-.5,1);
\draw[thick,fill=white,rounded corners=3] (-1.2,-.2) rectangle (-1+0.2,0.2);
\node at (-1,0-.05) {\scriptsize $\theta_{\! A}$};
\draw[thick] (0,0) circle (0.02);
\draw[thick,red,dashed,rounded corners=3] (-.2,-.2) rectangle (0.2,0.2);
\end{tikzpicture} \\
& 
\begin{tikzpicture}[baseline={([yshift=0]current bounding box.center)}]
\draw[thick] (0,0) arc (0:180:0.5) -- (-1,-.5);
\draw[thick] (-.5,.5) -- (-.5,1);
\draw[thick,fill=white,rounded corners=3] (-.2,-.2) rectangle (0.2,0.2);
\node at (0,0-.05) {\scriptsize $\theta_{\!\! A'}$};
\draw[thick,fill=white,rounded corners=3] (-1.2,-.2) rectangle (-1+0.2,0.2);
\node at (-1,0-.1) {\scriptsize $F$};
\draw[thick] (-1,-.5) circle (0.02);
\draw[thick,blue,dashed,rounded corners=3] (-.2-1.05,-.2-.125) rectangle (0.25,1);
\end{tikzpicture}
&& 
\begin{tikzpicture}[baseline={([yshift=0]current bounding box.center)}]
\draw[thick] (0,-.5) -- (0,0) arc (0:180:0.5);
\draw[thick] (-.5,.5) -- (-.5,1);
\draw[thick,fill=white,rounded corners=3] (-1.2,-.2) rectangle (-1+0.2,0.2);
\node at (-1,0-.05) {\scriptsize $\theta_{\!A}$};
\draw[thick,fill=white,rounded corners=3] (-.2,-.2) rectangle (0.2,0.2);
\node at (0,0-.1) {\scriptsize $F'$};
\draw[thick] (0,-.5) circle (0.02);
\end{tikzpicture}
\\
& {} && {}
	\arrow[Rightarrow, from=1-3, to=2-1,blue,"l^{\sqdot}"']
	\arrow[Rightarrow, from=2-1, to=3-2,blue,"F^0"']
	\arrow[Rightarrow, from=3-2, to=3-4,blue,"\theta"']
	\arrow[Rightarrow, from=1-3, to=2-5,red,"r^{\sqdot}"]
	\arrow[Rightarrow, from=2-5, to=3-4,red,"(F')^0"]
\end{tikzcd}\]
in $\CstarTwoCat(\bbC \to \cA'(FA \to F'A)).$
\end{itemize}
Furthermore, we ask that the data of the underlying family of 1-morphisms, adjoint equivalences, and invertible modifications form a 3-natural transformation in the sense of \cite[Definition 4.16]{gurski_2013}. More explicitly, these include an associativity axiom relating $F^a$ and $(F')^a$, and two unitality axiom relating $F^l$, $(F')^l$ and $F^r$, $(F')^r$. 
\end{remark}

\begin{definition}
A $\dag$-3-modification $m \colon \theta \ttto \theta'$ between $\dag$-3-natural transformations $\theta,\theta' \colon F \tto F'$ consists of an underlying 3-modification in the sense of \cite{gurski_2013} such that the underlying coherence modifications are unitary.   
\end{definition}

\begin{remark}
We unpack the data of a $\dag$-3-modification as follows:
\begin{itemize}
\item[(0)] A family of 2-morphisms $\theta_A \in \cA'(\theta_A \tto \theta'_A)$ indexed by the objects of $\cA$, which we represent diagrammatically by
$$
\begin{tikzpicture}
\draw[thick] (0,0) -- (0,.5);
\draw[thick,fill=white,rounded corners=3] (-.2,-.2) rectangle (0.2,0.2);
\node at (0,0) {\scriptsize $\theta_{\! A}$};
\node at (0,.7) {\scriptsize $\cA'(FA,F'A)$};
\end{tikzpicture}
\;\xRightarrow{m_A}\;
\begin{tikzpicture}
\draw[thick] (0,0) -- (0,.5);
\draw[thick,fill=white,rounded corners=3] (-.2,-.2) rectangle (0.2,0.2);
\node at (0,0) {\scriptsize $\theta'_{\! A}$};
\node at (0,.7) {\scriptsize $\cA'(FA,F'A)$};
\end{tikzpicture}
$$
\item[(hom)] A unitary modification $m$
$$
\begin{tikzcd}[column sep = 2, row sep = 2]
\begin{tikzpicture}[baseline={([yshift=0]current bounding box.center)}]
\draw[thick] (0,0) arc (0:180:0.5) -- (-1,-.5);
\draw[thick] (-.5,.5) -- (-.5,1);
\draw[thick,fill=white,rounded corners=3] (-.2,-.2) rectangle (0.2,0.2);
\draw[thick,red,dashed,rounded corners=3] (-.25,-.25) rectangle (0.25,0.25);
\node at (0,0-.05) {\scriptsize $\theta_{\!\! A'}$};
\draw[thick,fill=white,rounded corners=3] (-1.2,-.2) rectangle (-1+0.2,0.2);
\node at (-1,0-.1) {\scriptsize $F$};
\draw[thick,blue,dashed,rounded corners=3] (-1.3,-.5) rectangle (0.3,1);
\end{tikzpicture}
\arrow[rr,Rightarrow,red,"m_{A'}"] \arrow[dd,Rightarrow,blue,"\theta"'] & & 
\begin{tikzpicture}[baseline={([yshift=0]current bounding box.center)}]
\draw[thick] (0,0) arc (0:180:0.5) -- (-1,-.5);
\draw[thick] (-.5,.5) -- (-.5,1);
\draw[thick,fill=white,rounded corners=3] (-.2,-.2) rectangle (0.2,0.2);
\node at (0,0-.05) {\scriptsize $\theta'_{\!\! A'}$};
\draw[thick,fill=white,rounded corners=3] (-1.2,-.2) rectangle (-1+0.2,0.2);
\node at (-1,0-.1) {\scriptsize $F$};
\draw[thick,red,dashed,rounded corners=3] (-1.3,-.5) rectangle (0.3,1);
\end{tikzpicture}
\arrow[dd,Rightarrow,red,"\theta'"] \\
& {m \atop \ttto} & \\
\begin{tikzpicture}[baseline={([yshift=0]current bounding box.center)}]
\draw[thick] (0,-.5) -- (0,0) arc (0:180:0.5);
\draw[thick] (-.5,.5) -- (-.5,1);
\draw[thick,fill=white,rounded corners=3] (-1.2,-.2) rectangle (-1+0.2,0.2);
\draw[thick,blue,dashed,rounded corners=3] (-1.25,-.25) rectangle (-1+0.25,0.25);
\node at (-1,0-.05) {\scriptsize $\theta_{\! A}$};
\draw[thick,fill=white,rounded corners=3] (-.2,-.2) rectangle (0.2,0.2);
\node at (0,0-.1) {\scriptsize $F'$};
\end{tikzpicture}
\arrow[rr,Rightarrow,blue,"m_A"']& & 
\begin{tikzpicture}[baseline={([yshift=0]current bounding box.center)}]
\draw[thick] (0,-.5) -- (0,0) arc (0:180:0.5);
\draw[thick] (-.5,.5) -- (-.5,1);
\draw[thick,fill=white,rounded corners=3] (-1.2,-.2) rectangle (-1+0.2,0.2);
\node at (-1,0-.05) {\scriptsize $\theta'_{\!A}$};
\draw[thick,fill=white,rounded corners=3] (-.2,-.2) rectangle (0.2,0.2);
\node at (0,0-.1) {\scriptsize $F'$};
\end{tikzpicture}
\end{tikzcd}
$$
in $\CstarTwoCat(\cA(A \to A') \to \cA'(FA \to F'A'))$.
\end{itemize}
Furthermore, we ask that the data of the underlying family of 2-morphisms and invertible modifications form a 3-modification in the sense of \cite[Definition 4.18]{gurski_2013}. More explicitly, these include a monoidality/composability axiom and a unitality axiom.
\end{remark}
\tikzset{Rightarrow/.style={double equal sign distance,>={Implies},->},
triple/.style={-,preaction={draw,Rightarrow}},
quadruple/.style={preaction={draw,Rightarrow,shorten >=0pt},shorten >=1pt,-,double,double
distance=0.2pt}}

\begin{definition}
A uniformly bounded perturbation 
\begin{tikzcd}[baseline={([yshift=0]current bounding box.center)},column sep = 7]
\sigma \colon m \arrow[r,quadruple]& m'
\end{tikzcd}
between $\dag$-3-modifications consists of a family of 3-morphisms $\sigma_A \in \cA'(m_a \ttto n_a)$ indexed by objects in $\cA$, satisfying the obvious compatibility axiom with the higher data for $m$ and $n$, and 
$$
\|\sigma \| \coloneqq \sup_{A \in \cA} \|\sigma_A\| < \infty.
$$
We refer the curious reader to \cite[Definition 4.21]{gurski_2013}.
\end{definition}

\begin{example}
From \cite{CP22} we see that $\CstarTwoCat$ forms a C*-3-category. Hence, the sub-$\dag$-3-category $\CstarTwoCat_{\Strict}$ of strict C*-2-categories, strict $\dag$-2-functors, all $\dag$-2-natural transformations, and uniformly bounded modifications forms a C*-3-category.\\
Analogously, $\WstarTwoCat$ forms a W*-3-category and there is a sub-W*-3-category $\WstarTwoCat_{\Strict}$ of strict W*-2-categories, strict normal $\dag$-2-functors, all $\dag$-2-natural transformations, and uniformly bounded modifications. 

Moreover, there are full sub-W*-3-categories of $\WstarTwoCat$, denoted by $\ThreeHilb$ and $\mathsf{FSSU2C}$, whose objects are concrete W*-2-categories and finitely semisimple unitary 2-categories respectively. 
\end{example}

\begin{example}
A monoidal C*-2-category can be viewed as a C*-3-category with a single object. Similarly, a braided C*-category can be viewed as a C*-3-category with a single object, and a single (identity) 1-morphism. In particular, every unitary braided multifusion category can be viewed as a W*-3-category.
\end{example}

The following examples are the appropriate operator-algebraic analogues for Morita 2-categories. We refer the reader to \cite{CHJP2022} for a detailed exposition.

\begin{example}
There is a C*-3-category $\CHaus$ whose:
\begin{itemize}
\item[(0)] objects are compact Hausdorff spaces;
\item[(1)] 1-morphisms $A\colon X \to Y$ are C*-algebras $A$ equipped with $*$-homomorphisms $C(X) \to Z(A)$ and $C(Y) \to Z(A)$;
\item[(2)] 2-morphisms $H\colon A \tto B$ are $A$-$B$ C*-correspondences compatible with the $C(X)$ and $C(Y)$ actions,
\item[(3)] 3-morphisms are adjointable intertwiners.
\end{itemize}
One defines the analogous W*-3-category $\Meas$ whose:
\begin{itemize}
\item[(0)] objects are $\sigma$-finite measure spaces $(X,\mu$);
\item[(1)] 1-morphisms $A\colon (X,\mu) \to (Y,\nu)$ are W*-algebras $A$ equipped with normal $*$-homomorphisms $L^\infty(X,\mu) \to Z(A)$ and $L^\infty(Y,\nu) \to Z(A)$;
\item[(2)] 2-morphisms $H\colon A \tto B$ are $A$-$B$ W*-correspondences compatible with the $L^\infty(X,\mu)$ and $L^\infty(Y,\nu)$ actions;
\item[(3)] 3-morphisms are adjointable normal intertwiners.
\end{itemize}   
We verify the details in Appendix \S \ref{sec:examples}. 
\end{example}

In what follows we make reference to multifusion categories and their bimodule categories, functors, and natural transformations. We direct the reader to the standard reference \cite{Etingof_Gelaki_Nikshych_Ostrik_2017} for details.

\begin{example}
There is a W*-3-category $\mathsf{UmFC}$ whose:
\begin{enumerate}[(1)]
\item[(0)] objects are unitary multifusion categories;
\item 1-morphisms $M \colon \cC \to \cD$ are unitary $\cC$-$\cD$ bimodule categories;
\item 2-morphisms $F \colon M \tto N$ are unitary $\cC$-$\cD$ bimodule functors;
\item 3-morphisms $\eta \colon F \ttto G$ are $\cC$-$\cD$ bimodule natural transformations.
\end{enumerate}
\end{example}

\begin{remark}
We conjecture there exists an equivalence $\mathsf{Mod}^\dag \colon \mathsf{UmFC} \cong \mathsf{FSSU2C}$, where the latter is the W*-3-category of finitely semisimple unitary 2-categories.
\end{remark}


The following example is the appropriate operator-algebraic analogue of a hom 3-category in the speculated Morita 4-category $\mathsf{UBmFC}$ of unitary braided multifusion categories.
\begin{example}
For unitary braided multifusion categories $\cA$ and $\cB$, there is a W*-3-category $\mathsf{UBmFC}(\cA \to \cB)$ whose:
\begin{enumerate}[(1)]
\item[(0)] Objects are unitary multifusion categories $\cC$ equipped with unitary braided functors $\cA \to Z(\cC) \leftarrow \cB^{\mathsf{rev}}$;
\item 1-morphisms $M \colon \cC \to \cD$ are unitary $\cC$-$\cD$ bimodule categories together with compatibility data;
\item 2-morphisms $F \colon M \tto N$ are unitary $\cC$-$\cD$ bimodule functors satisfying compatibility conditions;
\item 3-morphisms $\eta \colon F \ttto G$ are $\cC$-$\cD$ bimodule natural transformations.
\end{enumerate}
Notice $\mathsf{UBFC}(\bbC \to \bbC) = \mathsf{UmFC}$.
\end{example}

\subsection{Transport and change of structure}

In this section, we provide three results which will allow us to locally strictify an operator 3-category, into what will be known as a ``cubical'' operator 3-category. 
In particular, these three results will allow us to strictify hom-2-categories, composition of objects, and units respectively in a compatible way.

\begin{lemma}[Transport of structure] \label{lem:transportofstructure3cats}
Let $\cA$ be a C*-3-category and let
$$\{\widetilde{\cA}(A,B) \in \CstarTwoCat\}_{A,B \in \cA}$$
be a collection of C*-2-categories indexed by pairs of objects in $\cA$ with $\dag$-2-equivalences 
$$F_{A,B}\colon \cA(A,B) \to \widetilde{\cA}(A,B) \text{ for every } A,B \in \cA.$$
Then we may upgrade this data to a C*-3-category $\widetilde{\cA}$ and a $\dag$-3-functor $F\colon \cA \to \widetilde{\cA}$ such that
\begin{itemize}
\item $\Ob \widetilde{\cA} = \Ob \cA$,
\item $\widetilde{\cA}(A,B)$ is the hom-C*-2-category from $A$ to $B$ in $\widetilde{\cA}$,
\item Choosing biadjoint $\dag$-2-equivalences $(F_{AB},F_{AB}^{\sqdot},\eta^{F_{AB}},\epsilon^{F_{AB}})$ for each $F_{AB}$, 1-composition $\widetilde{\cccirc}$ in $\widetilde{\cA}$ is then defined by
\[\begin{tikzcd}[column sep = 50pt]
	{\widetilde{\cA}^{\,2}(A \to B \to C)} & {\cA^2(A \to B \to C)} \\
	{\widetilde{\cA}(A \to C)} & {\cA(A \to C)}
	\arrow["{F_{AB} \maxblacktimes F_{BC}}", from=1-1, to=1-2]
	\arrow["{\widetilde{\cccirc}}"', dashed, from=1-1, to=2-1]
	\arrow["\cccirc", from=1-2, to=2-2]
	\arrow["{F_{AC}^{\sqdot}}", from=2-2, to=2-1]
\end{tikzcd}\]
\item $F$ is a $\dag$-3-equivalence which acts as the identity on objects and as $F_{A,B}$ on hom-C*-2-categories.
\end{itemize}
Moreover, if $\cA$ is a W*-3-category, each $\widetilde{\cA}(A,B)$ is a W*-2-category, and each $F_{A,B}$ a normal $\dag$-2-equivalence, then $\widetilde{\cA}$ is a W*-3-category.
\end{lemma}

\begin{proof}
The structure equipped on the 3-category $\widetilde{\cA}$ and the 3-functor $F$ are built using the units and counits of the biadjoint $\dag$-2-equivalences (which are unitaries) and the constraint data on $\cA$. From this it is quite easy to see that $\widetilde{\cA}$ is a C*-3-category and $F$ is a $\dag$-3-functor. Furthermore, from the definition of $\widetilde{\cccirc}$ we see that 1-composition is separately normal in $\widetilde{\cA}$ the W*-case, and $F$ is automatically normal as a $\dag$-3-equivalence. We refer the reader to \cite[\S 7.4, Theorem 7.22]{gurski_2013} for the details of this construction for non-linear 3-categories. 
\end{proof}

We will state the next two lemmas without proof, as they are similar in spirit to Lemma \ref{lem:transportofstructure3cats}, and refer the interested reader to \cite[\S 7.4, Theorems 7.23, 7.24]{gurski_2013} for details.

\begin{lemma}[Change of composition] \label{lem:changeofcomposition3cats}
Let $\cA$ be a C*-3-category and let
$$\{\widetilde{\cccirc}\colon \cA^2(A \to B \to C) \to \cA(A \to C)\}_{A,B,C \in \cA}$$
be a collection of $\dag$-2-functors indexed by triples of objects in $\cA$ with $\dag$-2-natural equivalences
$$F^2\colon \cccirc \tto \widetilde{\cccirc} \quad\text{ for every } A,B,C \in \cA.$$
Then we may upgrade this data to a C*-3-category $\widetilde{\cA}$ and a $\dag$-3-functor $F\colon \cA \to \widetilde{\cA}$ such that
\begin{itemize}
    \item $\Ob \widetilde{\cA} = \Ob \cA$,
    \item $\widetilde{\cA}(A \to B) = \cA(A \to B)$ as C*-2-categories,
    \item $\widetilde{\cA}$ has the same units $I_A$ as $\cA$,
    \item 1-composition in $\widetilde{\cA}$ is given by $\widetilde{\cccirc}$,
    \item $F$ is a $\dag$-3-equivalence which as a map acts as the identity, and
    \item Choosing unitary adjoint equivalences for $F^2$, we obtain the 1-tensorator unitary adjoint equivalences for $F$.
\end{itemize}
Furthermore, when $\cA$ is a W*-3-category and each $\widetilde{\cccirc}$ is separately normal, it is immediate that $\widetilde{\cA}$ is a W*-3-category.
\end{lemma}

\begin{lemma}[Change of units] \label{lem:changeofunits3cats}
Let $\cA$ be a C*-3-category and let
$$\{\widetilde{I}_A\colon \bbC \to \cA(A,A)\}_{A \in \cA}$$
be a collection of $\dag$-2-functors indexed by the objects in $\cA$ with $\dag$-2-natural equivalences
$$F^1\colon I_A \tto \widetilde{I}_A \quad\text{for every } A \in \cA.$$
Then we may upgrade this data to a C*-3-category $\widetilde{\cA}$ and a $\dag$-3-functor $F\colon \cA \to \widetilde{\cA}$ such that
\begin{itemize}
    \item $\Ob \widetilde{\cA} = \Ob \cA$,
    \item $\widetilde{\cA}(A \to B) = \cA(A \to B)$ as C*-2-categories,
    \item Units in $\widetilde{\cA}$ are given by $\widetilde{I}_A$,
    \item $\widetilde{\cA}$ has the same 1-composition $\cccirc$ as $\cA$,
    \item $F$ is a $\dag$-3-equivalence which as a map acts as the identity, and
    \item Choosing unitary adjoint equivalences for $F^1$, we obtain the 1-unitor unitary adjoint equivalences for $F$.
\end{itemize}
When $\cA$ is a W*-3-category, it is automatic that $\widetilde{\cA}$ is a W*-3-category.
\end{lemma}

\subsection{Operator cubical categories}

In this section, we provide an intermediate strictification step for our main coherence result.
\begin{definition}
We say that a C*-3-category $\cA$ (resp. W*-3-category) if its underlying 3-category is cubical, i.e. for all objects $A,B,C\in \cA$ we have
\begin{itemize}
    \item[(hom)] each $\cA(A \to B)$ is \emph{strict} C*-2-category (resp. W*-2-category);
    \item[($\cccirc$)] composition $\cccirc \colon \cA^2(A \to B \to C) \to \cA(A \to C)$ is a (resp. separately normal) \emph{cubical} $\dag$-2-functor;
    \item[(I)] the unit $\dag$-2-functor $I_A \colon \cA(A \to A)$ is \emph{strict}.
\end{itemize}
\end{definition}

\begin{theorem}\label{thm:everyoneisequivtocubical}
Every C*-3-category $\cA$ is equivalent to a cubical C*-3-category $\cA^{\cubical}$. Furthermore, if $\cA$ is a W*-3-category, then $\cA^{\cubical}$ is a cubical W*-3-category.
\end{theorem}

\begin{proof}
For a C*-3-category $\cA$, we define $\cA^{\cubical}$ as follows:
\begin{itemize}
\item $\Ob \cA^{\cubical} = \Ob \cA$,
\item $\cA^{\cubical}(A \to B) = \widehat{\cA(A \to B)}$ for every $A , B \in \cA$.
\end{itemize}
Using \hyperref[lem:transportofstructure3cats]{transport of structure}, we obtain a C*-3-category equivalent to $\cA$ with strict hom-C*-2-categories and 1-composition given by
$$\widehat{\cA(A \to B)} \maxblacktimes \widehat{\cA(B \to C)} \xrightarrow{ \ev \maxblacktimes \ev} \cA^2(A \to B \to C) \xrightarrow{\cccirc} \cA(A \to C) \xrightarrow{\ev^{\sqdot}} \widehat{\cA(A \to C)},
$$
where we are choosing unitary adjoint equivalences $(\ev,\ev^{\sqdot},\eta^{\ev},\epsilon^{\ev})$ for $\ev$.
Observe, by Propositions \ref{prop:cofibfunctors} and \ref{prop:cofibtensorator} we have the following $\dag$-2-natural equivalence
\[\begin{tikzcd}[column sep = 50pt, row sep = 50pt]
& {\widehat{\cA^2}} & {\widehat{\cA}} \\
	{(\widehat{\cA}\,)^2} & {\cA^2} & {\cA} & {\widehat{\cA}}
	\arrow["\cccirc"', from=2-2, to=2-3]
	\arrow["{\ev^{\sqdot}}"', from=2-3, to=2-4]
	\arrow["{\widehat{\cccirc}}", from=1-2, to=1-3]
	\arrow["\ev"{description}, from=1-3, to=2-3]
	\arrow["\ev"{description}, from=1-2, to=2-2]
	\arrow[""{name=0, anchor=center, inner sep=0}, "\id", from=1-3, to=2-4]
	\arrow["{u^{\cccirc}}"{description}, shorten <=18pt, shorten >=18pt, Rightarrow, from=1-3, to=2-2]
	\arrow[""{name=1, anchor=center, inner sep=0}, "{\ev \maxblacktimes \ev}"', from=2-1, to=2-2]
	\arrow["C", from=2-1, to=1-2]
	\arrow["{\eta^{\ev}}"{description}, shorten <=4pt, Rightarrow, from=0, to=2-3]
	\arrow["{u^{\ev}}"{description}, shorten >=12pt, shorten <=18pt,Rightarrow, from=1-2, to=1]
\end{tikzcd}\]
between the cubical $\dag$-2-functor $\widehat{\cccirc} \circ C$ and the previously mentioned 1-composition. 
Using \hyperref[lem:changeofcomposition3cats]{change of composition}, we obtain a C*-3-category with strict hom-C*-2-categories and cubical 1-composition. 
Finally, by using \hyperref[lem:changeofunits3cats]{change of units} to replace each $I_A$ with its strictification, we obtain the desired equivalent cubical C*-3-category $\cA^{\cubical}$.
In the case when $\cA$ is a W*-3-category, we use the corresponding results to conclude that $\cA^{\cubical}$ is a W*-3-category. 
\end{proof}

\subsection{Coherence and concreteness for operator algebraic tricategories}

\begin{definition}
A $\CstarGray$-category is a category enriched in $\CstarGray$ in the sense of \cite{kelly_2005}.
A functor of $\CstarGray$-categories is then just a strict $\dag$-3-functor. 
We denote the category of $\CstarGray$-categories and strict $\dag$-3-functors by $\CstarGrayCat$.\\
Similarly, a $\WstarGray$-category is a category enriched in $\WstarGray$.
A functor of $\WstarGray$-categories is then just a strict normal $\dag$-3-functor. 
We denote the category of $\WstarGray$-categories and strict normal $\dag$-3-functors by $\WstarGrayCat$.
\end{definition}

In what remains, we consider the Yoneda embedding for cubical operator 3-categories, referring the interested reader to Appendix \S \ref{sec:yoneda} for the details of this construction. 

\begin{proposition}\label{prop:yonedaformsgraycat}
For $\cA$ a C*-3-category and $\cB$ a C*-Gray-category, $\CstarThreeCat(\cA \to \cB)$ forms a C*-Gray-category of
\begin{itemize}
\item[$(0)$] $\dag$-3-functors from $\cA$ to $\cB$,
\item[$(\text{hom})$] Hom-C*-2-categories $\CstarThreeCat(\morph{\cA}{F}{\cB} \tto \morph{\cA}{G}{\cB})$ as in Appendix \S\ref{sec:yoneda} Lemma \ref{lem:hombetween3cats},
\item[$(\cccirc)$] Cubical 1-composition $\cccirc$ as in Appendix \S\ref{sec:yoneda} Lemma \ref{lem:compositionin3Cats}
\end{itemize}
Moreover, when $\cA$ is a W*-3-category and $\cB$ is a W*-Gray-category, we have that $\WstarThreeCat(\cA \to \cB)$ forms a W*-Gray-category of normal $\dag$-3-functors from $\cA$ to $\cB$.
\end{proposition}

\begin{theorem}\label{thm:yoneda}
The Yoneda embedding for a cubical C*-3-category $\cB$, given on objects by
$$B \mapsto \cB(- \to B),$$
can be equipped with the structure of monic $\dag$-3-functor
$$\yo\colon \cB \to \CstarThreeCat(\cB^{\oneop} \to \CstarTwoCat_{\Strict}),$$
which is locally a 2-equivalence. When $\cB$ is a W*-3-category, we obtain a monic normal $\dag$-3-functor
$$\yo\colon \cB \to \WstarThreeCat(\cB^{\oneop} \to \WstarTwoCat_{\Strict}).$$
\end{theorem}

\begin{theorem}[Coherence for operator 3-categories]
Every C*-3-category is 3-equivalent to a C*-Gray-category, and every W*-3-category is 3-equivalent to a W*-Gray-category.
\end{theorem}

\begin{proof}
By Theorem \ref{thm:everyoneisequivtocubical} and Theorem \ref{thm:yoneda}
$$\cA \xrightarrow{\sim} \cA^{\cubical} \xhookrightarrow{\yo} \CstarThreeCat\big((\cA^{\cubical})^{\oneop} \to \CstarTwoCat_{\Strict}\big)$$
where the latter is a C*-Gray-category by Proposition \ref{prop:yonedaformsgraycat}.
\end{proof}
 
\begin{theorem}[Gelfand-Naimark for operator 3-categories]
Every small C*-3-category is 3-equivalent to a sub-C*-Gray-category of $\ThreeHilb$.
\end{theorem}

\begin{proof}
Notice there is a monic $\dag$-3-functor $\yo^{\amalg}\colon \cA \to \CstarTwoCat_{\small}$ given on objects by
$$
\yo^{\amalg}(B) = \coprod_{A \in \cA} \cA(A \to B).
$$
By extending the GNS construction for C*-2-categories into a monic $\dag$-3-functor $\GNS\colon \CstarTwoCat_{\small} \to \ThreeHilb$, the image of the composition
$$
\cA \xhookrightarrow{\yo^{\amalg}} \CstarTwoCat_{\small} \xhookrightarrow{\GNS} \ThreeHilb
$$
is equivalent to $\cA$.
\end{proof}
\appendix

\section{Background on operator categories} \label{1cats}
In this section, we provide the requisite background for operator 1-categories in the style of Section \S \ref{sec:2cats}.

\subsection{Operator categories}\label{subsec:operator1categories} 

\begin{definition}[Operator category]
    A C*-category is a $\bbC$-linear category $\cA$ equipped with:
\begin{itemize}
    \item[($\dag$)] a conjugate linear contravariant involution $\dag$ fixing objects, and 
    \item[(C*)] a sub-multiplicative norm $\|\cdot\|$ satisfying the C*-axiom. 
\end{itemize}
    As a technical condition, we require algebraically positive morphisms to be spectrally positive, i.e.
\begin{itemize}
    \item[($\geq 0$)] for $f \in \cA(A \to B)$, there exists $g \in \End(A)$ such that $f^\dag \circ f = g^\dag \circ g$.
\end{itemize}
The previous condition holds automatically when $\cA$ admits direct sums. We further say that a C*-category $\cA$ is W* if:
\begin{itemize}
    \item[(W*)] every hom-space $\cA(A \to B)$ admits a predual $\cA(A \to B)_*$, i.e. a Banach space with 
    $$\cA(A \to B)_*^* \cong \cA(A \to B).$$  
\end{itemize}
It automatically follows that $\circ$ is separately weak*-continuous.    
\end{definition}

\begin{example}[Concrete operator 1-categories]
The prototypical example of an operator category is $\Hilb$, the W*-category of Hilbert spaces and bounded linear transformations. Here, the norm of a morphism is given by the operator norm and the dagger of a morphism is given by its adjoint. The predual 
$$\Hilb(A \to B)_*$$
of the space of bounded linear transformations between the Hilbert spaces $A$ and $B$ is given by the following Banach space of trace-class operators
$$\cL_1(B \to A) \coloneqq \{x\colon B \to A\mid \Tr(|x|) \leq \infty\}$$
after quotienting out by the negligibles of $\|\cdot\| \coloneqq \Tr(|\cdot|)$.\footnote{Here $|x|$ denotes the square root $(x^\dag x)^{1/2}$ of the element $x^\dag \circ x$ in the C*-algebra $\End(A)$.}

Furthermore, every norm-closed linear subcategory of $\Hilb$ forms a C*-category, while every WOT-closed subcategory of $\Hilb$ forms a W*-category. We call such operator 1-categories \emph{concrete}.
\end{example}

The previous example motivates the following notion of an isomorphism compatible with an involution.
\begin{definition}[Unitaries]
We say that a morphism $x$ in an operator category is unitary when $x^\dag = x^{-1}$.
\end{definition}

\begin{definition}[(Normal) $\dag$-functor]
A $\dag$-functor $F\colon \cA \to \cB$ between C*-categories is a linear functor which is $\dag$-preserving, i.e.
$$F(x)^\dag = F(x^\dag) \qquad\text{for } x \in \cA(A \to B).$$ 
Furthermore, we say that a $\dag$-functor $F\colon \cA \to \cB$ between W*-categories is \emph{normal} when it is weak*-continuous on each hom-space.
\end{definition}



\begin{definition}[Uniformly bounded natural transformations]
    We say that a natural transformation $\alpha \colon F \Rightarrow G$ between $\dag$-functors $F,G\colon \cA \to \cB$ is \emph{uniformly bounded} when 
    $$\|\alpha\| \coloneqq \sup_{A \in \cA} \|\alpha_A\| < \infty.$$
\end{definition}

Operator 1-categories \emph{are} 1-categories of operators on Hilbert spaces. This is made precise by the Gelfand-Naimark theorem  \cite{Gelfand_Neumark_1994} and, in particular, the Gelfand-Naimark-Segal (GNS) construction \cite{Segal_1947}. The following categorification of the usual concreteness result for algebras was shown by \cite{GLR1985}. We also note that the GNS construction involves the Yoneda embedding for an operator category.

\begin{theorem}[GNS]\label{lem:FunctionalsUniversalRep}
Every small C*-category $\cA$ admits an embedding\footnote{By embedding, we mean injective on objects and faithful} $\dag$-functor 
$$\Upsilon\colon \cA \to \Hilb.$$
This construction satisfies the following property:
\begin{itemize}
\item For objects $A, B \in \cA$ and a functional $\varphi \in \cA(A \to B)^*$, there exist $\xi \in \Upsilon(A)$ and $\eta \in \Upsilon(B)$
such that
$$\varphi(x) = \langle \Upsilon(x) \xi, \eta \rangle, \quad
\text{ for all } x \in \cA(A \to B).$$
In this case we will write $\varphi = \langle \Upsilon \cdot \xi, \eta \rangle$.
\end{itemize}
Moreover, when $\cA$ is a small W*-category, there exists a normal embedding $\dag$-functor $\Upsilon^{\rmW^*}\colon \cA \to \Hilb$.
\end{theorem}

\begin{definition}
We define the C*-category $\GNS(\cA)$ to be the image of $\Upsilon$ in $\Hilb$, for which $\Upsilon\colon \cA \to \GNS(\cA)$ is an isomorphism of C*-categories. Moreover, we call operator subcategories of $\Hilb$ concrete.
\end{definition}


\begin{definition}[Bicommutant]
Given any C*-subcategory $\cA$ of $\Hilb$, we may take its WOT-closure in $\Hilb$, a W*-category which we will call the bicommutant $\cA''$ of $\cA$. 
\end{definition}

\begin{theorem}[Kaplansky Density Theorem]\label{thm:KDT}
For a subset $S \subset \Hilb(A \to B)$, if $x\colon A \to B$ is in the SOT-closure of $A$, then there exist $(x_\lambda) \subset S$ with $\|x_\lambda\| \leq \|x\|$ such that $x_\lambda \to x$ SOT.
\end{theorem}





\subsection{Tensor products of C*-categories}\label{subsec:tensorproductsCstar1cats}

\begin{definition}
For operator 1-categories $\cA_1$ and $\cA_2$, the algebraic tensor product $\cA_1 \otimes \cA_2$ is a linear category equipped with a dagger $\dag$ consisting of:
\begin{itemize}
    \item[(0)] objects $(A_1,A_2)$ where $A_i \in \cA_i$, and
    \item[(1)] morphisms are given by $(\cA_1 \otimes \cA_2)((A_1,A_2) \to (B_1,B_2)) \coloneqq \cA_1(A_1 \to B_1) \otimes \cA_2(A_2 \to B_2)$.
\end{itemize}
Composition and $\dag$ are determined on each tensorand.
\end{definition}

\begin{definition}
For small C*-categories, the \emph{minimal tensor product} $\cA_1 \mintimes \cA_2$ is the completion of $\cA_1 \otimes \cA_2$ on each hom-space with respect to the spatial $\rmC^*$-norm $$\|t\|_\sigma \coloneqq \|(\Upsilon_1 \otimes \Upsilon_2)(t)\|$$
where $\Upsilon_i\colon \cA_i \to \Hilb$ is the universal embedding of $\cA_i$ as seen in Theorem \ref{lem:FunctionalsUniversalRep}.
\end{definition}

\begin{definition}
The \text{maximal tensor product} $\cA_1 \maxtimes \cA_2$ is the completion of $\cA_1 \otimes \cA_2$ on each hom-space with respect to the maximal C*-norm
$$\|t\|_{\mu} \coloneqq \sup \sset{\|Ft\|}{F\colon \cA_1 \otimes \cA_2 \to \Hilb \text{ is a } \dag\text{-functor}}.$$
\end{definition}

\begin{universalprop}\label{universalprop:maxtensor}
For C*-categories $\cA_1$ and $\cA_2$, if $H\colon \cA_1 \times \cA_2 \to \cB$ is a $*$-bilinear functor\footnote{By $\cA_1 \times \cA_2$ we mean the cartesian product of the underlying categories, and by $\dag$-bilinear functor we mean a functor which is linear and $\dag$-preserving in each component.} into a C*-category $\cB$, then there exists a unique $*$-functor $H\colon \cA_1 \maxtimes \cA_2 \to \cB$ such that the following diagram commutes:
\[\begin{tikzcd}
\cA_1 \times \cA_2 \arrow[r,"H"] \arrow[d] & \cB\\
\cA_1 \maxtimes \cA_2 \arrow[ur,dashed,"H"'] & 
\end{tikzcd}\]
\end{universalprop}
\begin{theorem}
[Hom-Tensor Adjunction] For C*-categories $\cA$, $\cB$, and $\cC$, we have that the following C*-categories are unitarily naturally equivalent:
$$\CstarCats(\cA \maxtimes \cB \to \cC) \cong \CstarCats(\cA \to \CstarCats(\cB \to \cC)).$$
\end{theorem}

\begin{remark}
As was shown for C*-algebras by Takesaki and
Guichardet, there are many possible norms one can equip the algebraic tensor product $\cA_1 \otimes \cA_2$ of small C*-categories such that completion yields a C*-category. It turns out that for any such norm $\|\cdot\|$, we have
$$\|\cdot \|_\sigma \leq \|\cdot \| \leq \|\cdot \|_\mu.$$
We refer the reader to Appendix T of \cite{Wegge-Olsen_2004} for an exposition on this topic.
\end{remark}

\subsection{W*-completion of a C*-category}\label{subsec:Wstarcompletion1cats}

The following section is a categorification of the results in \cite{Arens1} and \cite{Arens2} for C*-algebras. We refer the reader to \cite{Pal74} for an exposition on the topic. As we will build on these facts to construct the W*-completion of a C*-2-category in Section \ref{sec:2}, we will provide sketches for some of our proofs concerning these results.

For a C*-category $\cA$, we wish to construct the enveloping W*-category $\rmW^*(\cA)$ together with an embedding $\dag$-functor $\cA \hookrightarrow \rmW^*(\cA)$.
Intuitively, we wish to find the ``smallest'' C*-category containing $\cA$ whose hom-spaces admit preduals.
Recall that, for a C*-algebra $\cA$, there exists an organic inclusion $\ev\colon \cA \hookrightarrow \cA^{**}$ given by
$$\ev_x(\varphi) \coloneqq \varphi(x) \quad\text{for } x \in \cA \text{ and } \varphi \in \cA^{*}.$$
Clearly $\cA^{**}$ has a predual, namely $\cA^*$, and the Goldstine theorem tells us that $\cA^{**}$ is ``small'' in the following sense:

\begin{theorem}[Goldstine]
For a C*-algebra\footnote{This theorem holds more generally for any Banach space.} $\cA$, $\ev(\cA)$ is weak*-dense in $\cA^{**}$.
\end{theorem}
This leads us to the following construction.
\begin{definition}[Double dual of a C*-category]
When $\cA$ is a C*-category, we construct $\cA^{**}$ as follows: 
\begin{itemize}
    \item[(0)] objects are the same as in $\cA$,
    \item[(1)] morphisms are given by $\cA^{**}(A \to B) \coloneqq \cA(A \to B)^{**}$;
    \item[($\circ$)] we define two so-called Arens compositions on $\cA^{**}$, which equip $\cA^{**}$ with the structure of a (linear) category.
\end{itemize}
\end{definition}

\begin{remark}
    Notice each $\cA^{**}(A \to B)$ admits a vector space structure and a norm, namely,
    $$\|\Phi\| \coloneqq \sup_{\varphi \in \cA(A \to B)^*} \frac{|\Phi(\varphi)|}{\|\varphi\|}, \qquad \text{for } \Phi \in \cA^{**}(A \to B).$$
\end{remark}

\begin{definition}[Arens compositions]
For $\Phi \in \cA^{**}(A \to B)$ and $\Psi \in \cA^{**}(B \to C)$, we define the left and right Arens compositions $\circ_\ell$ and $\circ_r$ as follows:
\begin{itemize}
\item[($\ell$)] For $\varphi \in \cA^*(A \to C)$, we set $(\Psi \circ \Phi)(\varphi) \coloneqq \Psi(\Phi \triangleright \varphi)$ where $\Phi \triangleright \varphi \in \cA(B \to C)^*$ is given by:
\begin{itemize}
    \item[($\triangleright$)] For $b \in \cA(B \to C)$, we set $(\Phi \triangleright \varphi)(b) \coloneqq \Phi(\varphi \triangleleft b)$ where $\varphi \triangleleft b \in \cA(A \to B)^*$ is given by:
\begin{itemize}
    \item[($\triangleleft$)] For $a \in \cA(A \to B)$, we set $(\varphi \triangleleft b)(a) \coloneqq \varphi(b \circ a)$. 
\end{itemize}
\end{itemize}
\item[($r$)] For $\varphi \in \cA^*(A \to C)$, we set $(\Psi \circ_r \Phi)(\varphi) \coloneqq \Phi(\varphi \triangleleft \Psi)$ where $\varphi \triangleleft \Psi \in \cA(A \to B)^*$ is given by:
\begin{itemize}
    \item[($\triangleleft$)] For $a \in \cA(A \to B)$, we set $(\varphi \triangleleft \Psi)(a) \coloneqq \Psi(a \triangleright \varphi)$ where $a \triangleright \varphi \in \cA(B \to C)^*$ is given by: 
\begin{itemize}
    \item[($\triangleright$)] For $b \in \cA(B \to C)$, we set $(a \triangleright \varphi)(b) \coloneqq \varphi(b \circ a)$. 
\end{itemize}
\end{itemize}
\end{itemize}
\end{definition}

\begin{definition}
We define $\ev\colon \cA \hookrightarrow \cA^{**}$ as follows:
\begin{itemize}
\item[(0)] For an object $A \in \cA$, 
$\ev(A) \coloneqq A,$
\item[(1)] For $x \in \cA(A \to B)$, we define $\ev(x) = \ev_x \in \cA^{**}(A \to B)$ by
$$\ev_x(\varphi) \coloneqq \varphi(x) \qquad\text{ for } \varphi \in \cA(A \to \cB).$$
\end{itemize}
\end{definition}

\begin{lemma}\label{lem:dagpreserving}
$\ev\colon \cA \hookrightarrow \cA^{**}$ is $\dag$-preserving when we equip $\cA^{**}$ with either Arens composition.
\end{lemma}





\begin{definition}[$\dag$]
We define a conjugate-linear contravariant map $\dag\colon \cA^{**}(A \to B) \to \cA^{**}(B \to A)$ as follows:
\begin{itemize}
    \item[($\dag$)] For $\Phi \in \cA^{**}(A \to B)$, we define $\Phi^\dag \in \cA^{**}(B \to A)$ by 
    $$\Phi^\dag(\varphi) \coloneqq \overline{\Phi(\varphi^\dag)} \quad\text{for } \varphi \in \cA(B \to A)^*,$$
    where $\varphi^\dag \in \cA(A \to B)^*$ is given by $\varphi^\dag(a) \coloneqq \overline{\varphi(a^\dag)}$ for $a \in \cA(A \to B)$.
    
\end{itemize} 
\end{definition}



We now relate the Arens compositions via the following identity, which follows from a straightforward computation.
\begin{lemma}\label{lemma:arensanddags}
For $\Phi \in \cA^{**}(A \to B)$ and $\Psi \in \cA^{**}(B \to C)$,
$$(\Psi \circ_\ell \Phi)^\dag = \Phi^\dag \circ_r \Psi^\dag.$$
\end{lemma}
From our previous result, we see that if the Arens compositions on $\cA^{**}$ agree then:
\begin{itemize}
    \item $(\Psi \circ \Phi)^\dag = \Phi^\dag \circ \Psi^\dag$ where $\circ = \circ_\ell = \circ_r$ and
    \item $\ev\colon \cA \hookrightarrow \cA^{**}$ is a $\dag$-functor by \ref{lem:dagpreserving}.
\end{itemize}
It turns out this is always the case for C*-categories, after which one proves that $\cA^{**}$ satisfies the universal property required of the W*-completion of $\cA$.

\begin{theorem}\label{thm:ArensCoincide}
For a C*-category $\cA$, the left and right Arens compositions on $\cA^{**}$ coincide. 
Furthermore, these serve to equip $\cA^{**}$ with the structure of a W*-category.
\end{theorem}

\begin{proof}
Without loss of generality, one assumes $\cA$ is small since compositions coincide if and only if they agree on each small subcategory. We first extend the universal representation $\Upsilon\colon \cA \to \Hilb$ along $\ev$
\[\begin{tikzcd}
	\cA & \Hilb \\
	{\cA^{**}}
	\arrow["\ev"', hook, from=1-1, to=2-1]
	\arrow["\Upsilon", from=1-1, to=1-2]
	\arrow["{\widetilde{\Upsilon}}"', dashed, from=2-1, to=1-2]
\end{tikzcd}\]
by declaring that for each $\Phi \in \cA^{**}(A \to B)$, we have
$$\langle \widetilde{\Upsilon}(\Phi) \xi,\eta \rangle = \Phi(\langle \Upsilon \cdot \xi,\eta \rangle ) \quad\text{for all } \xi \in \Upsilon(A) \text{ and } \eta \in \Upsilon(B).$$
In particular, this determines a bounded operator $\widetilde{\Upsilon}(\Phi) \colon \Upsilon(A) \to \Upsilon(B)$ with $\|\widetilde{\Upsilon}(\Phi)\| \leq \|\Phi\|$.
Since every functional $\varphi \in \cA(A \to B)^*$ is of the form $\varphi = \langle \Upsilon \,\cdot\,\xi,\eta \rangle$, it follows that  $\Phi = 0$ if and only if $\widetilde{\Upsilon}(\Phi) = 0$.
For $\Phi \in \cA^{**}(A \to B)$ and $\Psi \in \cA^{**}(B \to C)$, one computes
$$\langle \widetilde{\Upsilon}(\Psi \circ_\ell \Phi) \xi,\eta \rangle = \langle \widetilde{\Upsilon}(\Psi) \widetilde{\Upsilon}(\Phi) \xi,\eta \rangle = \langle \widetilde{\Upsilon}(\Psi \circ_r \Phi) \xi,\eta \rangle $$
which implies $\widetilde{\Upsilon}(\Psi \circ_\ell \Phi) = \widetilde{\Upsilon}(\Psi) \widetilde{\Upsilon}(\Phi) = \widetilde{\Upsilon}(\Psi \circ_r \Phi),$ and hence $\circ_\ell = \circ_r$.   

By Lemma \ref{lemma:arensanddags} we obtain that $\cA^{**}$ is a $\dag$-category. Since $\Hilb$ is an operator category, to show $\cA^{**}$ is a C*-category it suffices to show that $\widetilde{\Upsilon}$ is an isometric $\dag$-functor. We then compute
$$
\langle \widetilde{\Upsilon}(\Phi^\dag) \xi,\eta \rangle
=
\langle \widetilde{\Upsilon}(\Phi)^\dag \xi,\eta \rangle,
$$
from which we see that $\widetilde{\Upsilon}$ is $\dag$-preserving. 

We now verify that $\|\widetilde{\Upsilon}(\Phi)\| = \|\Phi\|$ for $\Phi \in \cA^{**}(A \to B)$.
By construction, it is clear that $\widetilde{\Upsilon}\colon \cA^{**} \to \Hilb$ is weak*-WOT continuous.
Let $\varepsilon > 0$. 
As a property of the \hyperref[lem:FunctionalsUniversalRep]{GNS construction} for $\cA$, we know there exist $\xi \in \Upsilon(A), \eta \in \Upsilon(B)$ with $\|\langle \Upsilon \cdot \xi, \eta \rangle\| = 1$ such that
$$|\langle \widetilde{\Upsilon}(\Phi) \xi, \eta \rangle | = | \Phi(\langle \Upsilon \cdot \xi, \eta \rangle) | \geq \|\Phi\| - \varepsilon.$$
Since $\widetilde{\Upsilon}$ is weak*-WOT continuous, $\widetilde{\Upsilon}(\Phi) \in \overline{\Im \Upsilon}^\mathrm{WOT} = \overline{\Im \Upsilon}^\mathrm{SOT} \subseteq \Hilb(\Upsilon A \to \Upsilon B)$. By the \hyperref[thm:KDT]{Kaplansky density theorem}, there exist $(x_\lambda) \subset \cA(A \to B)$ with $\|x_\lambda\| = \|\Upsilon(x_\lambda)\| \leq \|\widetilde{\Upsilon}(\Phi)\|$ such that $\Upsilon(x_\lambda) \to \widetilde{\Upsilon}(\Phi)$ SOT. Observe
$$\|\widetilde{\Upsilon}(\Phi)\| \geq \|x_\lambda\| \geq |\langle \Upsilon(x_\lambda) \xi, \eta \rangle| \to |\langle \widetilde{\Upsilon}(\Phi) \xi, \eta \rangle | \geq \|\Phi\| - \varepsilon.$$
Since $\varepsilon \geq 0$ was arbitrary, $\|\widetilde{\Upsilon}(\Phi)\| \geq \| \Phi\|$. We conclude that $\widetilde{\Upsilon}$ is isometry, and hence $\cA^{**}$ is a W*-category.
\end{proof}

\begin{universalprop}
For every $\dag$-functor $F\colon \cA \to \cB$ into a W*-category $\cB$, there exists a unique normal extension $\widetilde{F}\colon \cA^{**} \to \cB$ making the following diagram commute.
\[
\begin{tikzcd}
	\cA & \cB \\
	{\cA^{**}}
	\arrow[hook, from=1-1, to=2-1, "\ev"']
	\arrow["F", from=1-1, to=1-2]
	\arrow["{\exists!\widetilde{F}}"', dashed, from=2-1, to=1-2]
\end{tikzcd}
\]
\end{universalprop}

\begin{definition}\label{defn:F**}
For every $\dag$-functor $F \colon \cA \to \cB$ between C*-categories $\cA$ and $\cB$, there exists a unique normal extension $F^{**}\colon \cA^{**} \to \cB^{**}$ afforded by the universal property of $\cA^{**}$, which makes the following diagram commute.
\[
\begin{tikzcd}
	\cA & \cB \\
	{\cA^{**}} & {\cB^{**}}
	\arrow[hook, from=1-1, to=2-1, "\ev"']
    \arrow[hook, from=1-2, to=2-2, "\ev"]
	\arrow["F", from=1-1, to=1-2]
	\arrow["{\exists!F^{**}}"', dashed, from=2-1, to=2-2]
\end{tikzcd}
\]
\end{definition}
Using the facts shown in the proof of Theorem \ref{thm:ArensCoincide}, one proves the following result.
\begin{corollary}[Sherman-\!Takeda for C*-categories]\label{cor:shermantakeda}
For a small C*-category $\cA$, the $\dag$-functor 
$$\widetilde{\Upsilon}\colon \cA^{**} \to \GNS(\cA)''$$
constructed in Theorem \ref{thm:ArensCoincide} is an equivalence of W*-categories extending $\Upsilon\colon \cA \to \GNS(A)$ as follows:
\[\begin{tikzcd}
	\cA & {\GNS(\cA)} \\
	{\cA^{**}} & {\GNS(\cA)''}
	\arrow["\ev"', hook, from=1-1, to=2-1]
	\arrow["\Upsilon", from=1-1, to=1-2]
	\arrow["{\widetilde{\Upsilon}}"', from=2-1, to=2-2]
	\arrow[hook, from=1-2, to=2-2]
\end{tikzcd}\]
\end{corollary}

\subsection{W*-tensor product of W*-categories}\label{subsec:Wstartensorproducts1cats}

The following result is a categorification of the results in \cite{Dauns_1972} for W*-algebras. As we will adapt all of these results for W*-2-categories in Section \ref{sec:W*Graytensor}, we will omit the details in this work. We however present the main result here for completeness.

\begin{universalprop}
For W*-categories $\cA_1$ and $\cA_2$, there is a W*-category $\cA_1 \Wmaxtimes \cA_2$ equipped with separately normal $*$-bilinear functor $\cA_1 \times \cA_2 \to \cA_1 \Wmaxtimes \cA_2$ satisfying the following universal property:

\begin{itemize}
    \item For every separately normal $*$-bilinear functor $H\colon \cA_1 \times \cA_2 \to \cB$ into a W*-category $\cB$, there exists a unique normal $*$-functor $\overline{H}\colon \cA_1 \Wmaxtimes \cA_2 \to \cB$ such that the following diagram commutes:
\[\begin{tikzcd}
\cA_1 \times \cA_2 \arrow[r,"H"] \arrow[d] & \cB\\
\cA_1 \Wmaxtimes \cA_2 \arrow[ur,dashed,"\overline{H}"'] & 
\end{tikzcd}\]
\end{itemize}
\end{universalprop}

\begin{remark}
    More specifically, we construct $\cA_1 \Wmaxtimes \cA_2$ as the quotient of $(\cA_1 \maxtimes \cA_2)^{**}$ by the polar of so-called separately normal functionals on $\cA_1 \maxtimes \cA_2$.
\end{remark}

\section{Background on operator 2-categories}\label{sec:background2cats}
\subsection{Coherence and concreteness for operator algebraic bicategories}\label{subsec:coherencesforonecats}

\begin{definition}
For a C*-2-category $\cA$, the Yoneda embedding 
$$\yo: \cA^{\twoop} \to \CstarTwoCat(\cA \to \CstarCats)$$
is the $\dag$-2-functor given by:
\begin{enumerate}
\item[(0)] For an object $A \in \cA$, we define the $\dag$-2-functor $\yo^A: \cA \to \CstarCats$ by:
\begin{itemize}
\item 
For an object $B \in \cA$, we set the C*-category $\yo^A(B) \coloneqq \cA(A \to B)$,
\item
For a morphism $_{B}X_C$ in $\cA,$
we set the $\dag$-functor
$\yo^A(_{B}X_{C}) \coloneqq - \ccirc_B X_C$,
\item For a 2-morphism $x \colon \phantom{}_BX_C \to \phantom{}_BX'_C$,
we set the uniformly bounded natural transformation $\yo^A(x) \coloneqq - \ccirc_B x$. Notice $\|\yo^A(x)\| = \|x\|$.
\item We define the tensorator component $(\yo^A)^2_{X,X'}: \yo^A(X) \ccirc_{\yo^A(C)} \yo^A(X') \Rightarrow \yo^A(X \ccirc_C X')$ for two composable morphisms $_B X_C$ and $_C X' _D$ in $\cA$ to be the natural unitary $(\yo^A)^2_{X,X'} \coloneqq \alpha_{-,X,X'}^{\dag}$.
\item We define the unitor component $(\yo^A)^1: \id_{\yo^A(B)} \Rightarrow \yo^A(\id_B)$ for an object $B \in \cA$ to be the natural unitary $(\yo^A)^1_{B} \coloneqq \rho^\dag$.
\end{itemize}
\item[(1)] For a morphism $_{A'} Y _A$ in $\cA$, we define the $\dag$-2-natural transformation $\yo^Y\colon \yo^{A} \Rightarrow \yo^{A'}$ by:
\begin{itemize}
\item For an object $B \in \cA$, we set the $\dag$-functor $\yo^Y_B \coloneqq \phantom{}_{A'} Y \ccirc_A -$.
\item For a morphism $_B X _C$ in $\cA$, we set the uniformly bounded natural transformation $\yo^Y_{X} \coloneqq \alpha_{Y,-,X}^\dag$.
\end{itemize}
\item[(2)] For a 2-morphism $y\colon \phantom{}_{A'} Y _A \Rightarrow \phantom{}_{A'} Y' _A$, we define the uniformly bounded modification $\yo^y\colon \yo^Y \ttto \yo^{Y'}$ by:
\begin{itemize}
\item For an object $B \in \cA$, we set the uniformly bounded natural transformation $\yo^y_B \coloneqq y \ccirc_A -$.
\end{itemize}
\item[($\ccirc$)] We define the compositor component $\yo^2_{Y',Y}: \yo^{Y'} \ccirc \yo^Y \Rightarrow \yo^{Y' \ccirc Y}$ for two composable morphisms $_{A''}Y'_{A'}$ and $_{A'}Y_A$ in $\cA$ to be $\yo^2_{Y',Y} \coloneqq \alpha_{Y',Y,-}$,
\item[(I)] We define the unitor component $\yo^1_A: \id_{\yo_A} \Rightarrow \yo^{\id_A}$ for an object $A \in \cA$ to be $\yo^1_A \coloneqq \lambda^\dag$. 
\end{enumerate}
\end{definition}
\begin{remark}
When $\cA$ is a strict C*-2-category, it is clear from construction that the Yoneda embedding for $\cA$ is a strict $\dag$-2-functor which lands in $\CstarTwoCat_{\mathsf{st}}(\cA \to \CstarCats)$.
\end{remark}

\begin{theorem}
For a C*-2-category $\cA$, the Yoneda embedding $\yo: \cA^{\twoop} \to \CstarTwoCat(\cA \to \CstarCats)$ is fully faithful. Hence, every C*-2-category is equivalent to a strict one.
\end{theorem}

\begin{proof}
This follows by the Yoneda embedding theorem for ordinary 2-categories.
\end{proof}

\begin{theorem}
For a C*-2-category $\cA$, the universal representation $\yo^\amalg \in \CstarTwoCat( \cA \to \CstarCats)$ given by
$$\yo^{\amalg} := \coprod_{A \in \cA} \yo^{A}$$
is monic. Thus, every C*-2-category can be realized as a norm closed $\dag$-2-subcategory of $\CstarCats$. Moreover
\begin{itemize}
\item when $\cA$ is a strict, $\yo^\amalg \in \CstarTwoCat( \cA \to \CstarCats)$ is a strict $\dag$-2-functor; and
\item when $\cA$ is small, we have that $\yo^\amalg \in \CstarTwoCat( \cA \to \CstarCats_{\Small})$.
\end{itemize}
\end{theorem}
\begin{proof}
To see that $\yo^\amalg$ is injective, suppose $\yo^\amalg(A) = \yo^\amalg(B)$ for $A,B \in \cA$. Then $\id_{A} \in \yo^\amalg(B)$, which occurs only when $A = B$. To see $F$ is injective on 1-morphisms, suppose $\yo^\amalg(_A X_B) = \yo^\amalg(_A X' _B)$ and observe
$$
_A X_B
=
\yo^\amalg(_A X_B) (\id_A) 
=
\yo^\amalg(_A X'_B) (\id_A) 
=
\phantom{}_A Y_B.
$$
An identical argument yields that $\yo^\amalg$ is injective at level of 2-morphisms. Furthermore, when $\cA$ is strict, $\yo$ is a strict $\dag$-2-functor, hence $\yo^\amalg$ is strict as well. Finally, when $\cA$ is small, the objects $B \in \cA$ form a set and each hom C*-category $\cA(A \to B)$ is small, so
$$\yo^\amalg(B) = \coprod_{A \in \cA} \cA(A \to B)$$
is a small C*-category for every $A \in \cA$. 
\end{proof}

\begin{definition}
Recall that, for a small C*-category $\cB$, there exists a universal representation
$$
\cB \xrightarrow[\sim]{\Upsilon} \GNS(\cB) \hookrightarrow \GNS(\cB)''$$
where $\widetilde{\Upsilon}: \cB^{**} \xrightarrow{\sim}  \GNS(\cB)''$ is an isomorphism by Corollary \ref{cor:shermantakeda}. We may upgrade this construction to a $\dag$-2-functor 
$$\GNS''\colon \CstarCats_{\Small} \to \TwoHilb$$
as follows:
\begin{enumerate}
\item[(0)] 
On a small C*-category $\cB$, $\GNS''(\cB) = \GNS(\cB)''$.
\item[(1)]
For a $\dag$-functor $F\colon \cB_1 \to \cB_2$ between small C*-categories,
we define the normal $\dag$-functor 
$$\GNS''(F)\colon \GNS(\cB_1)'' \to \GNS(\cB_2)''$$
as follows:
\begin{enumerate}
\item[(F0)] 
On a Hilbert space $\Upsilon B \in \GNS(\cB_1)''$ where $B \in \cB_1$, we set
$$\GNS''(F)(\Upsilon B) \coloneqq \Upsilon (FB) \in \GNS(\cB_2)''.$$
\item[(F1)]
On an operator $S \in \GNS(\cB_1)''( \Upsilon B \to \Upsilon B')$, consider $\widetilde{\Upsilon}^{-1} S \in \cB_1^{**}(B \to B')$. As seen in Definition \ref{defn:F**}, consider the morphism $F^{**} \widetilde{\Upsilon}^{-1} S \in \cB_2^{**}(FB \to FB')$. We then define the morphism $\GNS''(F)(S) \in \GNS(\cB_2)''(\Upsilon FB \to \Upsilon FB')$ by
$$\GNS''(F)(S) \coloneqq \widetilde{\Upsilon} F^{**} \widetilde{\Upsilon}^{-1} S.$$
\end{enumerate}
Notice $\GNS''(F)$ is a normal $\dag$-functor since $\widetilde{\Upsilon}$ and $F^{**}$ are normal $\dag$-functors. Moreover, $\GNS''$ is strictly 1-composition-preserving since $(-)^{**}$ is strict.

\item[(2)] For a uniformly bounded natural transformation $\alpha \in \CstarCats_{\Small}(\morph{\cB_1}{F}{\cB_2} \tto \morph{\cB_1}{G}{\cB_2})$, we define the uniformly bounded natural transformation 
$$\GNS''(\alpha)\colon \GNS''(F) \Rightarrow \GNS''(G)$$
as follows:
\begin{enumerate}
    \item[($\alpha0$)] On a Hilbert space $\Upsilon B \in \GNS(\cB_1)''$ where $B \in \cB_1$, we define the component 
    $$\GNS''(\alpha)_{\Upsilon B}\colon \Upsilon(FB) \to \Upsilon(GB)$$
    in $\GNS(\cB_2)''$ by
$$\GNS''(\alpha)_{\Upsilon B} \coloneqq \Upsilon(\alpha_B) = \widetilde{\Upsilon}(\ev_{\alpha_B}).$$
\end{enumerate}
\end{enumerate}

\end{definition}

\begin{lemma}\label{lem:GNS''ismonic}
$\GNS''\colon \CstarCats_{\Small} \to \TwoHilb$ is monic.
\end{lemma}

\begin{proof}
Quite pedantically, it is clear that $\GNS''$ is injective on objects. Now consider $\dag$-functors $F,G\colon \cB_1 \to \cB_2$ between small C*-categories such that $\GNS''(F) = \GNS''(G)$. Then, for every $B \in \cB_1$, we have $\Upsilon(FB) = \Upsilon(GB)$, which implies $FB = GB$ since $\Upsilon\colon \cB_2 \to \Hilb$ is a monic $\dag$-functor. For $b \in \cB_1(B \to B')$, consider $\widetilde{\Upsilon}(\ev_b) \in \GNS(\cB_1)''(\Upsilon B \to \Upsilon B')$ and observe 
$$\widetilde{\Upsilon} F(b) = \widetilde{\Upsilon} F^{**} \ev_b =
\GNS''(F)(\widetilde{\Upsilon}(\ev_b)) = \GNS''(G)(\widetilde{\Upsilon}(\ev_b)) 
= \widetilde{\Upsilon} G(b).$$
Thus $F = G$ and we conclude $\GNS''$ is injective on 1-morphisms. Finally, consider $\alpha \in \CstarCats_{\small}(F \Rightarrow G)$ between arbitrary $\dag$-functors $F,G$ such that $\GNS''(\alpha) = 0$. For $B \in \cB_1$, observe
$$\Upsilon(\alpha_B) = \GNS''(\alpha)_{\Upsilon B} = 0,$$
hence $\alpha_B = 0$ since $\Upsilon: \cB_2 \to \Hilb$ is a monic $\dag$-functor. Therefore $\alpha = 0$ and we conclude that $\GNS''$ is injective on 2-morphisms. 
\end{proof}

\begin{theorem}[Gelfand-Naimark for C*-2-categories]
For a small C*-2-category $\cA$, the universal representation 
$$\Upsilon_2\colon \cA \to \TwoHilb$$
given by $\Upsilon_2 \coloneqq \GNS'' \circ \yo^{\amalg}$ is monic. Thus, every C*-2-category $\cA$ can be realized as a norm-closed $\dag$-2-category $\GNS_2(\cA) \coloneqq \Im \Upsilon_2$ of weakly closed $\dagger$-categories of Hilbert spaces and operators. Moreover, if $\cA$ is strict, then $\Upsilon_2$ is a strict $\dag$-2-functor.
\end{theorem}

\begin{proof}
Both $\dag$-functors $\yo^\amalg\colon \cA \to \CstarCats_{\small}$ and $\GNS''\colon \CstarCats_{\small} \to \TwoHilb$ are monic, so $\Upsilon_2$ is monic. Recall that $\GNS''$ is always strict, whereas $\yo^{\amalg}$ is strict whenever $\cA$ is. Therefore $\Upsilon_2$ is a strict $\dag$-2-functor if $\cA$ is a strict C*-2-category.
\end{proof}

\subsection{Cofibrant replacement for operator 2-categories}\label{subsec:CofibrantReplacement}

\begin{proposition}
For every C*-2-category $\cA$, there exists a strict C*-2-category $\widehat{\cA}$ together with an epic 2-equivalence $\ev_{\cA}\colon \widehat{\cA} \to \cA$. Moreover, when $\cA$ is a W*-2-category we have that $\widehat{\cA}$ is a W*-2-category and hence $\ev_{\cA}$ is automatically normal.
\end{proposition}

\begin{proof}
We provide the construction in \cite[\S 2.2.3]{gurski_2013}, noting that all the relevant data is compatible with dagger structures.

Let $\cA$ be a C*-2-category. 
We define the C*-2-category $\widehat{\cA}$ to have the same objects as $\cA$, and free paths of 1-morphisms in $\cA$ as 1-morphisms. 
More specifically, $X \in \widehat{\cA}(A \to A')$ is some finite tuple 
$$(A \xrightarrow{X_0} A_1,\, A_1 \xrightarrow{X_1} A_2,\, \hdots,\, A_n \xrightarrow{X_n} A')$$
of composable 1-morphisms in $\cA$ starting on $A$ and ending on $A'$.
We note that for every object $A \in \cA$ there is an empty string $\varnothing_A \in \widehat{\cA}(A \to A)$. 
Composition $\widehat{\ccirc}$ of 1-morphisms in $\widehat{\cA}$ is given by concatenation, which is strictly associative and unital with $\varnothing_A$ acting as the identity on $\cA$.

Before we define what 2-morphisms are in $\widehat{\cA}$, we define $\ev_\cA$ on objects to act as the identity. 
For $X \in \widetilde{\cA}(A \to A')$ we define
$$\ev_\cA(X) = (\cdots ((X_0 \ccirc X_1) \ccirc X_2) \cdots ) \ccirc X_n,$$
that is, $\ev_\cA$ acts on a free path by evaluating its left-most parenthesization in $\cA$.
In particular, when $X = \varnothing_A$, we have that $\ev_\cA(\varnothing_A) = \id_A$. 
We may thus choose the unitor $\ev^1_A: \ev_\cA(\varnothing_A) \tto \id_{\ev_\cA(A)}$ to be $\id_{\id_\cA}$. 
Note that $\ev_\cA$ is surjective on 1-morphisms since every 1-morphism $X_0 \in \cA(A \to A')$ forms a path $(X_0)$ of length 1.

For $X,Y \in \widehat{\cA}(A \to A')$, we now define the space of 2-morphisms from $X$ to $Y$ to be
$$\widehat{\cA}(X \tto Y) \coloneqq \cA(\ev_\cA(X) \tto \ev_\cA(Y)).$$

Notice that $\widehat{\cA}$ inherits compositions, linear and dagger structures, and norms from $\cA$, equipping each $\cA(A \to A')$ with the structure of a C*-category.
When $\cA$ is a W*-2-category, we further have that each $\cA(A \to A')$ is a W*-category.

By setting $\ev_\cA$ to act as the identity on 2-morphisms, we immediately have that $\ev_{\cA}$ is locally a fully-faithful $\dagger$-functor. 
For composable 1-morphisms $X = (X_0,\hdots,X_n)$ and $Y = (Y_0,\hdots,Y_m)$ in $\widehat{\cA}$, we set the tensorator $\ev^2_{X,Y}: \ev_\cA(X) \ccirc \ev_\cA(Y) \tto \ev_\cA(X \ccirc Y)$ to be the unique coherence unitary
{\small
$$\big((\cdots (X_0 \ccirc X_1) \cdots ) \ccirc X_n \big) \ccirc \big((\cdots (Y_0 \ccirc Y_1) \cdots ) \ccirc Y_m \big) \;\tto\; (\cdots((((\cdots (X_0 \ccirc X_1) \cdots ) \ccirc X_n ) \ccirc Y_0 ) \ccirc Y_1 ) \cdots )\ccirc Y_m,$$}
given by coherence for 2-categories. 
The 1-composition $\widehat{\ccirc}$ of composable 2-morphisms $a \in \widehat{\cA}(X \tto Y)$ and $a' \in \widehat{\cA}(X' \tto Y')$ is given by the unique 2-morphisms in $a \,\widehat{\ccirc}\, a' \in \widehat{\cA}((X \ccirc X') \tto (Y \circ Y'))$ such that the following square commutes
\[\begin{tikzcd}
	{\ev_{\cA}(X)\ccirc \ev_{\cA}(Y)} & {\ev_{\cA}(X \ccirc Y)} \\
	{\ev_{\cA}(X')\ccirc \ev_{\cA}(Y')} & {\ev_{\cA}(X' \ccirc Y')}
	\arrow["{a \ccirc a'}"', from=1-1, to=2-1]
	\arrow["{\ev^2_{X,Y}}", from=1-1, to=1-2]
	\arrow["{\ev^2_{X',Y'}}"', from=2-1, to=2-2]
	\arrow["{a \,\widehat{\ccirc}\, a'}", dashed, from=1-2, to=2-2]
\end{tikzcd}\]
This automatically implies that the tensorator $\ev^2_{X,Y}$ is natural in $X$ and $Y$. 
By coherence for 2-categories it follows that $\ev_\cA$ satisfies all coherence axioms for $\dag$-2-functors and $\widehat{\cA}$ satisfies all coherence axioms for C*-2-categories.
Finally, when $\cA$ is a W*-2-category, we see that $\widehat{\ccirc}$ is separately normal since both $\circ$ and $\ccirc$ are separately normal in $\cA$. 
In this case we conclude that $\widehat{\cA}$ is also a W*-2-category and that $\ev_\cA: \widehat{\cA} \to \cA$ is automatically normal as an equivalence between W*-2-categories. 
\end{proof}

\begin{proposition}
For each $\dag$-2-functor $F\colon \cA \to \cB$ between C*-2-categories, there exists a strict $\dag$-2-functor $\widehat{F}\colon \widehat{\cA} \to \widehat{\cB}$ and a unitary icon $u^F$ as follows: 
\[\begin{tikzcd}[row sep = 30pt, column sep = 30pt]
	{\widehat{\cA}} & {\widehat{\cB}} \\
	{\cA_1} & {\cA_2}
	\arrow["{\widehat{F}}", from=1-1, to=1-2]
	\arrow["{\ev_\cA}"', from=1-1, to=2-1]
	\arrow["{\ev_{\cB}}", from=1-2, to=2-2]
	\arrow["F"', from=2-1, to=2-2]
	\arrow["{u^F}"{description}, shorten <=4pt, shorten >=4pt, Rightarrow, from=1-2, to=2-1]
\end{tikzcd}\]
Moreover, when $F$ is a normal $\dag$-2-functor between W*-2-categories, then $\widehat{F}$ is also normal.
\end{proposition}
\begin{proof}
We provide the construction in \cite[\S 2.3.3]{gurski_2013}, noting that all the relevant data is compatible with dagger structures.
Let $F: \cA \to \cB$ be a $\dag$-2-functor between C*-2-categories.
We define $\widehat{F}$ to act as $F$ on objects, and on 1-morphisms by
$$\widehat{F}(A \xrightarrow{X_0} A_1,\, A_1 \xrightarrow{X_1} A_2,\, \hdots,\, A_n \xrightarrow{X_n} A')
=
(FA \xrightarrow{FX_0} FA_1,\, FA_1 \xrightarrow{FX_1} FA_2,\, \hdots,\, FA_n \xrightarrow{FX_n} FA').
$$
In particular $\widehat{F}(\varnothing_A) = \varnothing_{FA}$, so $F$ is strictly unital.
From construction it is clear that $\widehat{F}$ is also strictly $\ccirc$-preserving. 
For a 2-morphism $a \in \widehat{A}(X \tto Y)$ we define $\widehat{F}(a) \in \widehat{\cB}(\widehat{F}X \tto \widehat{F}Y)$ to be the unique 2-morphism in $\cA$ such that the following diagram commutes.
\[\begin{tikzcd}
	{F((\cdots(X_0 \ccirc X_1) \cdots ) \ccirc X_n)} & {(\cdots(FX_0 \ccirc FX_1) \cdots ) \ccirc FX_n} \\
	{F((\cdots(Y_0 \ccirc Y_1) \cdots ) \ccirc Y_m)} & {(\cdots(FY_0 \ccirc FY_1) \cdots ) \ccirc FY_m}
	\arrow["Fa"', from=1-1, to=2-1]
	\arrow["\thicksim", from=1-1, to=1-2]
	\arrow["\thicksim"', from=2-1, to=2-2]
	\arrow["{\widehat{F}a}", dashed, from=1-2, to=2-2]
\end{tikzcd}\]
Here the horizontal morphisms are the unique constraint unitaries provided by the coherence theorem for 2-functors \cite[\S 2.3]{gurski_2013}.
By this uniqueness it follows that $\widehat{F}$ satisfies the coherence axioms for a strict 2-functor.
Moreover, $\widehat{F}$ is $\dag$-preserving since $F$ is $\dag$-preserving and the constraints for $F$ are unitary. 
From our definition of $\widehat{F}$ on 2-morphisms, it is also clear that $\widehat{F}$ is linear. 
When $F$ is a normal $\dag$-2-functor between W*-2-categories, it follows that $\widehat{F}$ is normal since $F$ is and 2-composition $\circ$ is separately normal.

Notice that $\ev_\cB \circ \widehat{F}$ and $\ev_\cA \circ F$ both act like $F$ on objects, so they agree at the level of objects. 
For $X = (X_0,\hdots,X_n) \in \widehat{\cA}(A \to A')$, we define the unitary
$$u^F_X \in \cA_2\big(\ev(FX_0,\hdots,FX_n),F(\ev(X_0,\hdots,X_n))\big)$$
to be the unique constraint unitary provided by the coherence theorem for 2-functors. 
Note that $u^F_X$ is natural in $X$ since constraints are natural, and $u^F$ satisfies the axioms of a unitary icon due to uniqueness of constraints. 
\end{proof}

\begin{proposition}
For C*-2-categories $\cA_1$ and $\cA_2$, there exists a cubical $\dag$-2-functor
$$C\colon \widehat{\cA}_1 \maxblacktimes \widehat{\cA}_2 \to \widehat{\cA_1 \maxblacktimes \cA_2}$$
which is the identity on objects, and a unitary icon $u$ as follows:
\[\begin{tikzcd}[row sep = 30pt, column sep = 10pt]
	& {\widehat{\cA_1 \maxblacktimes \cA_2}} \\
	{\widehat{\cA}_1 \maxblacktimes \widehat{\cA}_2} && {\cA_1 \maxblacktimes \cA_2}
	\arrow["C", from=2-1, to=1-2]
	\arrow["\ev", from=1-2, to=2-3]
	\arrow[""{name=0, anchor=center, inner sep=0}, "{\ev \maxblacktimes \ev}"', from=2-1, to=2-3]
	\arrow["u^{\ev}"{description}, shorten >=3pt, Rightarrow, from=1-2, to=0]
\end{tikzcd}\]
Moreover, when $\cA_1$ and $\cA_2$ are W*-2-categories, we may upgrade $C$ to a separately normal cubical $\dag$-2-functor
$$\overline{C}\colon \widehat{\cA}_1 \Wblacktimes \widehat{\cA}_2 \to \widehat{\cA_1 \Wblacktimes \cA_2}$$
and $u$ to a unitary icon $\overline{u}$ as follows:
\[\begin{tikzcd}[row sep = 30pt, column sep = 10pt]
	& {\widehat{\cA_1 \Wblacktimes \cA_2}} \\
	{\widehat{\cA}_1 \Wblacktimes \widehat{\cA}_2} && {\cA_1 \Wblacktimes \cA_2}
	\arrow["\overline{C}", from=2-1, to=1-2]
	\arrow["\ev", from=1-2, to=2-3]
	\arrow[""{name=0, anchor=center, inner sep=0}, "{\ev \Wblacktimes \ev}"', from=2-1, to=2-3]
	\arrow["\overline{u}^{\ev}"{description}, shorten >=3pt, Rightarrow, from=1-2, to=0]
\end{tikzcd}\]
\end{proposition}

\begin{proof}
We adapt \cite[\S 8.1, Prop. 8.5]{gurski_2013}. 
To define $C$, we will provide the equivalent data seen in Proposition \ref{prop:dataofcubicals}. 
For objects $A_1 \in \cA_1$ and $A_2 \in \cA_2$ we define
$$C_{A_1}(A_2) = (A_1,A_2) = C_{A_2}(A_1).$$
Recall that we denote the identity of $A$ in $\cA_1$ by $\id_A$, and its identity in $\widehat{\cA_1}$ by $\varnothing_{A_1}$. 
For $X = (X_0,\hdots,X_n) \in \widehat{\cA}_2(A_2 \to A_2')$, we define
$$C_{A_1}(X) \coloneqq \big((\id_{A_1}, X_0), \hdots, (\id_{A_1}, X_n)\big).$$
In particular $C_{A_1}(\varnothing_{A_2}) = \varnothing_{(A_1,A_2)}$, so $C_{A_1}$ is strictly unital.
It is clear from definition that $C_{A_1}$ is also strictly $\ccirc$-preserving.
For $k \geq 0$, let $E^k_1$ be the 1-morphism in $\cA_1$ given by
$$E^k_1 \coloneqq \ev_{\cA_1}(\underbrace{\id_{A_1},\hdots,\id_{\cA_1}}_{k \text{ times}}).$$
For $a \in \widehat{\cA}_2(X \Rightarrow Y)$ where $X = (X_0,\hdots,X_n)$ and $Y = (Y_0,\hdots,Y_m)$, we define the 2-morphism 
$$C_{A_1}(a) \in \widehat{\cA_1 \maxblacktimes \cA_2}\big(C_{A_1}(X) \tto C_{A_1}(Y)\big)$$
as follows.
First note that $C_{A_1}(a)$ must be a 2-morphism in $\cA_1 \maxblacktimes \cA_2$ with source $(E^{n+1}_1,\ev_{\cA_1}(X))$ and target $(E^{m+1}_1,\ev_{\cA_1}(Y))$. 
We may then define
$$C_{A_1}(a) \coloneqq \gamma_{n+1,m+1} \otimes a$$
where $\gamma_{n+1,m+1} \in \cA_1(E^{n+1}_1 \tto E^{m+1}_1)$ is the unique constraint unitary given by coherence for 2-categories.
From this definition, it is immediate that $C_{A_1}$ is linear.
Since $\upsilon_{n+1,m+1}^\dag = \upsilon_{n+1,m+1}^{-1} = \upsilon_{m+1,n+1}$ by uniqueness of constraints, we conclude that $C_{A_1}$ is a strict $\dag$-2-functor. 
When $\cA_1,\cA_2,\cB$ are W*-2-categories and $C$ is normal, we may identically define $\overline{C}_{A_1}\colon \cA_2 \to \widehat{\cA_1 \Wblacktimes \cA_2}$, which is normal since $\otimes$ is separately normal in $\cA_1 \Wblacktimes \cA_2$. 

One defines the (normal) strict $\dag$-2-functor $F_{A_2}$ in a similar fashion.
We now define the unitary $\Sigma_{X,Y}$ for 1-morphisms $X = (X_0,\hdots,X_n) \in \widehat{\cA}_1(A_1 \to A_1')$ and $Y = (Y_0,\hdots,Y_m) \in \widehat{\cA}_2(A_2 \to A_2')$ as follows.
\[\begin{tikzcd}[column sep = 40pt, row sep = 40pt]
	{(A_1,A_2)} & {(A_1,A_2')} \\
	{(A_1',A_2)} & {(A_1',A_2')}
	\arrow["{C_{A_1}(Y)}", from=1-1, to=1-2]
	\arrow["{C_{A_2}(X)}"', from=1-1, to=2-1]
	\arrow["{C_{A_2'}(X)}", from=1-2, to=2-2]
	\arrow["{C_{A_1'}(Y)}"', from=2-1, to=2-2]
	\arrow["{\Sigma_{X,Y}}"{description}, shorten <=5pt, shorten >=5pt, Rightarrow, from=1-2, to=2-1]
\end{tikzcd}\]
First note that $\Sigma_{X,Y}$ must be a 2-morphism in $\cA_1 \maxblacktimes \cA_2$ with source 
$$\ev_{\cA_1 \maxblacktimes \cA_2}\big((\id_{A_1},Y_1),\hdots,(\id_{A_1},Y_m),(X_1,\id_{A_2'}),\hdots,(X_n,\id_{A_2'})\big)$$
and target
$$\ev_{\cA_1 \maxblacktimes \cA_2}\big((X_1,\id_{A_2}),\hdots,(X_n,\id_{A_2}),(\id_{A_1'},Y_1),\hdots,(\id_{A_1'},Y_m)\big).$$
We may then define $\Sigma_{X,Y}$ to be the unique constraint unitary in $\cA_1 \maxblacktimes \cA_2$ given by coherence for 2-categories.
Note that $(\Sigma_1)$ follows by naturality of constraints, while $(\Sigma_2)$ and $(\Sigma_3)$ follows from their uniqueness. Hence, this data serves to determine a cubical $\dag$-2-functor 
$$C\colon \widehat{\cA}_1 \maxblacktimes \widehat{\cA}_2 \to \widehat{\cA_1 \maxblacktimes \cA_2},$$
and in the W* case, a separately normal cubical $\dag$-2-functor
$$\overline{C}\colon \widehat{\cA}_1 \Wblacktimes \widehat{\cA}_2 \to \widehat{\cA_1 \Wblacktimes \cA_2}.$$
Note that $\ev \circ\, C$ and $\ev \maxblacktimes \ev$ act as the identity on objects, so they agree at the level of objects. 
For a 1-morphism $(X,Y)$ in $\widehat{A}_1 \maxblacktimes \widehat{A}_2$ where $X = (X_0,\hdots,X_n)$ and $Y = (Y_0,\hdots,Y_m)$ we define the unitary 
$$u^{\ev}_{X,Y} \in (\cA_1 \maxblacktimes \cA_2)\Big(\ev\big(F_{A_1}(Y) \ccirc F_{A_2}(X)\big) \tto \big(\ev(X),\ev(Y)\big)\Big)$$
as follows. 
Notice
$\ev\big(F_{A_1}(Y) \ccirc F_{A_2}(X)\big) = (E^{n+1}_1 \ccirc \ev(X), \ev(Y) \ccirc E^{m+1}_2)$
We may then define $u^{\ev}_{X,Y} \coloneqq v_1 \otimes v_2$ where $v_1 \in \cA_1( E^{n+1}_1 \ccirc \ev(X) \tto \ev(X))$ and $v_2 \in \cA_2(\ev(Y) \ccirc E^{m+1}_2 \tto \ev(X))$ to be the unique constraint unitaries given by coherence for 2-categories.
We see that $u^{\ev}_{X,Y}$ is natural in $X$ and $Y$ since constraints are natural, and $u$ satisfies the axioms for unitary icons by uniqueness of constraints.
In the W* case, we define $\overline{u}^{\ev}$ identically to $u^{\ev}$.
\end{proof}

\section{Examples of operator algebraic tricategories: $\CHaus$ and $\Meas$} \label{sec:examples}

\begin{example}
We define $\CHaus$ to be the following C*-3-category:
\begin{enumerate}
\item Objects are compact Hausdorff spaces;
\item A morphism $A\colon X \to Y$ is a C*-algebra equipped with $*$-homomorphisms $C(X) \to Z(A)$ and $C(Y) \to Z(A)$;
\item A 2-morphism $H\colon A \tto B$ is an $A$-$B$ C*-correspondence, i.e. a right B-Hilbert module $H$ together with a left $A$-action $A \to \cL_B(H)$, such that 
$$ay \cdot h \cdot b = a \cdot h \cdot yb \quad\text{for } a \in A, y \in C(Y), h \in H, b \in B,$$
and a similar compatibility axiom for the $C(X)$-action.
Here $\cL_B(H)$ is the C*-algebra of adjointable $B$-intertwiner
\item A 3-morphism $\varphi\colon H \to K$ is an adjointable intertwiner, and define $\varphi^\dag$ to be its adjoint.
\end{enumerate}

\begin{itemize}
\item[($\circ$)] The composition $\circ$ of 3-morphisms is given by composition of maps, which is clearly linear and $\dag$-preserving.

\item[($\ccirc$)] The composition $\ccirc$ of 2-morphisms $H\colon A \tto B$ and $K\colon B \tto C$ is given by the right C-Hilbert module 
$$H \otimes_{B} K \coloneqq H \otimes K / \overline{\Span}(hb \otimes k - h \otimes bk : h \in H, b \in B, k \in K)$$
with $C$-valued inner product determined by
$$\langle h \otimes k, h' \otimes k' \rangle \coloneqq \langle k, \langle h, h' \rangle \cdot k' \rangle \quad\text{for } h,h' \in H \text{ and } k,k' \in K.$$
We determine the left action of $A$ on $H \otimes_B K$ by
$$a \cdot (h \otimes k) \coloneqq (a \cdot h) \otimes k,$$
which is well-defined since the action of $A$ on $H$ is given by right $B$-module homomorphisms. We define the right $C$-action on $H \otimes_B K$ similarly. Furthermore, notice
\begin{align*}
ay \cdot (h \otimes k) \cdot c
&=
(ay \cdot h) \otimes (k \cdot c)\\
&=
(a \cdot h \cdot y) \otimes (k \cdot c)\\
&=
(a \cdot h) \otimes (y \cdot k \cdot c)\\
&=
(a \cdot h) \otimes (k \cdot yc)\\
&=
a \cdot (h \otimes k) \cdot yc,
\end{align*}
and a similar argument shows the compatibility axiom for the $C(X)$-action.
Hence $H \ccirc_B K$ satisfies the property required of 2-morphisms. The composition $\ccirc$ of 3-morphisms $\varphi\colon \morph{A}{H}{B} \ttto \morph{A}{H'}{B}$ and $\psi\colon \morph{B}{K}{C} \ttto \morph{B}{K'}{B}$ is determined by
$$(\varphi \ccirc_B \psi)(h \otimes k) = \varphi(h) \otimes \psi(k) \quad\text{for } h \in H, k \in K.$$
This is well-defined on $H \otimes_B K$ since
$$(\varphi \ccirc_B \psi)(hb \otimes k) =
\varphi(h)b \otimes \psi(k)
=
\varphi(h) \otimes b\psi(k)
=
(\varphi \ccirc_B \psi)(h \otimes bk).
$$
One similarly shows that $\varphi \ccirc_B \psi$ is intertwiner. To see that it is also adjointable, observe
\begin{align*}
\langle (\varphi \ccirc_B \psi)(h \otimes k), h' \otimes k' \rangle &=
\langle \psi(k), \langle \varphi(h),h' \rangle \cdot k' \rangle \\
&=
\langle k , 
\langle h, \varphi^\dag(h') \rangle  \cdot \psi^\dag(k') \rangle\\
&=
\langle h \otimes k, (\varphi^\dag \otimes_B \psi^\dag)(h' \otimes k') \rangle
\end{align*}
Hence $(\varphi \ccirc_B \psi)^\dag = \varphi^\dag \ccirc_B \psi^\dag$. It is clear that $\ccirc$ is bilinear from this definition.

\item[($\cccirc$)] We define the composition $\cccirc$ of 1-morphisms $A\colon X \to Y$ and $B \colon Y \to Z$ to be
$$A \cccirc_Y B \coloneqq (A \maxtimes B) / \overline{\Span}(ay \otimes b - a \otimes yb: a \in A, y \in C(Y), b \in B),$$
where we note that $\overline{\Span}(ay \otimes b - a \otimes yb: a \in A, y \in C(Y), b \in B)$ is a (norm-closed) two-sided ideal in $A \maxtimes B$ since
\begin{align*}
(a' \otimes b') (ay \otimes b - a \otimes yb) &=
a'ay \otimes b'b - a'a \otimes b' y b = 
a'ay \otimes b'b - a'a \otimes y b' b\\
(ay \otimes b - a \otimes yb) (a' \otimes b') 
&=
aya' \otimes bb' - aa' \otimes ybb'
=
aa' y \otimes bb' - aa' \otimes ybb'.
\end{align*}
In fact, $A \cccirc_Y B$ is the coequalizer of $C(Y) \to A \to A \maxtimes B$ and $C(Y) \to B \to A \maxtimes B$ given by
$$
A \cccirc_Y B = (A \maxtimes B)/ \langle y \otimes 1_B - 1_A \otimes y : y \in C(Y)\rangle.
$$
We determine the maps $C(X) \to Z(A \ccirc_Y B)$ and $C(Y) \to Z(A \ccirc_Y B)$ by $x \mapsto x \otimes 1_B$ and $y \mapsto 1_A \otimes y$ respectively, which are clearly $*$-homomorphisms.

We define the composition $\cccirc$ of 2-morphisms $H\colon \morph{X}{A}{Y} \tto \morph{X}{A'}{Y}$ and $K\colon \morph{Y}{B}{Z} \tto \morph{Y}{B'}{Z}$ to be the $(A \cccirc_Y B)$-$(A' \cccirc_Y B')$ C*-correspondence with right $(A' \cccirc_Y B')$-Hilbert module
$$H \cccirc_Y K 
\coloneqq
H \otimes K / \overline{\Span}(hy \otimes k - h \otimes yk : h \in H, y \in C(Y), k \in K).$$
Indeed, we equip $H \cccirc_Y K$ with the $(A' \cccirc_Y B')$-valued inner product determined by
$$
\langle h \otimes k, h' \otimes k' \rangle 
=
\langle h, h' \rangle \otimes \langle k, k' \rangle.
$$
Notice this is well-defined since
\begin{align*}
\langle hy \otimes k, h' \otimes k'\rangle 
&=
\langle hy, h' \rangle \otimes \langle k, k' \rangle
=
\langle h,h' \rangle y^* \otimes \langle k,k' \rangle\\
&=
\langle h,h' \rangle \otimes y^* \langle k,k' \rangle
=
\langle h,h' \rangle \otimes  \langle k,k' \rangle y^*\\
&=
\langle h,h' \rangle \otimes  \langle ky,k' \rangle 
=
\langle h,h' \rangle \otimes  \langle yk,k' \rangle \\
&= \langle h \otimes  yk, h' \otimes k' \rangle.
\end{align*}
We then determine the right $(A' \cccirc_Y B')$-action on $H \cccirc_Y K$ by
$$(h \otimes k) \cdot (a' \otimes b') = ha' \otimes kb',$$
which is well-defined since
\begin{align*}
(h \otimes k) \cdot (a'y \otimes b')
&=
ha'y \otimes kb'
=
ha' \otimes y kb'
=
ha' \otimes kyb'
=
(h \otimes k) \cdot (a' \otimes y b')\\
(hy \otimes k) \cdot (a' \otimes b')
&=
hya' \otimes kb'
=
ha'y \otimes kb'
=
ha' \otimes ykb'
=
(h \otimes yk) \cdot (a' \otimes b').
\end{align*}
One determines the left $(A \cccirc_Y B)$-action similarly.

Finally, one determines the composition $\cccirc$ of 3-morphisms akin to $\ccirc$ and it is easy to show $\cccirc$ is $\dag$-preserving and bilinear.

We now define the tensorator (or interchanger) for $\cccirc$.
In particular, for $A,A',A'' \colon X \to Y$, $B,B',B''\colon Y \to Z$ and $A \xrightarrow{H} A' \xrightarrow{H'} A''$, $B \xrightarrow{K} B' \xrightarrow{K'} B''$, we define a unitary
$$
\Sigma\colon
(H \cccirc_{Y} K) \ccirc_{A' \cccirc_Y B'} (H' \cccirc_Y K') 
\;\ttto\;
(H \ccirc_{A'} H') \cccirc_Y (K \ccirc_{B'} K'),
$$
determined by 
$$(h \otimes k) \otimes (h' \otimes k') \mapsto (h \otimes h') \otimes (k \otimes k').$$
A simple computation reveals that $\Sigma$ is a well-defined intertwiner, clearly admitting an inverse intertwiner. Hence, to show $\Sigma$ is unitary, it suffices to show it is isometric. Observe
\begin{align*}
\langle (h \otimes k) \otimes (h' \otimes k'), (\widetilde{h} \otimes \widetilde{k}) \otimes (\widetilde{h}' \otimes \widetilde{k}') \rangle 
&=
\langle h' \otimes k', \langle h \otimes k, \widetilde{h} \otimes \widetilde{k} \rangle \cdot \widetilde{h}' \otimes \widetilde{k}' \rangle\\
&=
\langle h' \otimes k', (\langle h,\widetilde{h} \rangle \otimes  \langle k, \widetilde{k} \rangle ) \cdot \widetilde{h}' \otimes \widetilde{k}' \rangle\\
&=
\langle h' \otimes k', (\langle h,\widetilde{h} \rangle \cdot \widetilde{h}' )\otimes  (\langle k, \widetilde{k} \rangle \cdot \widetilde{k}') \rangle\\
&=
\langle h', \langle h, \widetilde{h} \rangle \cdot \widetilde{h}' \rangle \otimes 
\langle k', \langle k, \widetilde{k} \rangle \cdot \widetilde{k}' \rangle \\
&=
\langle h \otimes h', \widetilde{h} \otimes \widetilde{h}' \rangle
\otimes
\langle k \otimes k', \widetilde{k} \otimes \widetilde{k}' \rangle\\
&=
\langle (h\otimes h') \otimes (k \otimes k'), (\widetilde{h} \otimes \widetilde{h}') \otimes (\widetilde{k} \otimes \widetilde{k}')\rangle.
\end{align*}
On the other hand, the unitor for $\ccirc$ is defined on components $\id_{A \cccirc_Y B} \tto \id_A \cccirc_Y \id_B$ to be the identity unitary interchanger on $A \cccirc_Y B$ viewed as a C*-correspondence over itself. Since the tensorator (interchanger) is morally a swap between the middle tensorands, it is quite easy to see that these satisfy the associativity and unitalily axioms for a $\dag$-2-functor.
\end{itemize}
It is well-known that C*-algebras, C*-correspondences, and adjointable intertwiners form a C*-2-category, which we will denote by $\CstarAlg$. \cite{CHJP2022} \cite{Blecher_LeMerdy_2004} \cite{Paschke_1973} \cite{Rieffel_1974} One can slightly modify this argument to show that $\CHaus(X,Y)$ is a C*-2-category for every $X,Y \in \CHaus$.

\begin{itemize}
    \item[(I)] For an object $X \in \CHaus$, we set $I_A \coloneqq C(X)$ where the maps $C(X) \to Z(C(X))$ are given by the identity map on $C(X)$. We then extend $I_A \colon \bbC \to \CHaus(X \to X)$ trivially into a $\dag$-2-functor.
\end{itemize}
Recall that any (non-degenerate) $*$-homomorphism $\varphi\colon A \to B$ between C*-algebras induces an $A-B$ C*-correspondence $H_\varphi$ given by:
    \begin{itemize}
        \item The underlying vector space is $B$ with $\langle b, b' \rangle = b^* b'$ for $b,b' \in B$. 
        \item The left $A$-action is given by 
        $$a \cdot b \coloneqq \varphi(a)b,$$
        for $a \in A$ and $b \in B$.
        \item The right $B$-action is given by right multiplication in $B$.
    \end{itemize}
When $\varphi$ is unitary, there exists an induced unitary adjoint equivalence $(H_\varphi,H^{\sqdot}_\varphi,\eta,\epsilon)$ given by:
\begin{itemize}
    \item the $B$-$A$ C*-correspondence $H^{\sqdot}_\varphi \coloneqq H_{\varphi^*}$,
    \item a unitary intertwiner $\eta\colon \id_A \to H_{\varphi} \ccirc_B H^{\sqdot}_{\varphi}$ defined as follows. Note that 
    $$b \otimes a = 1 \cdot \varphi(\varphi^*(b)) \otimes a = 1 \otimes \varphi^*(b) \cdot a$$
    in $H_\varphi \ccirc_B H^{\sqdot}_\varphi$. Hence every element $\sum b_i \otimes a_i$ in $H_\varphi \ccirc_B H^{\sqdot}_\varphi$ is of the form $1 \otimes a$, and it is easy to show there is a unique such $a \in A$. We thus define $\eta$ to be the bijection $a \mapsto 1 \otimes a$. Notice $\eta$ is an intertwiner
    \begin{align*}
    \eta(a) \cdot \widetilde{a} &= (1 \otimes a) \cdot \widetilde{a} = 1 \otimes (a\widetilde{a}) = \eta(a\widetilde{a})\\
    \widetilde{a} \cdot \eta(a) &= \widetilde{a} (1 \otimes a) = \varphi(\widetilde{a}) \otimes a = 1 \otimes \varphi^*(\varphi(\widetilde{a})) a = 1 \otimes \widetilde{a}a = \eta(\widetilde{a}a),
    \end{align*}
    and an isometry
    \begin{align*}
        \langle 1 \otimes a, 1 \otimes \widetilde{a} \rangle =
        \langle a , \langle 1, 1\rangle \cdot \widetilde{a} \rangle = 
        \langle a, (1^* 1) \cdot \widetilde{a} \rangle 
        =
        \langle a, \widetilde{a} \rangle.
    \end{align*}
    Therefore $\eta$ is indeed a unitary intertwiner.
    \item One defines the unitary intertwiner $\epsilon\colon H^{\sqdot}_\varphi \ccirc_B H_\varphi \to \id_B$ similarly by exchanging the roles of $A,B$ and of $\varphi,\varphi^{\sqdot}$.
\end{itemize}
A formal calculation reveals that $\epsilon$ and $\eta$ satisfies the zig-zag equations.
\begin{itemize}
    \item[(a)] For composable 1-morphisms $\morph{X}{A}{Y}$ and $\morph{Y}{B}{Z}$, The universal property of coequalizers yields an isomorphism $\varphi_{a} \colon A \cccirc_Y (B \cccirc_Z C) \to (A \cccirc_Y B) \cccirc_Z C$. This map is determined by $a \otimes (b \otimes c) \mapsto (a \otimes b) \otimes c$, and on easy verifies $\varphi_a$ is unitary. We thus define the $A \cccirc ( B \cccirc C) - (A \cccirc B) \cccirc C$-correspondence $a_{A,B,C}$ to be $H_{\varphi_a}$ and extend it to a unitary adjoint equivalence as mentioned previously.

    We now define the naturality constraint for $a$. In particular, for $X \xrightarrow{A} Y \xrightarrow{B} Z \xrightarrow{C} W$, $X \xrightarrow{A'} Y \xrightarrow{B'} Z \xrightarrow{C'} W$ , and C*-correspondences $\morph{A}{H}{A'}$, $\morph{B}{K}{B'}$, $\morph{C}{L}{C'}$, we define a unitary
    $$a_{HKL} \colon (H \cccirc (K \cccirc L)) \ccirc a_{A' B' C'} \;\ttto\; a_{ABC} \ccirc ((H \cccirc K) \cccirc L).$$
    We note that that every element in $(H \cccirc (K \cccirc L)) \ccirc a_{A' B' C'}$ can be written as a sum of elements of the form
    $$(h \otimes (k \otimes l)) \otimes ((1_{A'} \otimes 1_{B'}) \otimes 1_{C'}).$$
    We then determine $a_{HKL}$ to be the map
    $$(h \otimes (k \otimes l)) \otimes ((1_{A'} \otimes 1_{B'}) \otimes 1_{C'}) \mapsto ((1_A \otimes 1_B) \otimes 1_C) \otimes ((h \otimes k) \otimes l).$$
    A simple computation reveals that $a_{HKL}$ is a well-defined intertwiner, clearly admitting an inverse intertwiner. 
    To see that $a_{HKL}$ is unitary, we will show it is isometric. Observe,
    \begin{align*}
    &\langle (h \otimes (k \otimes l)) \otimes ((1 \otimes 1) \otimes 1), (\widetilde{h} \otimes (\widetilde{k} \otimes \widetilde{l} )) \otimes ((1 \otimes 1) \otimes 1) \rangle\\
    =\;
    &\langle 
    (1 \otimes 1) \otimes 1, 
    \langle h \otimes (k \otimes l), \widetilde{h} \otimes (\widetilde{k} \otimes \widetilde{l}) \rangle \cdot ((1 \otimes 1) \otimes 1) \rangle\\
    =\;
    &\langle
    (1 \otimes 1) \otimes 1,
    (\langle h, \widetilde{h} \rangle \otimes \langle k, \widetilde{k} \rangle) \otimes \langle l, \widetilde{l} \rangle \rangle\\
    =\;
    &(\langle h, \widetilde{h} \rangle \otimes \langle k, \widetilde{k} \rangle) \otimes \langle l, \widetilde{l} \rangle\\
    =\;
    &\langle (h \otimes k) \otimes l , \langle (1 \otimes 1) \otimes 1, (1 \otimes 1) \otimes 1 \rangle ((\widetilde{h} \otimes \widetilde{k}) \otimes \widetilde{l}) \rangle\\
    =\;
    &\langle ((1 \otimes 1) \otimes 1) \otimes ((h \otimes k) \otimes l), ((1 \otimes 1) \otimes 1) \otimes ((\widetilde{h} \otimes \widetilde{k}) \otimes \widetilde{l}) \rangle.
    \end{align*}
    It is also clear from construction that $a_{HKL}$ is natural in $H$, $K$, and $L$.
    
    \item[(u)] Notice $I_X \cccirc_X A = C(X) \cccirc_X A \cong A$ for a 1-morphism $\morph{X}{A}{Y}$.
    Denoting this unitary by $\varphi_\ell$, we define the $(I_X \cccirc A)-A$ C*-correspondence $\ell_A$ to be $H_{\varphi_\ell}$, as extend it to a unitary adjoint equivalence as mentioned previously.
    One defines the C*-correspondence $r_A = H_{\varphi_r}$ similarly through a map $\varphi_r\colon A \cccirc_{Y} I_Y \to A$. For a C*-correspondence $\morph{A}{H}{B}$ between C*-algebras $A,B\colon X \to Y$, one determines the naturality constraint unitary interchangers 
    \begin{align*}
    \ell_X\colon (I_X \cccirc H) \ccirc \ell_{B} \,&\ttto\, \ell_A \ccirc H\\
    r_X\colon (H \cccirc I_Y ) \ccirc r_{B} \,&\ttto\, r_A \ccirc H
    \end{align*}
    respectively by
    \begin{align*}
        (1_{C(X)} \otimes h) \otimes 1_B &\mapsto 1_A \otimes h,\\
        (h \otimes 1_{C(Y)}) \otimes 1_B &\mapsto 1_A \otimes h. 
    \end{align*}
    
    \item[($\pi$)] In what follows, we will supress the associators for $\ccirc$. For composable C*-algebras $A,B,C,D$, we determine the natural unitary interchanger 
    $$\pi_{ABCD}\colon (\id_A \cccirc a_{BCD}) \ccirc a_{A,B \cccirc C,D} \ccirc (a_{ABC} \cccirc \id_D) \;\ttto\; a_{A,B,C \cccirc D} \ccirc a_{A \cccirc B,C,D},$$
    to be the map which sends
    $$(1_A \otimes ((1_B \otimes 1_C )\otimes 1_D)) \otimes ((1_A \otimes (1_B \otimes 1_C)) \otimes 1_D) \otimes (((a \otimes b) \otimes c) \otimes d),$$
    to
    $$((1_A \otimes 1_B) \otimes (1_C \otimes 1_D)) \otimes (((a \otimes b) \otimes c) \otimes d).$$
    
    \item[(coh)] For composable C*-algebras $\morph{X}{A}{Y}$ and $\morph{Y}{B}{Z}$, we define the middle, left, and right unity coheretor  unitary intertwiners
    \begin{align*}
        \mu_{AB}&\colon (\id_a \cccirc \ell^{\sqdot}_B) \ccirc a_{A,I_Y,B} \ccirc (r_A \cccirc \id_B) \;\ttto\; \id_{A \cccirc B} \\
        \lambda_{AB}&\colon \ell^{\sqdot}_A \cccirc \id_B \;\ttto\; \ell^{\sqdot}_{A \cccirc B} \ccirc a_{I_X,A,B} \\
        \rho_{AB}&\colon \id_A \cccirc r_B \;\ttto a_{A,B,I_Z} \ccirc r_{A \cccirc B}
    \end{align*}
    to be the maps determined by
    \begin{align*}
        (1_A \otimes ( 1_{C(Y)} \otimes 1_B)) \otimes ((1_A \otimes 1_{C(Y)}) \otimes 1_B) \otimes (a \otimes b) &\mapsto a \otimes b,\\
        (1_{C(X)} \otimes a) \otimes b &\mapsto (1_{C(X)} \otimes (1_A \otimes 1_B)) \otimes ((1_{C(X)} \otimes a) \otimes b),\\
        a \otimes b &\mapsto ((1_A \otimes 1_B) \otimes 1_{C(Z)}) \otimes (a \otimes b).
    \end{align*}
\end{itemize}
The associativity condition for $\cccirc$ compares two 3-morphisms with source
{\footnotesize$$
(\id_A \cccirc \id_B \cccirc a_{CDE}) \ccirc (\id_A \cccirc a_{B,CD,E}) \ccirc (\id_A \cccirc (a_{BCD} \cccirc \id_E)) \ccirc a_{A,(BC)D,E} \ccirc (a_{A,BC,D} \cccirc \id_E) \ccirc ((a_{ABC} \cccirc \id_D) \cccirc \id_E),
$$}
and target
$$
a_{A,B,C(DE)} \ccirc a_{AB,C,DE} \ccirc a_{(AB)C,D,E}.
$$
A formal calculation reveals that both 3-morphisms are determined by mapping 
$$
1_{A(B((CD)E))} \otimes 1_{A((B(CD))E)} \otimes 1_{A(((BC)D)E} \otimes 1_{(A((BC)D))E} \otimes 1_{(((AB)C)D)E} \otimes ((((a \otimes b) \otimes c) \otimes d) \otimes e)
$$
to
$$
1_{(AB)(C(DE))} \otimes 1_{((AB)C)(DE)} \otimes ((((a \otimes b) \otimes c) \otimes d) \otimes e)
$$
where we denote $1_A \otimes (1_B \otimes (( 1_C \otimes 1_D) \otimes 1_E))$ by $1_{A(B((CD)E))}$ and so on.

The two axioms relating the unity coheretors with the associators and pentagonator compare two 3-morphisms with sources
\begin{align*}
(\id_A \cccirc (\id_B \cccirc \ell^{\sqdot}_C)) \ccirc (\id_A \cccirc a_{B,I_Z,C}) \ccirc (\id_A \cccirc (r_B \cccirc B)) \ccirc a_{ABC},\\
a_{ABC} \ccirc ((\id_A \cccirc \ell^{\sqdot}_B) \cccirc \id_C) \ccirc (a_{A,I_Y,B} \cccirc \id_C) \ccirc ((r_A \cccirc \id_B) \cccirc \id_C), 
\end{align*}
and targets
\begin{align*}
a_{ABC} \ccirc \id_{(AB)C},\\
\id_{A(BC)} \ccirc a_{ABC}
\end{align*}
respectively. Another formal calculation reveals that both 3-morphisms are determined by mapping
\begin{align*}
1_{A(B(I_ZC)} \otimes 1_{A((BI_Z)C)} \otimes 1_{A(BC)} \otimes ((a \otimes b) \otimes c),\\
1_{(AB)C} \otimes 1_{(A(I_Y B))C} \otimes 1_{((AI_Y)B)C} \otimes ((a \otimes b) \otimes c),
\end{align*}
into
\begin{align*}
1_{(AB)C} \otimes ((a \otimes b) \otimes c),\\
1_{A(BC)} \otimes ((a \otimes b) \otimes c),
\end{align*}
respectively. 
\end{example}

\begin{example}
One defines the analogous W*-3-category $\Meas$ whose:
\begin{itemize}
    \item objects are $\sigma$-finite measure spaces $(X,\mu$);
    \item 1-morphisms $A\colon (X,\mu) \to (Y,\nu)$ are W*-algebras $A$ equipped with normal $*$-homomorphisms $L^\infty(X,\mu) \to Z(A)$ and $L^\infty(Y,\nu) \to Z(A)$;
    \item 2-morphisms $H\colon A \tto B$ are $A$-$B$ W*-correspondences compatible with the $L^\infty(X,\mu)$ amd $L^\infty(Y,\nu)$ actions,
    \item 3-morphisms are adjointable normal intertwiners.
\end{itemize}    
\end{example}


\section{Yoneda for operator algebraic tricategories} \label{sec:yoneda}

We will now construct the Yoneda embedding for a \emph{cubical} C*-3-category $\cB$
$$\yo\colon \cB \to \CstarThreeCat(\cB^{\oneop} \to \CstarGray),$$
and a corresponding normal version 
$$\yo\colon \cB \to \WstarThreeCat(\cB^{\oneop} \to \WstarGray),$$
when $\cB$ is a W*-3-category. We begin by constructing the targets of these $\dag$-3-functors. 

\begin{lemma}\label{lem:hombetween3cats}
For C*-3-categories $\cA$ and $\cB$, and $\dag$-3-functors $F,G\colon \cA \to \cB$, there exists an organic C*-2-category structure on $\CstarThreeCat(\morph{\cA}{F}{\cB} \tto \morph{\cA}{G}{\cB})$, the 2-category of 
\begin{enumerate}
\item[(0)] $\dag$-3-natural transformations from $F$ to $G$,
\item[(1)] $\dag$-3-modifications, and
\item[(2)] uniformly bounded perturbations.
\end{enumerate}
which is strict whenever $\cB$ is locally strict.

Moreover, when $\cA$ and $\cB$ are W*-3-categories and $F,G$ are normal $\dag$-3-functors, we have that $\WstarThreeCat(F \tto G)$ forms a W*-3-category.
\end{lemma}

\begin{proof}
By \cite[\S 9.1, Theorem 9.1]{gurski_2013} we know that all 3-natural transformations from $F$ to $G$, together with all 3-modifications and perturbations between them form a 2-category $\ThreeCat(F \tto G)$ where the components for the associators and unitors are given by the associators and unitors of the components in $\cB$ respectively. 
First note that composites of unitary constraints for 3-natural transformations and 3-modifications are also unitary, so that $\dag$-3-natural transformations, $\dag$-3-modifications, and all perturbations form a 2-subcategory of $\ThreeCat(F \tto G)$. 
Since linear combinations and compositions occur componentwise, it is clear that $\CstarGray(F \tto G)$ (resp. $\WstarGray(F \tto G)$) forms a (linear) 2-subcategory of $\ThreeCat(F \tto G)$.
By defining $\dag$ to act componentwise on uniformly bounded perturbations, this yields a C*-2-category structure on $\CstarGray(F \tto G)$ (resp. $\WstarGray(F \tto G)$). 

In the W* case, one then uses the Kaplansky density theorem together with \cite[Lemma 2.13]{CP22} to show that $\WstarGray(F \tto G)$ is locally a C*-category and that $\ccirc$ is separately normal. 
\end{proof}

\begin{lemma}\label{lem:compositionin3Cats}
Let $\cA$ be a C*-3-category, $\cB$ a C*-Gray-category, and $F,G,H\colon \cA \to \cB$ $\dag$-3-functors. Then 1-composition forms cubical $\dag$-2-functor
$$\cccirc\colon \CstarThreeCat(F \tto G) \maxblacktimes \CstarThreeCat(G \tto H) \to \CstarThreeCat(F \tto H).$$
Moreover, when $\cA$ is a W*-3-category, $\cB$ a W*-Gray-category, and $F,G,H$ are normal $\dag$-3-functors, we have that $\cccirc$ is separately normal.
\end{lemma}

\begin{proof}
One easily verifies that the underlying composite of $\dag$-3-natural transformations, $\dag$-3-modifications, and uniformly bounded perturbations is again of the respective type. By \cite[Thm. 9.3]{gurski_2013}, we then know that $\cccirc$ assembles into a cubical 2-functor of the underlying 2-categories. It is also easy to see that $\cccirc$ is $\dag$-bilinear since linear combinations of perturbations are obtained componentwise, and that the interchanger $\Sigma$ for $\cccirc$ is unitary as it arises from the C*-Gray-category $\cB$ on each component of the relevant $\dag$-3-modifications. Therefore $\cccirc$ forms a cubical $\dag$-2-functor.

In the W* case, since $\cccirc$ is separately normal in $\cB$, the Kaplansky density theorem together with \cite[Lemma 2.13]{CP22} yields that $\cccirc$ is separately normal.
\end{proof}

\begin{corollary}
For $\cA$ a C*-3-category and $\cB$ a C*-Gray-category, $\CstarThreeCat(\cA \to \cB)$ forms a C*-Gray-category of
\begin{itemize}
\item[(0)] $\dag$-3-functors from $\cA$ to $\cB$,
\item[(hom)] Hom-C*-2-categories $\CstarThreeCat(\morph{\cA}{F}{\cB} \tto \morph{\cA}{G}{\cB})$ as in Lemma \ref{lem:hombetween3cats},
\item[($\cccirc$)] Cubical 1-composition $\cccirc$ as in Lemma \ref{lem:compositionin3Cats}
\end{itemize}
Moreover, when $\cA$ is a W*-3-category and $\cB$ is a W*-Gray-category, we have that $\WstarThreeCat(\cA \to \cB)$ forms a W*-Gray-category of normal $\dag$-3-functors from $\cA$ to $\cB$.
\end{corollary}

We now begin constructing the actual embedding.

\begin{lemma}
For a cubical C*-3-category $\cB$ and an object $B \in \cB$, there is an organic contravariant hom-$\dag$-3-functor
$$\yo_B\colon \cB^{\oneop} \to \CstarTwoCat_{\Strict}.$$
When $\cB$ is a cubical W*-3-category, this construction yields a normal $\dag$-3-functor
$$\yo_B\colon \cB^{\oneop} \to \WstarTwoCat_{\Strict}.$$
\end{lemma}
\begin{proof}
Let $\cB$ be a C*-3-category (resp. W*-3-category).
We will present the definition for the underlying Yoneda embedding for 3-categories seen in \cite{gurski_2013}, noting that all the relevant data is compatible with dagger structures.
\begin{enumerate}
\item[(0)] On an object $B' \in \cB$, we set
$$\yo_B(B') \coloneqq \cB(B' \to B),$$
which is a strict C*-2-category (resp. W*-2-category) since $\cB$ is cubical.
\item[(1)] On a 1-morphism $Y \in \cB(B'' \to B')$, we define the strict (normal) $\dag$-2-functor
$$\yo_B(Y)\colon \yo_B(B') \to \yo_B(B'')$$
as follows:
\begin{enumerate}
\item[(1.0)] For an object $X \in \yo_B(B') = \cB(B' \to B)$, which is a 1-morphism in $\cB$, we set
$$\yo_B(Y)(X) \coloneqq \morph{B''}{Y}{} \cccirc_{B'} \morph{}{X}{B}.$$
\item[(1.1)] For a 1-morphism $b$ in $\yo_B(B')$, which is a 2-morphism in $\cB$, we set
$$\yo_B(Y)(b) \coloneqq \id_Y \cccirc_{B'} b.$$
\item[(1.2)] Similarly, for a 2-morphism $\beta$ in $\yo_B(B')$, which is a 3-morphism in $\cB$, we set
$$\yo_B(Y)(\beta) \coloneqq \id_{\id_{Y}} \cccirc_{B'} \beta.$$
\end{enumerate}
Notice $\yo_B(Y)$ is linear, $\dag$-preserving, and preserves 3-composition $\circ$ and identities at the level of 2-morphisms in $\yo_B(B')$ (which are 3-morphisms in $\cB$). Moreover, $\yo_B(Y)$ is strict since $\cccirc$ is cubical. 

In the W* case, since $\cccirc$ is separately normal in $\cB$, the Kaplansky density theorem together with \cite[Lemma 2.13]{CP22} yields that $\yo_B(Y)$ is normal.

\item[(2)] For a 2-morphism $b \in \cB(\morph{B''}{Y}{B'} \tto \morph{B''}{Y'\!}{B'})$, we define the $\dag$-2-natural transformation
$$\yo_B(b)\colon \yo_B(Y) \tto \yo_B(Y').$$
as follows:
\begin{enumerate}
\item[(2.0)] For an object $X \in \yo_B(B')$, we define the component
$$\yo_B(b)_X\colon \underbrace{\yo_B(Y)(X)}_{Y \cccirc X} \to \underbrace{\yo_B(Y')(X)}_{Y' \cccirc X}$$
to be $\yo_B(b)_X \coloneqq b \cccirc \id_X$.
\item[(2.1)] For a 1-morphism $a$ in $\yo_B(B')$, we define the component unitary $\yo_B(b)_a$ to be $\Sigma^\dag_{a,b}$, which comes from the unitary interchanger for the cubical 1-composition $\cccirc$ in $\cB$.
\end{enumerate}
\item[(3)] On a 3-morphism $\beta \in \cB(b \ttto b')$, we define the uniformly bounded modification
$$\yo_B(\beta)\colon \yo_B(b) \ttto \yo_B(b')$$
as follows:
\begin{enumerate}
\item[(3.0)] For an object $X \in \yo_B(B')$, we define the component
$$\yo_B(\beta)_X \colon \underbrace{\yo_B(b)_X}_{b \cccirc \id_X} \ttto \underbrace{\yo_B(b')_X}_{b' \cccirc \id_X}$$
to be $\yo_B(\beta)_X \coloneqq \beta \cccirc \id_{\id_X}$. 
\end{enumerate}
Notice $\|\yo_B(\beta)\| = \sup_X \|\beta \cccirc \id_{\id_X}\| \leq \|\beta\|$, so $\yo_B(\beta)$ is indeed uniformly bounded.
\end{enumerate}
From this definition it is clear that $\yo_B$ is linear, $\dag$-preserving, and preserves 3-composition $\circ$ and identities at the level of 3-morphisms. 

In the W* case, since $\cccirc$ is separately normal in $\cB$, the Kaplansky density theorem together with \cite[Lemma 2.13]{CP22} yields that $\yo_B$ is normal.

Using cubicality, one can further check that $\yo_B$ is locally a strict (normal) $\dag$-2-functor between strict C*-2-categories (resp.\! W*-2-categories). We now recall the construction for the constraint unitary adjoint equivalences $\yo^2_B$ and $\yo^0_B$ for the $\dag$-3-functor $\yo_B$:
\begin{itemize}
\item[($\yo_B^2$)] For composable 1-morphisms $\morph{B'''}{Y'\!}{B''},\morph{B''}{Y}{B'}$ in $\cB$, we provide a tensorator $\dag$-2-natural transformation
$$(\yo_B^2)_{Y',Y}\colon \yo_B(Y') \cccirc \yo_B(Y) \tto \yo_B(Y' \cccirc_{B''} Y').$$
as follows:
\begin{itemize}
\item[($\yo_B^2.0$)] On an object $X \in \yo_B(B')$, we define the component
$$((\yo_B^2)_{Y',Y})_X \colon \underbrace{\big(\yo_B(Y') \cccirc \yo_B(Y)\big)(X)}_{Y' \cccirc_{B''} (Y \cccirc_{B'} X)} \tto \underbrace{\big(\yo_B(Y' \cccirc_{B''} Y')\big)(X)}_{(Y' \cccirc_{B''} Y) \cccirc_{B'} X},$$
to be $((\yo_B^2)_{Y'\!,Y})_X \coloneqq a_{Y'\!,Y,X}$. 
\item[($\yo_B^2.1$)] For a 1-morphism $b$ in $\yo_B(B')$, we define the unitary
$$((\yo_B^2)_{Y'\!,Y})_b \coloneqq a_{\id_{Y'},\id_{Y},b}.$$
\end{itemize}
One continues by defining the unitary adjoint equivalence for $\yo_B^2$ to be the unitary adjoint equivalence for $a$ with the first two variables held constant.
\item[($\yo_B^0$)] For an object $B' \in \cB$, we provide a unitor $\dag$-2-natural transformation
$$(\yo^0_B)_{B'} \colon \id_{\yo_B(B')} \tto \yo_B(\id_{B'}),$$
as follows:
\begin{itemize}
\item[($\yo_B^0.0$)] On an object $\morph{B'}{X}{B} \in \yo_B(B')$, we define the component
$$((\yo_B^0)_{B'})_X \colon X \tto \id_{B'} \cccirc X,$$
to be $((\yo_B^0)_{B'})_X \coloneqq r^{\sqdot}_X$. One continues defining the unitary adjoint equivalence for $\yo^0_B$ similarly to $\yo^2_B$.
\end{itemize}
\end{itemize}
We now provide the associativity and unitality constraint unitaries $\yo_B^{\alpha}$, $\yo^\lambda_B$, and $\yo^\rho_B$ for the $\dag$-3-functor $\yo_B$:
\begin{itemize}
\item[($\yo_B^{\alpha}$)] We define the associator unitary $\yo_B^{\alpha}$
to be the following mate of $\pi$, the pentagonator for $\cB$.

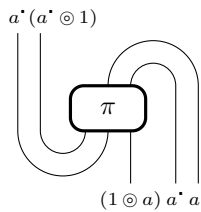
$$
\tikzmath{
\draw (0,.3) arc (180:0:.6cm) -- (1.2,-1);
\draw (0.3,.3) arc (180:0:.3cm) -- (.9,-1);
\draw (-.3,-.3) arc (0:-180:.3cm) -- (-.9,1);
\draw (0,-.3) arc (0:-180:.6cm) -- (-1.2,1);
\draw (.3,-.3) -- (.3,-1);
\roundNbox{fill=white}{(0,0)}{.3}{.2}{.2}{$\pi$};
\node at (.3,-1.22) {\scriptsize{$(1 \cccirc a)$}};
\node at (.9,-1.2) {\scriptsize{$a^{\sqdot}$}};
\node at (1.2,-1.2) {\scriptsize{$a^{\sqdot}$}};
\node at (-.6,1.2) {\scriptsize{$(a^{\sqdot} \cccirc 1)$}};
\node at (-1.2,1.22) {\scriptsize{$a^{\sqdot}$}};
}
$$

Indeed, $\yo_B^{\alpha}$ is unitary as $\pi$ is unitary, and the unit and counit for the unitary adjoint equivalence $\mathbf{\alpha}$ are unitaries.
\item[($\yo_B^{\lambda}$)] The left unitor unitary $\yo^\lambda_B$ is given by a mate of $\rho$. 
\item[($\yo_B^{\rho}$)] The right unitor unitary $\yo^\lambda_B$ is given by a mate of $\mu$. 
\end{itemize}
We refer the interested reader to \cite{gurski_2013} for the proof that $\yo_B$ satisfies both constraint axioms for a 3-functor.
\end{proof}

\begin{lemma}
For a cubical C*-3-category $\cB$ and a 1-morphism $X \in \cB(B \to B')$, there is an organic $\dag$-3-natural transformation
$$\yo_X \colon \yo_B \tto \yo_{B'}.$$
\end{lemma}

\begin{proof}
We will present the definition found in \cite[Lemma 9.8]{gurski_2013}, noting that all the relevant data is compatible with dagger structures.
\begin{itemize}
\item[(0)]
For an object $A \in \cB$, we define the component
$$(\yo_X)_A\colon \underbrace{\yo_B(A)}_{\cB(A \to B)} \to \underbrace{\yo_{B'}(A)}_{\cB(A \to B')},$$
to be the (normal) strict $\dag$-2-functor given by:
\begin{itemize}
\item[(0.0)] For an object $\morph{A}{Y}{B} \in \yo_B(A)$,
$$(\yo_X)_A(Y) \coloneqq \morph{A}{Y}{} \cccirc_B \morph{}{X}{B'}.$$
\item[(0.1)] For a 1-morphism $b$ in $\yo_B(A)$,
$$(\yo_X)_A(b) \coloneqq b \cccirc \id_X.$$
\item[(0.2)] For a 2-morphism $\beta$ in $\yo_B(A)$,
$$(\yo_X)_A(b) \coloneqq \beta \cccirc \id_{\id_X}.$$
\end{itemize}
\item[(1)] For a 1-morphism $Y \in \cB(A \to A')$, we provide a unitary adjoint equivalence for $(\yo_X)_Y$, where
$(\yo_X)_Y$
is a $\dag$-2-natural transformation given by:
\begin{itemize}
\item [(1.0)] On an object $\morph{A}{Z}{B} \in \yo_B(A)$, we define the component
$$(((\yo_X)_Y)_Z)\colon \underbrace{(\yo_X)_{A'} \cccirc \yo_B(Y) (Z)}_{(\morph{A'}{Y}{} \cccirc_A Z) \cccirc_B \morph{}{X}{B'}}\tto \underbrace{\yo_{B'}(Y) \cccirc (\yo_X)_A (Z)}_{\morph{A'}{Y}{} \cccirc_A ( Z \cccirc_B \morph{}{X}{B'})}$$
to be $(((\yo_X)_Y)_Z) \coloneqq a_{YZX}^{\sqdot}$.
\item [(1.1)] On a 1-morphism $c$ in $\yo_B(A)$, we set $(((\yo_X)_Y)_c \coloneqq a^{\sqdot}_{\id_Y,c,\id_{X}}$.
\end{itemize}
One continues by defining the unitary adjoint equivalence for $(\yo_X)_Y$ using the unitary adjoint equivalence for $a$ with the first and last variables held constant. We now provide the tensorator and unitality constraint unitaries $\yo_X^2$ and $\yo_X^0$.
\begin{itemize}
\item[($\yo_X^2$)] We define the unitary tensorator modification $\yo_X^2$
to be the following mate of $\pi^\dag$, coming from the pentagonator of $\cB$.

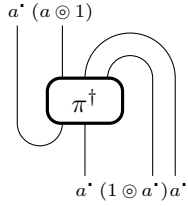
$$
\tikzmath{
\draw (-0,-.3) -- (-0,-1);
\draw (-0.3,-.3) arc (0:-180:.3cm) -- (-.9,1);
\draw (.3,.3) arc (180:0:.3cm) -- (.9,-1);
\draw (-0,.3) arc (180:0:.6cm) -- (1.2,-1);
\draw (-.3,.3) -- (-.3,1);
\roundNbox{fill=white}{(0,0)}{.3}{.2}{.2}{$\pi^\dag$};
\node at (-.3,1.18) {\scriptsize{$(a \cccirc 1)$}};
\node at (-.9,1.2) {\scriptsize{$a^{\sqdot}$}};
\node at (.65,-1.2) {\scriptsize{$(1 \cccirc a^{\sqdot})$}};
\node at (1.25,-1.18) {\scriptsize{$a^{\sqdot}$}};
\node at (0,-1.18) {\scriptsize{$a^{\sqdot}$}};
}
$$

Indeed, $\yo_X^2$ is unitary since $\pi$ is unitary and the unit and counit for the adjoint equivalence $a$ are unitaries.
\item[($\yo_X^0$)] We define the unitary unitor modification $\yo_X^0$
to be the following mate of $\rho$.

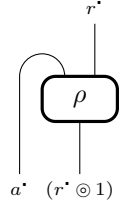
$$
\tikzmath{
\draw (-.2,.3) arc (0:180:.3cm) -- (-.8,-1);
\draw (0.2,.3) -- (.2,1);
\draw (0,-.3) -- (0,-1);
\roundNbox{fill=white}{(0,0)}{.3}{.2}{.2}{$\rho$};
\node at (0,-1.22) {\scriptsize{$(r^{\sqdot} \cccirc 1)$}};
\node at (-.8,-1.2) {\scriptsize{$a^{\sqdot}$}};
\node at (.2,1.22) {\scriptsize{$r^{\sqdot}$}};
}
$$

Indeed, $\yo_X^2$ is unitary since $\rho$ is unitary and the unit and counit for the adjoint equivalence $r$ are unitaries
\end{itemize}
\end{itemize}
We refer the interested reader to \cite{gurski_2013} for the proof that $\yo_X$ satisfies the three constraint axioms for a $3$-natural transformation.
\end{proof}

\begin{lemma}
For a cubical C*-3-category $\cB$ and a 2-morphism $b \in \cB(X \tto X')$, there is an organic $\dag$-3-modification
$$\yo_b \colon \yo_X \ttto \yo_{X'}.$$
\end{lemma}

\begin{proof}
We will present the definition found in \cite[Lemma 9.9]{gurski_2013}, noting that all the relevant data is compatible with dagger structures.
\begin{itemize}
\item[(0)] For an object $A \in \cB$, we define the component
$$(\yo_b)_A\colon (\yo_X)_A \tto (\yo_{X'})_A$$
to be the $\dag$-2-natural transformation given by:
\begin{itemize}
\item[(0.0)] On an object $\morph{A}{Y}{B}$, we define the component
$$(((\yo_b)_A)_Y\colon \underbrace{(\yo_X)_A(Y)}_{Y \cccirc X} \tto \underbrace{(\yo_{X'})_A(Y)}_{Y \cccirc X'}$$
to be $(((\yo_b)_A)_Y \coloneqq \id_Y \cccirc b$.
\item[(0.1)] On a 1-morphism $a$, we define the naturality unitary constraint $((\yo_b)_A)_a$ 
to be $((\yo_b)_A)_a \coloneqq \Sigma^\dag_{b,a}$, which comes from the unitary interchanger for the cubical 1-composition $\cccirc$ in $\cB$.
\end{itemize}
\item[(1)] For a 1-morphism $Z \in \cB(A \to A')$, we define the naturality constraint unitary $(\yo_b)_Z$
to be the unitary modification given by:
\begin{itemize}
    \item[(1.0)] On an object $Y \in \cB$, we define the component $((\yo_b)_Z)_Y \coloneqq a^{\sqdot}_{\id_Z,\id_Y,b}$.
\end{itemize}
\end{itemize}
We refer the interested reader to \cite{gurski_2013} for the proof that $\yo_b$ satisfies the two constraint axioms for a $3$-modification.
\end{proof}

\begin{lemma}
For a cubical C*-3-category $\cB$ and a 3-morphism $\beta \in \cB(b \tto b')$, there is an organic uniformly bounded perturbation
$$\yo_\beta \colon \yo_b \tttto \yo_{b'}.$$
This assignment is linear and $\dag$-preserving in $\beta$. Furthermore, when $\cB$ is a cubical W*-3-category, we have that this assignment is normal.
\end{lemma}

\begin{proof}
We will present the definition found in \cite[Lemma 9.10]{gurski_2013}, noting that all the relevant data is compatible with dagger structures.
\begin{itemize}
\item[(0)] For an object $A \in \cB$, we define the component
$$(\yo_\beta)_A \colon (\yo_b)_A \ttto (\yo_{b'})_A$$
to be the uniformly bounded modification given by:
\begin{itemize}
\item[(0.0)] For an object $\morph{A}{Y}{B}$, we define the component
$$((\yo_\beta)_A)_Y \colon \underbrace{((\yo_b)_A)_Y}_{\id_Y \cccirc b} \ttto \underbrace{((\yo_{b'})_A)_Y}_{\id_Y \cccirc b'}$$
to be $((\yo_\beta)_A)_Y \coloneqq \id_{\id_Y} \cccirc \beta$.
\end{itemize}
\end{itemize}
From this definition it is clear that the assignment $\beta \mapsto \yo_\beta$ is linear, $\dag$-preserving, and preserves 3-composition $\circ$ and identities at the level of 3-morphisms. 
Furthermore, this assignment is normal when $\cB$ is W* by our usual argument using \cite[Lemma 2.13]{CP22}. 
We refer the interested reader to \cite{gurski_2013} for the proof that $\yo_\beta$ satisfies the constraint axiom for perturbations.
\end{proof}

\begin{theorem}
We may upgrade the previous data to a monic $\dag$-3-functor
$$\yo\colon \cB \to \CstarThreeCat(\cB^{\oneop} \to \CstarTwoCat_{\Strict}),$$
which is locally a 2-equivalence.
When $\cB$ is a W*-3-category, we obtain a monic normal $\dag$-3-functor
$$\yo\colon \cB \to \WstarThreeCat(\cB^{\oneop} \to \WstarTwoCat_{\Strict}).$$
\end{theorem}

\begin{proof}
We provide the constraint unitary adjoint equivalences $\yo^2$ and $\yo^0$ for the $\dag$-3-functor $\yo_B$ found in  \cite[Theorem 9.12]{gurski_2013}.
\begin{itemize}
\item[($\yo^2$)] For composable 1-morphisms $\morph{B}{X}{B'}$ and $\morph{B'}{Y}{B''}$ in $cB$, we define the component
$$\yo^2_{X,Y} \colon \yo_X \cccirc \yo_Y \ttto \yo_{X \cccirc Y}$$
to be the $\dag$-3-modification given by:
\begin{itemize}
\item[($\yo^2.0$)] On an object $A \in \cA$, we define the component
$$(\yo^2_{X,Y})_A \colon (\yo_X \cccirc \yo_Y)_A \tto (\yo_{X \cccirc Y})_A$$
to be the $\dag$-2-natural transformation given by:
\begin{itemize}
    \item[($\yo^2.0.0$)] For an object $\morph{A}{Z}{B} \in \yo_B(A)$, 
    $$((\yo^2_{X,Y})_A)_Z \colon \underbrace{(\yo_X \cccirc \yo_Y)_A(Z)}_{(Z \cccirc X ) \cccirc Y} \tto \underbrace{(\yo_{X \cccirc Y})_A(Z)}_{Z \cccirc (X \cccirc Y)},$$
    to be $((\yo^2_{X,Y})_A)_Z \coloneqq a^{\sqdot}_{ZXY}$.
    \item[($\yo^2.0.1$)] For a 1-morphism $z$ in $\yo_B(A)$, we then define
    $$((\yo^2_{X,Y})_A)_Z \coloneqq a^{\sqdot}_{z,\id_X,\id_Y}.$$
\end{itemize}
\item[($\yo^2.1$)] On a 1-morphism $\morph{A}{V}{A'}$, we define the naturator $(\yo^2_{X,Y})_V$
to be the unitary modification given by:
\begin{itemize}
\item[($\yo^2.1.0$)] On an object $\morph{A'}{W}{B}$, we define the component $((\yo^2_{X,Y})_V)_W$
to be the following mate of $\pi$, the unitary pentagonator of $\cB$.

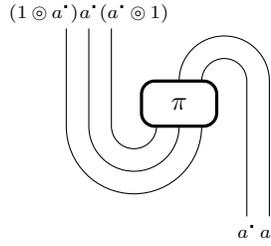
$$
\tikzmath{
\draw (0,.3) arc (180:0:.6cm) -- (1.2,-1.5);
\draw (0.3,.3) arc (180:0:.3cm) -- (.9,-1.5);
\draw (-.3,-.3) arc (0:-180:.3cm) -- (-.9,1);
\draw (0,-.3) arc (0:-180:.6cm) -- (-1.2,1);
\draw (.3,-.3) arc (0:-180:.9cm) -- (-1.5,1);
\roundNbox{fill=white}{(0,0)}{.3}{.2}{.2}{$\pi$};
\node at (-1.8,1.2) {\scriptsize{$(1 \cccirc a^{\sqdot})$}};
\node at (.9,-1.7) {\scriptsize{$a^{\sqdot}$}};
\node at (1.2,-1.7) {\scriptsize{$a^{\sqdot}$}};
\node at (-.6,1.2) {\scriptsize{$(a^{\sqdot} \cccirc 1)$}};
\node at (-1.2,1.22) {\scriptsize{$a^{\sqdot}$}};
}
$$

\end{itemize}
One continues by defining the unitary adjoint equivalence for $\yo^2_{X,Y}$ using the unitary adjoint equivalence for $a$ with the last two entries fixed, and a mate of $\pi^\dag$.
\end{itemize}
Then one defines the naturality constraints for $\yo^2$ using the unitary adjoint equivalence for $a$, with the first entry fixed.
\item[($\yo^0$)] We define the unitor adjoint equivalence $\yo^0$ similarly to $\yo^2$, using $\ell^{\sqdot}$ instead of $a^{\sqdot}$, and the following mate of $\lambda^{\dag}$.
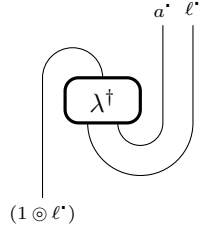
$$
\tikzmath{
\draw (-0,.3) arc (0:180:.4cm) -- (-.8,-1.3);
\draw (0.2,-.3) arc (-180:0:.3cm) -- (.8,1);
\draw (-0.2,-.3) arc (-180:0:.7cm) -- (1.2,1);
\roundNbox{fill=white}{(0,0)}{.3}{.2}{.2}{$\lambda^\dag$};
\node at (-.8,-1.52) {\scriptsize{$(1 \cccirc \ell^{\sqdot})$}};
\node at (.8,1.2) {\scriptsize{$a^{\sqdot}$}};
\node at (1.2,1.22) {\scriptsize{$\ell^{\sqdot}$}};
}
$$
We also define unitary perturbation $(\yo^0)_{\bullet}$ to be the identity.
\end{itemize}
We now provide the constraint unitaries $\yo^\alpha$, $\yo^\lambda$, and $\yo^\rho$ for the $\dag$-3-functor $\yo_B$.
\begin{itemize}
\item[($\yo^\alpha$)]
For composable 1-morphisms $\morph{B}{X}{B'}$, $\morph{B'}{Y}{B''}$, $\morph{B''}{Z}{B'''}$ in $\cB$, we define the components $((\yo^\alpha_{XYZ})_{A})_{\morph{A}{W}{B}}$ of the unitary modification $(\yo^\alpha_{XYZ})_{A}$ for the unitary perturbation $\yo^\alpha_{XYZ}$ using the following mate of the pentagonator $\pi_{WXYZ}$.
$$
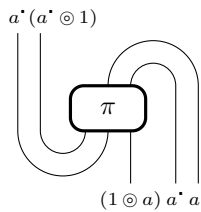
\tikzmath{
\draw (0,.3) arc (180:0:.6cm) -- (1.2,-1);
\draw (0.3,.3) arc (180:0:.3cm) -- (.9,-1);
\draw (-.3,-.3) arc (0:-180:.3cm) -- (-.9,1);
\draw (0,-.3) arc (0:-180:.6cm) -- (-1.2,1);
\draw (.3,-.3) -- (.3,-1);
\roundNbox{fill=white}{(0,0)}{.3}{.2}{.2}{$\pi$};
\node at (.3,-1.22) {\scriptsize{$(1 \cccirc a)$}};
\node at (.9,-1.2) {\scriptsize{$a^{\sqdot}$}};
\node at (1.2,-1.2) {\scriptsize{$a^{\sqdot}$}};
\node at (-.6,1.2) {\scriptsize{$(a^{\sqdot} \cccirc 1)$}};
\node at (-1.2,1.22) {\scriptsize{$a^{\sqdot}$}};
}
$$

\item[($\yo^\lambda$)]
Similarly, for a 1-morphism $X$ in $\cB$ we use the following mate of the left unitor $\lambda$ to define the unitary perturbation $(\yo^\lambda)_X$.

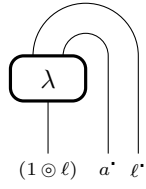
$$
\tikzmath{
\draw (0,-.3) -- (0,-1);
\draw (0.2,.3) arc (180:0:.3cm) -- (.8,-1);
\draw (-0.2,.3) arc (180:0:.7cm) -- (1.2,-1);
\roundNbox{fill=white}{(0,0)}{.3}{.2}{.2}{$\lambda$};
\node at (0,-1.22) {\scriptsize{$(1 \cccirc \ell)$}};
\node at (.8,-1.2) {\scriptsize{$a^{\sqdot}$}};
\node at (1.2,-1.22) {\scriptsize{$\ell^{\sqdot}$}};
}
$$

\item[($\yo^\rho$)]
Finally, for a 1-morphism $X$ in $\cB$ we use the following mate of the middle unitor $\mu$ to define the unitary perturbation $(\yo^\rho)_X$.

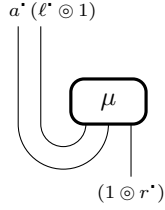
$$
\tikzmath{
\draw (-.3,-.3) arc (0:-180:.3cm) -- (-.9,1);
\draw (0,-.3) arc (0:-180:.6cm) -- (-1.2,1);
\draw (.3,-.3) -- (.3,-1);
\roundNbox{fill=white}{(0,0)}{.3}{.2}{.2}{$\mu$};
\node at (.3,-1.22) {\scriptsize{$(1 \cccirc r^{\sqdot})$}};
\node at (-.6,1.2) {\scriptsize{$(\ell^{\sqdot} \cccirc 1)$}};
\node at (-1.2,1.22) {\scriptsize{$a^{\sqdot}$}};
}
$$
\end{itemize}
We refer the interested reader to \cite{gurski_2013} for the proof that $\yo$ satisfies both constraint axioms for 3-functors and locally a biequivalence. Quite pedantically, it is also clear that $\yo$ is injective on every level. 
\end{proof}

\begin{theorem}[Gelfand-Naimark for operator 3-categories]
Every small C*-3-category $\cB$ is 3-equivalent to a sub-C*-Gray-category of $\ThreeHilb$.
\end{theorem}

\begin{proof}
When $\cB$ is a small $\CstarGray$-category, notice
\begin{align*}
\yo^{\amalg} \colon \cB &\to \CstarTwoCat_{\Strict,\Small}\\
B &\mapsto \coprod_{B' \in \cB} \cB(B' \to B)
\end{align*}
is a monic $\dag$-3-functor. Hence, we need only extend the Gelfand-Naimark-Segal construction for C*-2-categories into a monic $\dag$-3-functor
\begin{align*}
\GNS''_2 \colon \CstarTwoCat_{\Strict,\Small} &\to \ThreeHilb.\\
\cC &\mapsto \Upsilon_2(\cC). 
\end{align*}
This is done analogously to how one extends the Gelfand-Naimark-Segal construction for C*-1-categories into a monic $\dag$-2-functor
\begin{align*}
\GNS'' \colon \CstarCats_{\Small} &\to \TwoHilb.\\
\cD &\mapsto \Upsilon(\cD). 
\end{align*}
We will provide the details for how to define $\GNS''_2(F)\colon \Upsilon_2(\cC) \to \Upsilon_2(\cC')$ for a strict $\dag$-2-functor $F \colon \cC \to \cC'$ between strict, small C*-2-categories, and leave the remaining details to the reader. The strict $\dag$-2-functor $\GNS''_2(F)$ maps an object in $\Upsilon_2(\cC)$ given by
$$\GNS\left(\coprod_{C' \in \cC} \cC(C' \to C)\right)'' \cong \coprod_{C' \in \cC} \GNS(\cC(C' \to C))''$$
to the following object in $\Upsilon_2(\cC')$
$$
\GNS\left(\coprod_{C' \in \cC'} \cC'(C' \to FC)\right)''
\cong 
\coprod_{C' \in \cC'} \GNS(\cC'(C' \to FC))''.$$
On 1-morphisms and 2-morphisms in $\Upsilon_2(\cC)$, the action of $\GNS''_2(F)$ is then given by the universal completion of post-composition in each coproduct component. Notice this $\dag$-2-functor is indeed strict since $\cC'$ is a strict C*-2-category and $F$ is a strict $\dag$-2-functor.
\end{proof}

\pagebreak 
\bibliographystyle{amsalpha}
{\footnotesize{
\bibliography{refs}
}}
\end{document}